\documentclass{amsart}

\usepackage{etex}
\usepackage{bm}
\usepackage{amssymb}
\usepackage[stretch=10,shrink=10]{microtype}
\usepackage{tikz}
\usepackage{enumerate}
\usepackage{theoremref}
\usepackage{comment}
\usepackage[showframe=false, margin = 1.5 in]{geometry}
\usepackage{pgfplots}
\usepackage{hyperref}
\usepackage{pdfpages}
\usepackage{graphicx}

\newcommand{\good}{\textit{Good}}
\newcommand{\lobby}{\mathcal{E}^2}
\newcommand{\stool}{\textit{OK}}

\newcommand{\dining}{\textit{Mid}}
\newcommand{\well}{\textit{W}}

\newcommand{\eps}{\varepsilon}

\renewcommand{\emptyset}{\varnothing}
\newcommand{\E}{\mathbf{E}}
\renewcommand{\P}{\mathbf{P}}

\newcommand{\leb}{\mathbf{Leb}}

\newcommand{\pre}{\text{Prediction}}

\renewcommand{\L}{\mathcal L}

\renewcommand{\emptyset}{\varnothing}

\newcommand{\dir}{\textup{Dir}}
\newcommand{\arc}{\textup{Arc}}
\newcommand{\Z}{\mathbb Z}

\newcommand{\R}{\mathbb R}

\newcommand{\net}{\mathcal{F}}
\usepackage{color}
\definecolor{dark_green}{RGB}{1, 180, 1}

\usepackage{microtype}
\usepackage{mathrsfs}
\tikzset{vert/.style={circle,fill,inner sep=0,
    minimum size=0.15cm,draw}, nerve/.style={circle,inner sep=0,
    minimum size=0.15cm,draw},
    rightnerve/.style={circle,inner sep=0,
    minimum size=0.3cm,draw, fill}}

\newcommand{\ball}{\textup{Ball}}

\newcommand{\Ec}{\mathcal E}
\newcommand{\geo}{{\textup{Geo} }}

\newcommand{\shiftv}{\sigma_v }

\newcommand{\Geo}{\mathrm{Geo}}

\theoremstyle{plain}
\newtheorem{thm}{Theorem}[section]
\newtheorem{theorem}[thm]{Theorem}
\newtheorem{lemma}[thm]{Lemma}
\newtheorem{prop}[thm]{Proposition}

\newtheorem{define}[thm]{Definition}

\newtheorem{question}{Question}
\newtheorem{claim}[thm]{Claim}

\theoremstyle{remark}
\newtheorem{remark}[thm]{Remark}

\tikzset{every tree node/.style={align=center}}
\usepackage{url}

\def\tree{{\mathscr T}}
\def\path{{\mathscr P}}
\def\labels{{\mathscr L}}

\def\double{{\mathscr M}}
\def\support{{\mathscr S}}
\def\functionals{{\mathscr C}}
\def\rcg{{\mathscr G}}

\begin{document}

\title[Random Coalescing Geodesics]{Random coalescing geodesics in\\ first-passage percolation}

\author{Daniel Ahlberg}
\address{IMPA, Rio de Janeiro, and Department of Mathematics, Uppsala University. Current address: Department of Mathematics, Stockholm University}
\email{daniel.ahlberg@math.su.se}

\author{Christopher Hoffman}
\address{Department of Mathematics, University of Washington}
\email{hoffman@math.washington.edu}

\thanks{The authors are very grateful to Vladas Sidoravicius who encouraged them to pursue this work and made valuable contributions in some of their discussions.
The authors also thank IMPA and University of Washington for their hospitality during visits related to this work.
This research was supported by grants 637-2013-7302 and 2016-04442 from the Swedish Research Council (DA)
and by grant DMS-1308645 from the NSF (CH)}

\begin{abstract}
We continue the study of infinite geodesics in planar first-passage percolation, pioneered by Newman in the mid 1990s. Building on more recent work of Hoffman, and Damron and Hanson, we develop an ergodic theory for infinite geodesics via the study of what we shall call \emph{random coalescing geodesics}.
Random coalescing geodesics have a range of nice asymptotic properties, such as asymptotic directions and linear Busemann functions. We show that random coalescing geodesics are (in some sense) dense in the space of geodesics. This allows us to extrapolate properties from random coalescing geodesics to obtain statements on all infinite geodesics.
As an application of this theory we solve the `midpoint problem' of Benjamini, Kalai and Schramm and address a question of Furstenberg on the existence of bigeodesics.
\end{abstract}

\maketitle

\section{Introduction}

In first-passage percolation the edges (or sites) of the $\Z^2$ nearest neighbor lattice are equipped with non-negative random weights, thus giving rise to a random metric space. Since the first works on spatial growth by Eden~\cite{eden61} and Hammersley and Welsh~\cite{hamwel65}, first-passage percolation has attracted vast attention from mathematicians and physicists alike, see e.g.~\cite{kesten86,kruspo91,aufdamhan17}, aiming to understand the large-scale behavior of distances, balls and geodesics in this random metric space. Despite its success in inspiring powerful theories, first-passage percolation has proven to be one of the more difficult statistical physics models to analyze, and many of the most important conjectures remain unsettled.

The study of first-passage percolation has led to the development of powerful mathematical tools, like a rigorous theory for subadditive ergodic processes by Kingman~\cite{kingman68,kingman73}, as well as one of the most important avenues of study in mathematical physics, based on the predictions originating from the work of Kardar, Parisi and Zhang~\cite{karparzha86}.
The behavior of finite geodesics is an integral component in KPZ-theory. The theory predicts the existence of exponents $\chi$ and $\xi$, known as the \emph{fluctuation} and \emph{wandering} exponents, such that, with high probability, the distance between $(0,0)$ and $(n,0)$ deviates from its mean by $n^{\chi+o(1)}$ and the vertical displacement of the geodesic between $(0,0)$ and $(n,0)$ scales as $n^{\xi+o(1)}$. In two dimensions the two exponents $\chi$ and $\xi$ should equal $1/3$ and $2/3$ respectively, and thus be related through the equation $\chi=2\xi-1$. While these values are believed to differ in higher dimensions, the relation $\chi=2\xi-1$ is expected to prevail; see~\cite{kruspo91}. Even the existence of these exponents remains a mystery in first-passage percolation. However, if they do exist then the scaling relation $\chi=2\xi-1$ should also hold; see~\cite{chatterjee13,aufdam14}. There are closely related so-called `exactly solvable' models for which such a behavior has been rigorously 
established~\cite{baideijoh99,johansson00}. See for instance~\cite{corwin16}, and references therein, for an exposition of this development.

The theme of the present paper is geodesics, and our aim is to develop an ergodic theory describing infinite geodesics in first-passage percolation. The study of infinite geodesics was pioneered in the mid 1990s by Newman and collaborators~\cite{newman95,newpiz95,licnew96}. The program rolled out by Newman has evolved around two main themes. One of these has aimed to establish the existence of, and the scaling relation between, the exponents $\chi$ and $\xi$. The notion of Busemann functions was first explored in~\cite{newman95} for this purpose. In rough terms, Busemann functions measure locally the difference in distance ``to infinity'' in a given direction, and are globally expected to be described by linear functionals. The other theme concerns the existence of bi-infinite distance minimizing paths, also known as bigeodesics.

In Euclidean space the geodesic between two points is given by a line segment, and each line segment can be extended to a bi-infinite distance minimizing curve -- the straight line. Having uniform zero curvature, Euclidean geometry is critical in this aspect, and one could expect that bi-infinite distance minimizing curves should disappear in a random perturbation of this geometry.
First-passage percolation can be thought of such a perturbation, and bi-infinite geodesics have indeed been conjectured not to exist in that setting.
Kesten~\cite[p.\ 258]{kesten86} attributes the question of existence of bigeodesics in first-passage percolation to Hillel Furstenberg, and the question has since gained fame through its connection to the existence of non-trivial ground states of the two-dimensional Ising ferromagnet with random exchange constants; see~\cite{licnew96,wehr97}. Newman has given a convincing heuristic argument, based in part on the scaling behavior predicted by KPZ-theory, ruling out the existence of bigeodesics. This argument has been reproduced in~\cite[Section~4.5]{aufdamhan17}. In higher dimensions the existence of bigeodesics is further related to work by Kesten~\cite{kesten87b} on hypersurfaces with minimal random weight.

The structure of infinite geodesics has further been found to exhibit intriguing connections with other important probabilistic models. For instance, geodesics in first-passage percolation was closely linked to solutions of the Burgers equation in work of Bakhtin, Cator and Khanin~\cite{bakcatkha14}. Inspired by the study of geodesics in a Euclidean version of first-passage percolation by Howard and Newman~\cite{hownew01} and the construction of stationary measures in a last-passage percolation model by Cator and Pimentel~\cite{catpim12}, Bakhtin, Cator and Khanin incorporated these ideas to study the space of solutions to the Burgers equation.
They constructed space-time stationary solutions of the one-dimensional Burgers equation with random forcing in the absence of periodicity or compactness assumptions. More precisely, they showed that there is a unique global solution to the Burgers equation with any prescribed average velocity under a model where the forcing is given by a homogeneous Poissonian point field in space-time.
These solutions are based on Busemann functions.

We shall in this paper continue the study of infinite geodesics in first-passage percolation initiated by Newman, and continued by Hoffman~\cite{hoffman08} and Damron and Hanson~\cite{damhan14}.
Via the study of what we shall call \emph{random coalescing geodesics}, we build an ergodic theory for the study of infinite geodesics, incorporating elements like coalescence, Busemann functions and subsequential limiting procedures present also in previous work. Random coalescing geodesics have a range of nice asymptotic properties.
Although random coalescing geodesics do not (necessarily) account for all geodesics, we show that they are sufficiently dense in the space of infinite geodesics so that we may extrapolate certain properties of random coalescing geodesics to obtain global statements about geodesics.
We shall show that for almost every realization of the weight configuration the space of infinite geodesics has the same asymptotic behavior when it comes to cardinality, directions and Busemann functions.
While the existence of scaling exponents and bigeodesics remain open problems, we shall as a consequence of the theory we develop provide partial results in their direction.

We shall throughout the paper work under a stationary and ergodic assumption on the weight distribution, which enables us to obtain a very general theory. However, in this general setting different models of first-passage percolation are known to behave very differently. Consequently, our results will all be qualitative and not quantitative. In order to obtain quantitative estimates in a specific setting, say for independent continuous edge weights, one would have to incorporate the independence assumption in some fundamental way.

\section{Statement of results}

In this paper we consider a large number of models of first-passage percolation on $\Z^2$, including those with independent edge weights from a common continuous distribution with finite mean; see Section~\ref{sec: assumptions} below for a precise description. Let $\lobby$ denote the set of edges of the $\Z^2$ nearest-neighbor lattice. For each $\omega\in\Omega_1:=[0,\infty)^{\lobby}$ we define a metric on $\Z^2$ via\footnote{By \emph{path} we shall refer to a nearest-neighbor path in $\Z^2$. We will interchangeably think of a path as a sequence of vertices $(v_0,v_1,\ldots)$ or a sequence of edges $(e_0,e_1,\ldots)$, as suitable for each situation. For a finite path $\pi$ we shall below write $T(\pi)$ for the sum of its edge weights.}
\begin{equation}\label{Tdef}
T(x,y):=\inf\Big\{\sum_{e\in\pi}\omega_e:\pi\text{ a path connecting $x$ and $y$}\Big\}.
\end{equation}
A path attaining the infimum in~\eqref{Tdef} is referred to as a (finite) {\bf geodesic}.
In each of the first-passage models that we shall work with there will \emph{(i)}~exist a unique geodesic between any two points $x$ and $y$ -- we shall denote this path by $\geo(x,y)$; and \emph{(ii)}~exist a compact and convex set $\ball\subseteq\R^2$ with non-empty interior such that $\frac{1}{t}\{x:T(0,x)\le t\}$ approaches $\ball$ as $t$ increases. Again, we refer to Section~\ref{sec: assumptions} below for a precise statement.

The focus of this paper lies on infinite geodesics. More precisely, we shall develop an ergodic theory around what we shall call random coalescing geodesics, which we define shortly. Our reasons for this are two-fold. First, we aim to describe the set of infinite geodesics originating at the origin, relating the number of geodesics and their directions to the asymptotic shape $\ball$. Second, we address questions related to the existence of bigeodesics and scaling exponents. Random coalescing geodesics have, indeed, nice properties such as coalescence, asymptotic directions and asymptotically linear Busemann functions. By showing that random coalescing geodesics are in some sense `typical' we obtain our results.

The first of our results relates to both scaling exponents and bigeodesics. Based on the predictions of KPZ-theory it is widely believed that the probability that the geodesic between $(-n,0)$ and $(n,0)$ visits the origin should scale like $n^{-\xi+o(1)}$, where again $\xi=2/3$. Our result takes a modest first step in this direction, and answers a longstanding open question of Benjamini, Kalai and Schramm~\cite{benkalsch03}.

\begin{theorem} \label{hall of mirrors}
For any sequences $(u_k)_{k\ge1}$ and $(v_k)_{k\ge1}$ in $\Z^2$ so that $|u_k|,|v_k| \to \infty$ we have
$$
\P\big(0 \in \geo(u_k,v_k)\big) \to 0.
$$
\end{theorem}

As a corollary of the above result one obtains that the expected number of intersections between $\geo(0,v)$ and the straight line through $0$ and $v$ grows sublinearly in $|v|$.

A semi-infinite path $(v_k)_{k\ge1}$ will be referred to as a (semi-infinite) {\bf geodesic} if each finite segment is a geodesic, and a bi-infinite path $(v_k)_{k\in\Z}$ with the same property will be referred to as a bi-infinite geodesic, or a {\bf bigeodesic}. It has been conjectured that there are almost surely no bigeodesics. We are able to show that in each fixed direction (except for an at most countable set determined by $\ball$) this is true. This is closely related to the fact that multiple geodesics do not occur in directions of differentiability of the asymptotic shape.

For any geodesic $g$ (semi- or bi-infinite) we define the {\bf direction} $\dir(g)$ of $g$ as the set of limit points of the set $\{v/|v|:v\in g\}$. When $u\in\dir(g)$ we say that $u$ is {\bf a direction} of $g$. Note that $\dir(g)$ is an arc on the unit circle $S^1$ when $g$ is semi-infinite, and the union of two arcs when $g$ is a bi-infinite.
 
\begin{theorem} \label{versailles}
Let $u\in S^1$ be a direction of differentiability of $\partial\ball$. Then,
\begin{enumerate}[\quad (a)]
\item $\P\big(\exists \text{ two disjoint infinite geodesics with $u$ as a direction}\big)=0$;
\item $\P\big(\exists \text{ a bigeodesic with $u$ as a direction}\big)=0$.
\end{enumerate}
\end{theorem} 

Let $\tree_0=\tree_0(\omega)$ denote the set of infinite geodesics originating at the origin. Formally we may think of $\tree_0$ as the random element $\tree_0:\Omega_1\to\Omega_2$, where $\Omega_2:=\{0,1\}^{\lobby}$, that encodes whether or not an edge is traversed by an infinite geodesic starting at the origin. Our standing assumption on unique passage times assures that the graph induced by the edges in $\tree_0$ is a tree, almost surely.

That $\tree_0$ is non-empty follows easily from compactness. Under the additional assumption that the asymptotic shape $\ball$ is uniformly curved, Licea and Newman~\cite{newman95,licnew96} have shown that every infinite geodesic has a well-defined asymptotic direction, and that for every direction there is a geodesic almost surely. However, while uniform curvature seems plausible, there is no model of first-passage percolation with i.i.d.\ weights for which it has been verified. As will become clear throughout this paper, the lack of knowledge about $\ball$ is the major factor limiting our understanding of $\tree_0$.

Inspired by the work of Newman and his collaborators, later work has aimed at obtaining rigorous results, without further assumptions on the limiting shape. Proving that
$$
\P\big(\tree_0\text{ has size at least }2\big)=1
$$
is already non-trivial, and was established through a series of papers~\cite{hagpem98,garmar05,hoffman05}. Hoffman~\cite{hoffman08} used Busemann functions to show that $\tree_0$  contain at least one geodesic for each side\footnote{The \emph{sides} corresponds to tangent lines, of which there are $n$ if $\partial\ball$ is an $n$-gon and $\infty$ otherwise.} of $\partial\ball$, almost surely, and thus showed that
$$
\P\big(\tree_0\text{ has size at least }4\big)=1.
$$
Damron and Hanson~\cite{damhan14} strengthened these results to show that these geodesics are coalescing and asymptotically directed in the intersection of $\partial\ball$ and one of its tangent lines.
Recent work by Georgiou, Rassoul-Agha and Sepp\"al\"ainen~\cite{georassep17a,georassep17b} parallels this development in the related setting of last-passage percolation.
One of the main goals of this paper is to show that for many of the properties described in~\cite{damhan14} not only some, but all geodesics have these properties. Among these properties we find asymptotic `generalized' directions and linear Busemann function, and for other properties such as coalescence, we will show that the behavior described in these papers is in some sense typical.

Since comparisons between geodesics will recur throughout the paper, we provide here a small glossary on infinite geodesics. We shall by $\tree_v=\tree_v(\omega)$ denote the set of infinite geodesics originating from the vertex $v\in\Z^2$. Two geodesics $g\in\tree_u$ and $g'\in\tree_v$, starting from different vertices, will be said to {\bf intersect} if $g$ and $g'$ both visit some vertex $z$, and to {\bf coalesce} if the symmetric difference $g\Delta g'$ is finite. They are said to be {\bf non-crossing} if they either coalesce or are disjoint, and are said to {\bf cross} if they are not non-crossing.

We shall below give a flavor of the consequences of the theory developed in this paper. The fundamental object that we study is a random coalescing geodesic. As before, we formally think of an infinite geodesic as encoded by an element of $\Omega_2=\{0,1\}^{\lobby}$. Below, $\sigma_v$ will denote the usual shift that acts coordinate-wise by sending $z$ to $z-v$.

\begin{define}
We say that a measurable map $G:\Omega_1\to\Omega_2$ is a {\bf random coalescing geodesic} if for almost every $\omega\in\Omega_1$ and every $v\in\Z^2$ we have $G(\omega)\in\tree_0(\omega)$, i.e.\ $G(\omega)$ is an infinite geodesic starting at the origin, and $G(\omega)$ and $\sigma_{-v}(G(\sigma_v\omega))$ coalesce.
We shall below write $G(v)$ as a shorthand for the map $\sigma_{-v}\circ G\circ\sigma_v$.
\end{define}

A random coalescing geodesic is a nice object as it allows for an efficient use of the ergodic theorem. We shall explore these implications in detail in Section~\ref{properties}. We remark that the study of stationary random walks, carried out in parallel by Chaika and Krishnan~\cite{chakri}, is closely related to our concept of a random coalescing geodesic. Their results show that many stationary ways of generating non-crossing paths in the plane will give rise to paths with similar properties as a random coalescing geodesic.

Our study of random coalescing geodesics has two major goals. The first is to classify all random coalescing geodesics. The second is to use this classification to make statements about all geodesics in $\tree_0$. In order to accomplish the second goal it would be nice if every $g \in \tree_0$ was the image of some random coalescing geodesic. However, widely believed conjectures (for example that there is a geodesic in every direction) imply that this is not the case.\footnote{For each $z\in\Z^2$ for which a branch of $\tree_0$ splits in two, the clockwise-most geodesic using the counterclockwise-most edge out of $z$ and the counterclockwise-most geodesic using the clockwise-most edge out of $z$, either have the same direction or there is an interval of directions where there is no geodesic.\label{prediction}}
Instead we shall show that random coalescing geodesics are sufficiently dense in $\tree_0$ that we can still use our understanding of random coalescing geodesics to make statements about every geodesic in $\tree_0$. Together, the results to be stated in the remainder of this section provide an ergodic theory for the set of infinite geodesics. They show that for almost every realization of the weight configuration $\omega\in\Omega_1$ the set $\tree_0(\omega)$ has the same asymptotic properties when it comes to cardinality, directions and Busemann functions.

We first address the number of topological ends of the tree encoded by $\tree_0$. Recall that the cardinality of a set is either some finite number, countably infinite, or uncountably infinite.

\begin{theorem}\label{geo cardinality}
The cardinality of the set $\tree_0$ is almost surely constant.
\end{theorem}

In order to state our remaining theorems precisely we introduce a few definitions.
Busemann functions, named after the work of Herbert Busemann on metric spaces~\cite{busemann55}, were introduced to first-passage percolation through the works of Newman \cite{newman95} and Hoffman~\cite{hoffman05,hoffman08}. Given a geodesic $g \in \tree_0$ we define the {\bf Busemann function} 
$B_g:\Z^2\times\Z^2\to\R$ of $g=(0,v_1,v_2,\ldots)$ as the limit
\begin{equation}\label{gBusemann}
B_g(x,y):=\lim_{k \to \infty}\big[T(x,v_k)-T(y,v_k)\big].
\end{equation}
It is proved in~\cite{hoffman05} that this limit exists for all $g \in \tree_0$ and $x,y \in \Z^2$ almost surely; see Lemma~\ref{rioja} below. We say that a Busemann function is {\bf asymptotically linear} if there exists a linear functional $\rho:\R^2 \to \R$ such that 
$$
\limsup_{|y| \to \infty} \frac{1}{|y|}\big|B_g(0,y)-\rho(y)\big|=0.
$$

The Busemann function of a geodesic $g$ should be thought of as a measure of the difference in distance to infinity along the geodesic $g$. From the linearity of a Busemann function it is possible to obtain information on the direction of the geodesics used to define it, and hence to distinguish geodesics from one another. An exposition of this will be given in Section~\ref{properties}.

We shall call a linear functional $\rho:\R^2\to\R$ {\bf supporting} to $\ball$ if $\{x\in\R^2:\rho(x)=1\}$ is a supporting line for $\partial\ball$ at some point in $\partial\ball$, and {\bf tangent} to $\ball$ if $\{x:\rho(x)=1\}$ is the unique supporting line -- the tangent line -- to $\partial\ball$ at some point in $\partial\ball$. 
It is well known that $\ball$ can be expressed as $\{x\in\R^2:\mu(x)\le1\}$ for some norm $\mu:\R^2\to\R$. Consequently, the intersection of $\partial\ball$ and a supporting line of the form $\{x\in\R^2:\rho(x)=1\}$ can thus be represented by the arc $\{x\in S^1:\mu(x)=\rho(x)\}$.
The set of functionals supporting to $\ball$ is naturally parametrized by the direction of their gradients. The set of functionals supporting to $\ball$, which we denote henceforth by $\support$, thus inherits the topology of $S^1$.

\begin{theorem}\label{all geos}
With probability one there exists for each $g\in\tree_0$ a linear functional $\rho\in\support$ such that $B_g$ is asymptotically linear to $\rho$ and $\dir(g)$ is a subset of $\{x\in S^1:\mu(x)=\rho(x)\}$.
\end{theorem}

The above theorem provides a simultaneous description of the asymptotic properties of all infinite geodesics, but does not address existence and uniqueness of a specific functional. Since for almost every realization all geodesics have linear Busemann functions we shall proceed and describe the set of linear functionals associated to geodesics in $\tree_0$. Let
$$
\functionals(\omega):=\big\{\rho\in\support:\exists\text{ a geodesic in $\tree_0(\omega)$ with Busemann function linear to }\rho\big\}.
$$
The next result gives a more precise description of the ergodic properties of $\tree_0$ by addressing the topological properties of the random set $\functionals=\functionals(\omega)$.

\begin{theorem}\label{unique support}
There exists a closed set $\functionals_\star\subseteq\support$, containing all linear functionals tangent to $\ball$, such that $\P(\functionals=\functionals_\star)=1$.
Moreover, for every functional $\rho\in\functionals_\star$ we have
$$
\P\big(\exists\text{ two geodesics in $\tree_0$ with Busemann function linear to }\rho\big)=0.
$$
\end{theorem}

We remark that, in particular, if $\partial\ball$ is differentiable then $\functionals(\omega)=\support$ almost surely.


The remainder of this paper is organized as follows.
We first, in Section~\ref{sec: assumptions}, review the relevant background on first-passage percolation, and describe in detail the class of models with which we shall work.
In Section~\ref{properties} we describe some fundamental properties of random coalescing geodesics, and at the same time illustrate the role of coalescence and Busemann functions in the construction of an ergodic theory.
The existence of at least four random coalescing geodesics is derived in Section~\ref{dh}, based on previous work of Damron and Hanson~\cite{damhan14}.

In Sections~\ref{cola}--\ref{sec:coalescence} we then aim to characterize the set of random coalescing geodesics. In Section~\ref{cola} we introduce a shift invariant labeling of geodesics which is consistent with some natural ordering among geodesics. This gives us a way to identify geodesics by referring to their labels.
In Section~\ref{geometric} we present a central geometric argument that will be crucial in order to develop our theory without further assumptions on the asymptotic shape.
The set of labels obtained by the labeling procedure of Section~\ref{cola} is in Section~\ref{label of sections} shown to be a deterministic closed set, and that geodesics with the same label tend to form non-crossing families of geodesics. This allows us to talk about a \emph{random} non-crossing geodesic with a given label.
These random non-crossing geodesics are in Section~\ref{sec:coalescence} proven to be coalescing, via the adaptation of an argument due to Licea and Newman~\cite{licnew96}. We further show that there are no random coalescing geodesics apart from these ones, and hence obtain a classification of random coalescing geodesics in terms of the labeling.

We end the paper by exploring several consequences of the theory we develop and, in particular, prove the theorems stated above. In Section~\ref{other} we explore ergodic properties of infinite geodesics by linking labels to Busemann functions. In Section~\ref{midpoint} we resolve the midpoint problem from~\cite{benkalsch03}. These last two sections can be read independently of one another.
Finally, we state some open problems in Section~\ref{sec: open}.

\section{Background on model and assumptions} \label{sec: assumptions}

We shall work under the standard assumptions on passage time distributions outlined by earlier work of Hoffman~\cite{hoffman08} and Damron and Hanson~\cite{damhan14}. As above, we will denote by $\Omega_1=[0,\infty)^{\lobby}$ our state space, equipped with the product Borel sigma-algebra, completed with respect to $\P$. $\P$ will throughout the paper be a shift invariant probability measure on $\Omega_1$ satisfying either of the following two sets of conditions:
\begin{enumerate}
\item[{\bf A1}] $\P$ is a product measure whose common marginal distribution is continuous with
$$
\E\big[\min\{\omega_{e_1},\omega_{e_2},\omega_{e_3},\omega_{e_4}\}^2\big]<\infty,
$$
where $e_1, \ldots, e_4$ denote the four edges incident to the origin.
\item[{\bf A2}] $\P$ is ergodic with respect to translations of $\Z^2$ and has the following properties:
\begin{enumerate}[(i)]
\item $\P$ has all the symmetries of $\mathbb{Z}^2$;
\item $\E[\omega_e^{2+\eps}]<\infty$ for some $\eps>0$;
\item the asymptotic shape $\ball$ is bounded;
\item for any two finite paths $\pi$ and $\pi'$ that differ for at least one edge we have
$$
\P\big(T(\pi)=T(\pi')\big)=0;
$$
\item for any $e\in\lobby$ and $t>0$ such that $\P(\omega_e>t)>0$ we have almost surely
$$
\P\big(\omega_e>t\,\big|\,\{\omega_f:f\in\lobby,f\neq e\}\big)>0;
$$
\item there exists $L<\infty$ such that 
$$
\P\big(|\geo(x,y)|\le Ln\text{ for all }x,y\in[-n,n]^2\big)=1-o(1);
$$
\item 
there exists $t,\delta,\gamma>0$ such that
\begin{enumerate}[(a)]
\item for any edge $e\in\lobby$ we have almost surely
$$
\P\big(\omega_e<t-\delta \,\big|\, \{\omega_f:f\in\lobby,f\neq e\}\big)>\delta;
$$
\item with probability $1-o(1)$ we have for all $x\in\partial[-n,n]^2$ and $y\in\partial[-2n,2n]^2$
$$
\big|\big\{e\in\geo(x,y):\omega_e>t\big\}\big|>\gamma|\geo(x,y)|.
$$
\end{enumerate}
\end{enumerate}
\end{enumerate}

Conditions {\bf A1} and {\bf A2}~(i)-(iv) are (essentially) the conditions of~\cite{hoffman08}, and have been specified so to make sure the conditions of the shape theorem (see below) are satisfied, and that for each pair of points $x$ and $y$ there is a unique geodesic.
Condition~(v), known as the \emph{upward finite energy} condition, was added in~\cite{damhan14} to allow for local modifications of an edge configuration; see e.g.~\cite{hagjon06} for a further account on its relevance in the statistical mechanics literature.
Condition~(vi) implies that the number of edges in geodesics between faraway points is comparable to the Euclidean distance between them.
The first part of condition~(vii) is a uniform version of the \emph{downward finite energy} condition, and introduced so that if we resample any edge where the passage time is not too low then, with probability uniformly bounded away from 0, the weight will decrease by a fixed amount. The second part of condition~(vii) makes sure that there is a positive fraction of edges along a geodesic whose passage time is not too low.

%

\begin{remark}
The finite energy conditions of {\bf A2} clearly hold for any measure satisfying {\bf A1}. Under the conditions of {\bf A1} condition~(vi) of {\bf A2} can be obtained as a corollary of a theorem of Kesten~\cite{kesten80b}, while
the second part of condition~(vii) follows from a standard percolation argument. It will therefore suffice to work with the conditions of {\bf A2}.
\end{remark}

\begin{remark}
Most of our arguments require only conditions (i)-(v) of {\bf A2}. It is only the proof of Lemma~\ref{brassica}
that require the additional conditions~(vi) and~(vii). We have made no effort to optimize these conditions. 
\end{remark}

\begin{remark}
While the total ergodicity condition (i) in {\bf A2} makes many arguments easier, we do not believe that it is essential. With some extra effort we believe that it can be replaced by $\P$ is ergodic in any of our arguments or those in~\cite{damhan14}.
\end{remark}

\subsection{The shape theorem}

One of the most celebrated results in first-passage percolation is known as the shape theorem, and originates from Kingman's~\cite{kingman68,kingman73} ergodic theory for subadditive processes. Under either of the conditions {\bf A1} or {\bf A2} Kingman's theorem shows that for any $z\in\Z^2$ we have
\begin{equation}\label{time constant}
\exists\mu(z)\,:=\,\lim_{n\to\infty}\frac{T(0,nz)}{n}\,=\,\inf_{n\ge1}\frac{\E[T(0,nz)]}{n}\quad\text{almost surely and in }L^1.
\end{equation}
Richardson~\cite{richardson73}, and later Cox and Durrett~\cite{coxdur81} and Boivin~\cite{boivin90}, extended the radial converge in~\eqref{time constant} to obtain simultaneous convergence in all directions. Their results show that under either of the assumptions {\bf A1} or {\bf A2} we have
\begin{equation}\label{shape theorem}
\limsup_{|z|\to\infty}\frac{1}{|z|}\big|T(0,z)-\mu(z)\big|=0\quad\text{almost surely}.
\end{equation}
The function $\mu:\Z^2\to\R$ extends to a function $\mu$ on $\R^2$ through homogeneity, and inherits the properties of a norm. The unit ball $\ball:=\{x\in\R^2:\mu(x)\le1\}$ in this norm is a good approximation of a rescaled version of a large ball $\{z\in\Z^2:T(0,z)\le t\}$ in the first-passage metric. In these terms the result in~\eqref{shape theorem} is known as the {\bf shape theorem} as takes the familiar form
\begin{equation} \label{surgery}
\mathbb{P}\Big( (1-\epsilon)\ball \subset \frac1t \ball(t) \subset (1+\epsilon)\ball \text{ for all large }t \Big) = 1,
\end{equation}
for all $\eps>0$, where $\ball(t):=\{z\in\Z^2:T(0,z)\le t\}+[-1/2,1/2]^2$. It is straightforward to show that the properties of a norm implies that $\ball$ is compact, convex and has non-empty interior. The assumptions on $\P$ further imposes that $\ball$ necessarily has all the symmetries of $\Z^2$.

\subsection{Shapes and geodesics in ergodic first-passage percolation} \label{bolt}

The shape theorem gives a first-order approximation of large balls in the first-passage metric $T$ with a compact and convex shape $\ball$, but does not provide further insight to the topological properties of that shape. The results of~\cite{hoffman08,damhan14} relate existence and properties of infinite geodesics to the number of sides of $\ball$, of which there are at least four. It turns out that these results are sharp under the general ergodic assumption, but most likely not for independent models. 

In the ergodic setting H{\"a}ggstr{\"o}m and Meester~\cite{hagmee95} have shown that for any compact and convex shape $S \subset \R^2$ with the symmetries of $\Z^2$ there is a model of ergodic first-passage percolation with $\ball=S$. That is, the asymptotic shape can have as few sides as four, in the case that $S$ is either a square or a diamond.

Alexander and Berger~\cite{aleber18} have constructed a model of ergodic first-passage percolation where the shape is an octagon, with corners on the axes and the main diagonals, and where all geodesics (at least two for each axis direction) are directed along axes. In addition, Brito and Hoffman~\cite{brihof} have constructed a model that almost surely has exactly four infinite geodesics. These geodesics have directions that span an angle of $\pi/2$ each. These results suggest that the results of~\cite{hoffman08,damhan14} are sharp.

Very little is known about the asymptotic shape for edge weights that are independent. In particular, it is unknown whether for some edge distribution $\partial\ball$ may equal a circle or a square. Simulations indicate that for exponential edge weights the shape is very close, but not equal to, a circle~\cite{almdei15}. Getting better results about geodesics in independent first-passage percolation will require new techniques for the shape or geodesics which make use of independence in some fundamental way.

While our focus in this paper is strictly on first-passage percolation on the $\Z^2$ nearest-neighbor lattice, we mention that it has been observed by Benjamini and Tessera~\cite{bentes17} that bigeodesics may exist on graphs with different geometries.

\subsection{An extended shape theorem}

As emphasized above, there is a close connection between the asymptotic shape and infinite geodesics. As such, the shape theorem provides a convenient tool to obtain certain control on the location of geodesics. The following `extended version' of the usual shape theorem will be useful to us as several occasions. Loosely speaking, it says that the shape theorem has `kicked in' around a point at a time scale proportional to its distance from the origin, for all points sufficiently far from the origin.

\begin{prop}\label{extended shape}
For every $\eps>0$ there exists an almost surely finite $N\ge1$ such that for all $|z|\ge N$ and all $y\in\Z^2$ we have
$$
\big|T(z,z+y)-\mu(y)\big|\le\eps\max\{|z|,|y|\}.
$$
\end{prop}

\begin{proof}
Given $\eps>0$ and $z\in\Z^2$, let $C(\eps,z)$ denote the event
$$
C(\eps,z):=\big\{|T(z,z+y)-\mu(y)|\le\eps\max\{|z|,|y|\}\text{ for all }y\in\Z^2\big\}.
$$
We first argue that $\P(C(\eps,z))\to1$ as $|z|\to\infty$. This is an immediate consequence of the shape theorem: For every $\eps>0$ we may find $M\ge1$ such that
$$
\P\Big(|T(z,z+y)-\mu(y)|\le\eps|y|\text{ for all }|y|\ge M\Big)\ge1-\eps,
$$
and
$$
\P\Big(\max_{|y|\le M}|T(z,z+y)-\mu(y)|\le\eps|z|\Big)\ge1-\eps
$$
for $|z|$ large. Hence, for every $\eps>0$ we find $N$ so that $\P(C(\eps,z))>1-2\eps$ when $|z|\ge M'$.

Relying on the ergodic theorem we may for any $\delta>0$ find an almost surely finite $N\ge1$ such that for every $n\ge N$ the density of $z$ within distance $n$ from the origin for which $C(\eps/100,z)$ fails is at most $\delta$. So, for $z$ with $|z|\ge N$ either $C(\eps/100,z)$ occurs, or we may find $x$ within distance $2\sqrt{\delta}|z|$ of $z$ for which $C(\eps/100,x)$ occurs. (The contrary would give $4\delta|z|^2$ points within distance $|z|$ of the origin for which $C(100/\eps,z)$ fails.) In the latter case we have for these $z$ and $x$, and any $y$, that 
\begin{equation*}
\begin{aligned}
\big|T(z,z+y)-\mu(y)\big|\,&\le\,\big|T(x,z+y)-\mu(z+y-x)\big|+T(x,z)+\big|\mu(z+y-x)-\mu(y)\big|\\
&\le\,\frac{\eps}{100}\max\{|x|,|z+y-x|\}+\frac{\eps}{100}\max\{|x|,|z-x|\}+2\mu(z-x),
\end{aligned}
\end{equation*}
where we in the second step have used the triangle inequality once and the fact that $C(\eps/100,x)$ occurs twice. Using the fact that $\mu$ is comparable to Euclidean distance and that $|z-x|\le2\sqrt{\delta}|z|$, we obtain a constant $c$ and the further upper bound
$$
\big|T(z,z+y)-\mu(y)\big|\le\frac{\eps}{100}|y|+2\frac{\eps}{100}(1+\sqrt{\delta})|z|+2c\sqrt{\delta}|z|,
$$
valid for $|z|\ge N$ and all $y$.
In particular, if $\delta>0$ was chosen small enough this is all bounded by $\eps\max\{|z|,|y|\}$ and $C(\eps,z)$ holds, for all $z$ at distance at least $N$ from the origin, as required.
\end{proof}

\subsection{A note on measurability}

We have defined a random coalescing geodesic as a measurable map $G:\Omega_1\to\Omega_2$, where both $\Omega_1=[0,\infty)^{\lobby}$ and $\Omega_2=\{0,1\}^{\lobby}$ are equipped with the product topology and Borel sigma-algebra. Hence, events of the form $\{v\in G\}$ and $\{G\cap V\neq\emptyset\}$ are measurable. Moreover, since events like $\{B_G(y,z)>x\}$ can be written as countable unions and intersections of events involving $G$ and $T$, they are also measurable. So for $y,z\in\Z^2$ the Busemann function of $G$ is measurable as a function $B_G(y,z):\Omega_1\to\R$. These observations will allow us to rely on the ergodic theorem as we address various properties of $G$.

We also consider a number of set-valued functions throughout this paper. Examples of these are the set of directions $\dir(G)\subseteq S^1$ of a random coalescing geodesic, the set $\functionals\subseteq\support$ of functionals linear to the Busemann function of some geodesic, and the set $\labels\subseteq[0,1]$ of values appearing as the label of some geodesic in $\tree_0$. Events of the form $\{\dir(G)\cap I\neq\emptyset\}$, $\{\functionals\cap I\neq\emptyset\}$ and $\{\labels\cap I\neq\emptyset\}$, where $I$ is some interval, are easily seen to be measurable. It is less obvious that these sets are measurable as set-valued functions on $\Omega_1$. However, we show that these sets are constant outside of some null set, and thus automatically measurable when completing the Borel sigma-algebra with respect to the measure $\P$.

\section{Properties of random coalescing geodesics}
\label{properties}

In this section we collect some fundamental properties of random coalescing geodesics and simultaneously highlight the usefulness of Busemann functions and coalescence. First we give the statement of all the theorems and then we provide the proofs. Although random coalescing geodesics is a new concept, we mention that many of the arguments used to prove the results of this section has previously appeared elsewhere in the literature.

Recall that the Busemann function $B_g:\Z^2\times\Z^2\to\R$ of a geodesic $g=(v_0,v_1,v_2,\ldots)$ is defined as the almost sure limit
\begin{equation}\label{newBusemann}
B_g(x,y)=\lim_{k \to \infty}\big[T(x,v_k)-T(y,v_k)\big].
\end{equation}

\begin{lemma}\label{rioja}
With probability one the limit in~\eqref{newBusemann} exists for all $g\in\tree_0$ and satisfies
\begin{enumerate}[\quad (a)]
\item $B_g(x,z)=B_g(x,y)+B_g(y,z)$ for all $x,y,z\in\Z^2$;
\item $|B_g(x,y)|\le T(x,y)$ for all $x,y\in\Z^2$;
\item $B_g(x,y)=T(x,y)$ for all $x,y\in g$ such that $x\in\geo(0,y)$.
\end{enumerate}
\end{lemma}

The Busemann function of a random coalescing geodesic $G$ has especially nice properties. Since $G$ is coalescing, $B_G$ does not depend on the representative of $G$, and hence $B_G=B_{G(v)}$ for all $v\in\Z^2$. This leads to the equality
\begin{equation}\label{eq:trans}
B_G(x+v,y+v)=B_{G(v)}(x+v,y+v)=B_G(x,y)\circ\sigma_v.
\end{equation}
In particular, we have $B_G(x,y)$ and $B_G(x+v,y+v)$ equal in law. Together with additivity these observations imply that $B_G$ is asymptotically linear to some linear functional $\rho_G$.

\begin{prop}\label{sonoma}
Let $G$ be a random coalescing geodesic. There exists a linear functional $\rho_G:\R^2\to\R$ satisfying $\rho_G(z)=\E[B_G(0,z)]$ for all $z\in\Z^2$, and
$$
\P\Big(\limsup_{|z|\to\infty}\frac{1}{|z|}\big|B_G(0,z)-\rho_G(z)\big|=0\Big)=1.
$$
\end{prop}

Examining the proof of Proposition~\ref{extended shape} we note that Proposition~\ref{sonoma} extends in the same manner to the following stronger statement (proof omitted): For every $\eps>0$ we have
\begin{equation}\label{extended Busemann}
\P\Big(\big|B_G(z,z+y)-\rho_G(y)\big|\le\eps\max\{|z|,|y|\}\text{ for all $y$ and $|z|$ large}\Big)=1.
\end{equation}

Linearity of the Busemann function and the shape theorem together provide a bound on the asymptotic direction of $G$: It is confined by the intersection of the line $\{\rho_G(x)=1\}$ and the asymptotic shape $\partial\ball=\{\mu(x)=1\}$. Recall that the set of directions of $G$ is defined as the set of $x\in S^1$ such that $v_k/|v_k|\to x$ for some subsequence $(v_k)_{k\ge1}$ of $G$.

\begin{prop}\label{mendoza}
Let $G$ be a random coalescing geodesic. The line $\{x\in\R:\rho_G(x)=1\}$ is a supporting line for $\partial\ball$, and the set of directions $\dir(G)$ is a deterministic subset of $\arc(\rho_G):=\{x\in S^1:\mu(x)=\rho_G(x)\}$.
\end{prop}

We say that a random coalescing geodesic $G$ {\bf eventually moves into a half-plane} $H$ if for all parallel half-planes $H'\subseteq H$ and all $v\in\Z^2$ we have
$$
\P\big(|G(v)\cap H'^c|<\infty\big)=1.
$$
Let $H_i$, for $i=0,1,2,\ldots,7$, denote the half-plane\footnote{Here and below `$\cdot$' will denote inner product.}
$$
H_i:=\big\{x\in\R^2:x\cdot(\cos(i\pi/4),\sin(i\pi/4))\ge0\big\}.
$$
Due to convexity and symmetry we have that $\arc(\rho_G)$ has width at most $\pi/2$. It may thus contain $i\pi/4$ for at most three values of $i$. Consequently $\arc(\rho_G)$ is contained in the interior of some $H_i$, and it follows that any random coalescing geodesic $G$ eventually moves into one of the eight half-planes $H_i$.

Our next proposition says, in particular, that for any two distinct random coalescing geodesics $G$ and $G'$ we have that $\arc(\rho_G)$ and $\arc(\rho_{G'})$ share at most one point.

\begin{prop}\label{bordeaux}
Let $G$ and $G'$ be random coalescing geodesics such that $\rho_G=\rho_{G'}$. Then, $G=G'$ almost surely.
\end{prop}

Finally, we show that a given random coalescing geodesic cannot be part of a bigeodesic. We say that a random coalescing geodesic $G$ is {\bf backwards finite} if for every $x\in\Z^2$ we have $x\in G(y)$ for at most finitely many $y\in\Z^2$.

\begin{prop}\label{temecula}
A random coalescing geodesic is almost surely backwards finite.
\end{prop}

Lemma~\ref{rioja} has its origins in~\cite{hoffman05}. Results similar to Propostions~\ref{sonoma},~\ref{mendoza} and~\ref{temecula} have previously appeared in~\cite{damhan14}. A result similar to Proposition~\ref{bordeaux} has previously been obtained in~\cite{georassep17a,georassep17b} and~\cite{damhan17}. The proofs presented below differ from these in some details.

\subsection{Proof of Lemma~\ref{rioja}}

Let $g=(v_0,v_1,v_2,\ldots)$ be a geodesic and note that
$$
T(x,v_k)-T(y,v_k)\,=\,[T(x,v_k)-T(v_0,v_k)]-[T(y,v_k)-T(v_0,v_k)].
$$
The two expressions on the right-hand side are decreasing in $k$ and bounded from below by $-T(x,v_0)$ and $-T(y,v_0)$. Hence, the limit as $k\to\infty$ exists for each $g\in\tree_0$ almost surely.
The remaining properties are easy consequences of the definition and subadditivity of $T$.

\subsection{Proof of Proposition~\ref{sonoma}}

Let $\rho_G:\R^2\to\R$ be the linear functional defined by\footnote{Here, and throughout the paper, ${\bf e}_1$ and ${\bf e}_2$ will denote the coordinate vectors $(1,0)$ and $(0,1)$.} $x\mapsto\big(\E[B_G(0,{\bf e}_1)],\E[B_G(0,{\bf e}_2)]\big)\cdot x$. Using translation invariance and additivity of $B_G$, we find that for $v=(v_1,v_2)$ with non-negative coordinates that
$$
\E[B_G(0,v)]=\E\Big[\sum_{i=1}^{v_1}B_G((i-1){\bf e}_1,i{\bf e}_1)+\sum_{j=1}^{v_2}B_G((j-1){\bf e}_2,j{\bf e}_2)\Big]=\rho_G(v),
$$
and similarly for any other point $v$. Using the ergodic theorem we have by additivity and iterated use of~\eqref{eq:trans} that, almost surely,
\begin{equation}\label{eq:sonoma1}
\lim_{n\to\infty}\frac{1}{n}B_G(0,nv)\,=\,\lim_{n\to\infty}\frac1n\sum_{k=0}^{n-1}B_G(0,v)\circ\sigma_v^k\,=\,\rho_G(v)\quad\text{for all }v\in\Z^2.
\end{equation}

We wish to strengthen the radial convergence stated in~\eqref{eq:sonoma1} to simultaneous convergence in all directions. Given $\eps>0$, pick $v_1,v_2,\ldots,v_m$ such that for every $z\in\Z^2$ we have $|z-nv_k|<\eps|z|$ for some $n$ and $k$; write $v_z$ for the point of the form $nv_k$ that minimizes $|z-nv_k|$. By~\eqref{eq:sonoma1} we may choose $N<\infty$ large so that
$$
|B_G(0,v_z)-\rho_G(v_z)|\le\eps|v_z|\quad\text{for all }|z|\ge N.
$$
By additivity of $B_G$ and the triangle inequality, we find that
$$
|B_G(0,z)-\rho_G(z)|\le|B_G(0,v_z)-\rho_G(v_z)|+|B_G(v_z,z)|+|\rho_G(v_z-z)|.
$$
Since $|v_z-z|<\eps|z|$ and $|B_G(v_z,z)|\le T(v_z,z)$, we obtain for some constant $C<\infty$
$$
|B_G(0,z)-\rho_G(z)|\le T(v_z,z)+C\eps|z|.
$$
However, by Proposition~\ref{extended shape} we find $M<\infty$ such that for all $|z|\ge M$
$$
|T(z,z+y)-\mu(y)|\le\eps\max\{|z|,|y|\}\quad\text{for all }y\in\Z^2.
$$
Together with the above, we obtain for large $|z|$ that
$$
|B_G(0,z)-\rho_G(z)|\le \big(C+\max_{|x|=1}\mu(x)+1\big)\eps|z|.
$$
Since $\eps>0$ was arbitrary, this completes the proof.

\subsection{Proof of Proposition~\ref{mendoza}}

By the properties of $B_G$, we have for any sequence $(x_n)_{n\ge1}$ such that $|x_n|\to\infty$ and $x_n/|x_n|\to x$, that
$$
\rho_G(x_n/|x_n|)=\frac{1}{|x_n|}\E[B_G(0,x_n)]\le\frac{1}{|x_n|}\E[T(0,x_n)].
$$
Taking limits leaves $\rho_G(x)\le\mu(x)$.

If $x\in S^1$ is a direction of $G$, then there exists a subsequence $(x_n)_{n\ge1}$ of points on $G$ such that $x_n/|x_n|\to x$, and
$$
\frac{1}{|x_n|}B_G(0,x_n)=\frac{1}{|x_n|}T(0,x_n).
$$
Due to Proposition~\ref{sonoma} and the shape theorem, taking limits leaves us with $\rho_G(x)=\mu(x)$. It follows that $\{\rho_G(x)=1\}$ is a supporting line for $\partial\ball=\{\mu(x)=1\}$, and that every point in $\dir(G)$ is contained in $\arc(\rho_G)$.

It remains to conclude that the set of limiting directions of $G$ is almost surely constant. Since the set of directions is a closed interval, it suffices to show that the endpoints of this interval are deterministic. We first claim that for any interval $I$ we have
\begin{equation}\label{mendoza-implication}
\P\big(\dir(G)\cap I\neq\emptyset\big)>0\quad\Rightarrow\quad\P\big(\dir(G)\cap I\neq\emptyset\big)=1.
\end{equation}
To see this, note that the left-hand side and the ergodic theorem gives a density of points for which $\dir(G(v))$ intersects $I$. Since $G(v)$ and $G(0)$ coalesce, the implication follows.

Now, either $\arc(\rho_G)$ is a single point, in which case there is nothing to prove, or we can find a nested deterministic sequence of closed intervals $I_0\supseteq I_1\supseteq I_2\supseteq\ldots$ as follows: Let $I_0=\arc(\rho_G)$, and in each step we obtain $I_k$ from $I_{k-1}$ by splitting it in half and choosing the counterclockwise-most interval which has non-empty intersection with $\dir(G)$ with positive probability. By~\eqref{mendoza-implication}, these sets intersect $\dir(G)$ with probability one. The intersection of these intervals contains a unique point which is the counterclockwise-most point in $\arc(\rho_G)$ to be contained in $\dir(G)$. By an analogous argument we find also the clockwise-most element of $\dir(G)$.

\subsection{Interlude on half-plane geodesics}

Let $G$ be a random coalescing geodesic and assume that $G$ eventually moves into one of the eight half-planes $H=H_i$. Let $T_H$ denote the restriction of the first-passage metric to $H$ (that is, set $\omega_e=\infty$ if $e$ has some endpoint outside $H$). We define $G_H:\Omega_1\to\Omega_2$ as the limit
\begin{equation}\label{eq:GH}
G_H:=\lim_{k\to\infty}\geo_H(0,v_k),
\end{equation}
where $\geo_H(x,y)$ denotes the geodesic with respect to $T_H$ and $G=(0,v_1,v_2,\ldots)$.

\begin{lemma}\label{valpolicella}
Let $G$ be a random coalescing geodesic that eventually moves into the half-plane $H$. Then the limit in~\eqref{eq:GH} exists and $G_H$ and $G$ coalesce almost surely. Moreover,
$$
B_G^H(0,z):=\lim_{k\to\infty}[T_H(0,v_k)-T_H(z,v_k)]
$$
exists almost surely, and for $z$ on the boundary of $H$ it equals $\rho_G(z)$ in expectation.
\end{lemma}

\begin{proof}
We assume that $H$ is the right half-plane as the remaining cases are similar. We then observe that as $G$ eventually moves into $H$, there will be a density of points $z$ on the boundary of $H$ for which $G(z)$ is entirely contained in $H$. We may thus find $m<0<n$ for which this occurs, and note that $\geo_H(0,v_k)$ is sandwiched between $G(m{\bf e}_2)$ and $G(n{\bf e}_2)$ for all $k$. Since $G$ is coalescing the limit as $k\to\infty$ exists and coalesces with $G$.

That $B_G^H$ exists, is additive and invariant with respect to translations along the boundary of $H$ follow as before. Additivity and the ergodic theorem further give that
$$
\lim_{n\to\infty}\frac{1}{n}B_G^H(0,n{\bf e}_2)\,=\,\E[B_G^H(0,{\bf e}_2)]\quad\text{almost surely}.
$$
Since $B_G^H(m{\bf e}_2,n{\bf e}_2)=B_G(m{\bf e}_2,n{\bf e}_2)$ whenever $G(m{\bf e}_2)=G_H(m{\bf e}_2)$ and $G(n{\bf e}_2)=G_H(n{\bf e}_2)$, the limit can only equal $\rho_G({\bf e}_2)$. In particular, $\E[B_G^H(0,{\bf e}_2)]=\rho_G({\bf e}_2)$.
\end{proof}

\subsection{Proof of Proposition~\ref{bordeaux}}

Assume that $\rho_G=\rho_{G'}$. By Proposition~\ref{mendoza} the two geodesics are confined to the same sector of width at most $\pi/2$. There is thus a half-plane $H=H_i$ for which both $G$ and $G'$ eventually move into. We assume that $H$ is the right half-plane; the remaining cases being similar.

It follows from Lemma~\ref{valpolicella} that $G=G'$ if and only if $G_H=G_H'$, and shall therefore work with half-plane geodesics. Aiming for a contradiction we assume that $G_H\neq G_H'$ with positive probability. Since $G_H$ and $G_H'$ are coalescing, this will then have to hold with probability one. For the same reason, either $G_H$ lies asymptotically above (or counterclockwise of) $G_H'$ with probability one, or the other way around. We shall assume that $G_H$ lies asymptotically above $G_H'$ almost surely.

For $k,\ell\in\Z$, let
$$
\Delta_H(k,\ell):=B_G^H(k{\bf e}_2,\ell{\bf e}_2)-B_{G'}^H(k{\bf e}_2,\ell{\bf e}_2).
$$
By Lemma~\ref{valpolicella} and the assumption that $\rho_G=\rho_{G'}$ we have $\E[\Delta_H(k,\ell)]=0$. We will aim to show that $\Delta_H(k,\ell)=0$ almost surely.

Given $k<\ell$, define three intersection points $x,y,z$ as follows: Since $G_H(k{\bf e}_2)$ and $G_H(\ell {\bf e}_2)$, and $G_H'(k{\bf e}_2)$ and $G_H'(\ell{\bf e}_2)$, coalesce we have that $G_H(k{\bf e}_2)$ and $G_H'(\ell{\bf e}_2)$ must intersect at some point $z$. Next, take $x$ on $G_H'(k{\bf e}_2)\cap G_H'(\ell{\bf e}_2)$ beyond $z$, and $y$ on $G_H(k{\bf e}_2)\cap G_H(\ell{\bf e}_2)$ beyond $z$; see Figure~\ref{fig:rcg1}.
\begin{figure}[htbp]
\begin{center}
\includegraphics{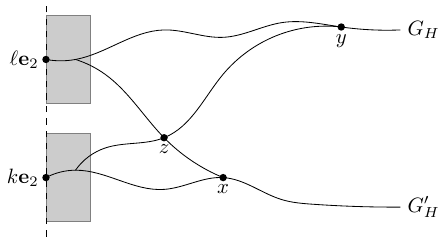}
\end{center}
\caption{The paths $G_H$ and $G_H'$ diverge before leaving the shaded regions.}
\label{fig:rcg1}
\end{figure}

By exploiting the intersection point $z$ we may construct two paths from $k{\bf e}_2$ to $x$: One being the segment of $G_H'(k{\bf e}_2)$ and one being the concatenation of the segments of $G_H(k{\bf e}_2)$ from $k{\bf e}_2$ to $z$ and of $G_H'(\ell{\bf e}_2)$ from $z$ to $x$. Denote this latter path by $\pi_x$. Similarly, we can construct two paths from $\ell{\bf e}_2$ to $y$: One being a segment of $G_H(\ell{\bf e}_2)$ and one being the concatenation of paths to $z$, which we denote by $\pi_y$. We see that
\begin{equation}\label{eq:some eq}
\begin{aligned}
\Delta_H(k,\ell)&=[T_H(k{\bf e}_2,y)-T_H(\ell{\bf e}_2,y)]-[T_H(k{\bf e}_2,x)-T_H(\ell{\bf e}_2,x)]\\
&=[T_H(\pi_y)-T_H(\ell{\bf e}_2,y)]+[T_H(\pi_x)-T_H(k{\bf e}_2,x)],
\end{aligned}
\end{equation}
which is non-negative. Since $\E[\Delta_H(k,\ell)]=0$ we conclude that $\Delta_H(k,\ell)=0$ almost surely.

Let $A_k$ denote the event that $|G_H(k{\bf e}_2)\cap G_H'(k{\bf e}_2)|\le m$. For large $m$ the event $A_k$ occurs with positive probability. Using the ergodic theorem we may find $k$ and $\ell>k+2m$ such that $A_k\cap A_\ell$ occurs. Maintaining the notation from before, the occurrence of $A_k\cap A_\ell$ implies that $z$ cannot lie on neither $G_H(\ell{\bf e}_2)$ nor $G_H'(k{\bf e}_2)$ (again, see Figure~\ref{fig:rcg1}). By~\eqref{eq:some eq}, since $\Delta_H(k,\ell)=0$, we have both $T_H(\pi_y)=T_H(\ell{\bf e}_2,y)$ and $T_H(\pi_x)=T_H(k{\bf e}_2,x)$, which is a contradiction to the assumption of unique passage times.

\subsection{Proof of Proposition~\ref{temecula}}

Consider the subgraph of the $\Z^2$ lattice containing all edges crossed by $G(v)$ for some $v\in\Z^2$. The resulting graph is connected, due to the coalescence of $G$, and does not contain any cycles, due to unique passage times. Since there is positive probability that neither $G(u)$ nor $G(v)$ contains the other, the ergodic theorem gives that a density of sites in this graph has degree at least three.

Assume that $G(0)$ is backwards infinite with positive probability. By the ergodic theorem there will be a density of sites $v\in\Z^2$ for which $G(v)$ is backwards infinite, and consequently a density of points where two backwards infinite points coalesce. These points are so-called \emph{trifurcation} points, meaning that if such a point would be removed it would disconnect the graph into (at least) three connected components. Via an induction argument one then shows that removing any $k$ trifurcation points disconnects the graph into at least $k+2$ infinite connected components. Consequently, the number of trifurcation points in a box of side-length $n$ cannot be larger than the number of points on the boundary, which contradicts the assumption that there is a positive density of trifurcation points.


\section{Existence of random coalescing geodesics}
\label{dh}

This section will contain two parts. We first recap some of the main results from~\cite{damhan14}, which we then use to establish the existence of random coalescing geodesics.

\subsection{Damron-Hanson geodesic measures}

In order to prove existence of geodesics with certain properties, Damron and Hanson~\cite{damhan14} worked on an enlarged probability space, on which they construct a family of limiting geodesics and keep track of Busemann functions at the same time. We will describe their procedure in some detail below.

Let $\rho:\R^2\to\R$ be any linear functional tangent to the asymptotic shape $\ball$. Let $\ell_\alpha:=\{x\in\R^2:\rho(x)\ge\alpha\}$ and consider the family $\mathcal{F}_\alpha=\{\geo(z,\ell_\alpha):z\in\Z^2\}$ of finite geodesics\footnote{Point-to-set geodesics are defined analogously as point-to-point geodesics; $\geo(z,V)$ is defined as the minimum weight path connecting $z$ to some point $v\in V$.} from points in $\Z^2$ to the half-plane $\ell_\alpha$. The goal will be to obtain a family of infinite geodesics by sending $\alpha$ to infinity. The family $\mathcal{F}_\alpha$ may be encoded as an element $\eta_\alpha\in\{0,1\}^{\bar{\mathcal{E}}^2}$, where $\bar{\mathcal{E}}^2$ denotes the set of \emph{directed} edges in $\Z^2$, as follows:
\begin{equation*}
\eta_\alpha(\bar e):=\left\{
\begin{aligned}
& 1 && \text{if $\bar e=(x,y)$ is traversed by the geodesic from $x$ to }\ell_\alpha,\\
& 0 && \text{otherwise}.
\end{aligned}
\right.
\end{equation*}
By considering ordered edges we may keep track of the direction in which an edge was traversed on its way to infinity.

In order to get their hands on the limiting object, Damron and Hanson encode alongside the finite geodesics their associated Busemann differences. Define for each $z\in\Z^2$ an element $\theta_\alpha(z)\in\R^2$ as follows:
$$
\theta_\alpha(z):=\big(T(z,\ell_\alpha)-T(z+{\bf e}_1,\ell_\alpha),T(z,\ell_\alpha)-T(z+{\bf e}_2,\ell_\alpha)\big).
$$
Encoded in $\theta_\alpha=(\theta_\alpha(z))_{z\in\Z^2}$ we find the difference in distance between any two points and the line $\ell_\alpha$, and thus serves as a Busemann function for the finite geodesics in $\mathcal{F}_\alpha$. Due to unique passage times, every site $z\not\in\ell_\alpha$ has out-degree one in the directed graph encoded by $\eta_\alpha$, and seen as an undirected graph $\eta_\alpha$ has no cycles.

Let $\Omega_1=[0,\infty)^{\lobby}$, $\Omega_2'=\{0,1\}^{\bar{\mathcal{E}}^2}$ and $\Omega_3=(\R^2)^{\Z^2}$. For each $\alpha\ge0$ we obtain a measurable map $\Psi_\alpha:\Omega_1\to\Omega_1\times\Omega_2'\times\Omega_3$ via $\omega\mapsto(\omega,\eta_\alpha,\theta_\alpha)$. Damron and Hanson use this map to push forward the measure $\P$ and obtain a measure $\nu_\alpha$ on $\widetilde\Omega:=\Omega_1\times\Omega_2'\times\Omega_3$, equipped with the product topology and Borel sigma-algebra. In order to obtain a limiting measure which is invariant with respect to translations, they consider the averages
$$
\nu_n^\ast(\cdot):=\frac1n\int_0^n\nu_\alpha(\cdot)\,d\alpha.
$$
From the observation that $\theta_\alpha(z)\le\big(\omega(z,z+{\bf e}_1),\omega(z,z+{\bf e}_2)\big)$ it follows that the sequence of measures $(\nu_n^\ast)_{n\ge1}$ is tight. Prokhorov's theorem then implies that $(\nu_n^\ast)_{n\ge1}$ has a weakly convergent subsequence. Damon and Hanson move on to show that every subsequential limit $\nu$ of the sequence $(\nu_n^\ast)_{n\ge1}$ is invariant with respect to translations, and since having out-degree one and no cycles are closed properties, these properties survive in the limit. Damon and Hanson prove in~\cite{damhan14}, among other things, the following:

\begin{theorem}\label{thm:DH}
Let $\rho:\R^2\to\R$ be a linear functional tangent to $\ball$. Every subsequential limit $\nu$ is invariant with respect to translations and satisfies the following properties: For $\nu$-almost every $(\omega,\eta,\theta)\in\widetilde\Omega$ and all $y,z\in\Z^2$ we have
\begin{enumerate}[\quad (a)]
\item a unique infinite forwards path $\gamma_z$ in $\eta$ which is a geodesic;
\item the geodesics $\gamma_y$ and $\gamma_z$ coalesce;
\item $\dir(\gamma_z)\subseteq\arc(\rho)=\{x\in S^1:\mu(x)=\rho(x)\}$;
\item the Busemann function of $\gamma_z$ is asymptotically linear to $\rho$.
\end{enumerate}
\end{theorem}

Parts~\emph{(a)--(c)} in the above theorem is essentially Theorem~1.11 in~\cite{damhan14}. Part~\emph{(d)} is not explicitly stated in~\cite{damhan14}, but is a consequence of Theorem~4.3, Corollary~4.7 and Proposition~5.1.
We shall next construct a set of (at least) four random coalescing geodesics based on Theorem~\ref{thm:DH}.

\subsection{From geodesic measures to random coalescing geodesics}

The existence of random coalescing geodesics is certainly hinted at from the geodesic measures considered by Damron and Hanson, but to construct them remains a non-trivial task even from their work. The goal of this section is to prove their existence.

\begin{theorem} \label{yahoo}
Let $\rho:\R^2\to\R^2$ be a linear functional tangent to the asymptotic shape $\ball$. Then, there exists a random coalescing geodesic geodesic $G$ with Busemann function asymptotically linear to $\rho$ and $\dir(G)\subseteq\arc(\rho)$.
\end{theorem}

Since the asymptotic shape has at least four sides, Theorem~\ref{yahoo} proves the existence of at least four random coalescing geodesics.
These four random coalescing geodesics are related to one another via rotations of the plane by right angles.

Recall that $\tree_0=\tree_0(\omega)$ denotes the set of one-sided geodesics starting at the origin, or more formally, the random element $\tree_0:\Omega_1\to\Omega_2$ encoding this set. Given a linear functional $\rho:\R^2\to\R$ tangent to $\ball$, let $\tree_0^\rho=\tree_0^\rho(\omega)$ denote the set of geodesics in $\tree_0$ whose set of directions has non-empty intersection with $\arc(\rho)$. Hence, also $\tree_0^\rho$ can be thought of as a measurable map from $\Omega_1$ to $\Omega_2$. We similarly define $\tree_v^\rho$ as a subset of $\tree_v$.

\begin{lemma}
For any linear functional $\rho$ tangent to $\ball$ the set $\tree_0^\rho$ is almost surely non-empty and totally ordered.
\end{lemma}

\begin{proof}
The fact that $\tree_0^\rho$ is non-empty follows from Theorem~\ref{thm:DH}. By the same theorem we also have that there exists a geodesic in $\tree_0\setminus\tree_0^\rho$. Due to the tree structure of $\tree_0$ any two geodesics will share at most a finite number of edges, after which they diverge, never to intersect again. Hence, for any two geodesics $g$ and $g'$ in $\tree_0^\rho$, one is attained via a counterclockwise motion from the other. Since both $\tree_0^\rho$ and $\tree_0\setminus\tree_0^\rho$ are non-empty, precisely one of the regions counterclockwise between $g$ and $g'$, or between $g'$ and $g$, will contain a geodesic of the complement. We say $g\le g'$ if we can move counterclockwise from $g$ to $g'$ staying in $\tree_0^\rho$. It follows that for any $g$ and $g'$ we have either $g \leq g'$ or $g' \leq g$. Finally, if both $g\leq g'$ and $g' \leq g$ then $g=g'$, so this is a total ordering.
\end{proof}

In order to construct a random coalescing geodesic we will use the Damron-Hanson geodesic measures to put a measure on the totally ordered set $\tree_0^\rho$. More precisely, given a linear functional $\rho:\R^2\to\R$ tangent to $\ball$ and the Damron-Hanson geodesic measure $\nu$ associated to $\rho$ we obtain, for $\P$-almost every $\omega\in\Omega_1$, a probability measure $\hat\nu=\hat\nu(\omega)$ on $\Omega_2$ through projection and conditional expectation. Since $\nu$ is supported on (coalescing) families of geodesics directed in $\arc(\rho)$, the same holds for $\hat\nu$ almost surely. We will interchangeably think of $\hat\nu$ as a measure on $\Omega_2$ and as a measure on $\tree_0^\rho$.

We will next exhibit a function $\psi:\tree_0^\rho \to [0,1] $  such that Lesbesgue measure on $[0,1]$, which we write as $\leb$, is the pushforward  of $\hat \nu$ by $\psi$.
For every $g \in \tree_0^\rho$ we define $\tree_0^{<g}\subseteq \tree_0^\rho$ as the subtree consisting of all geodesics $g'$ in $\tree_0^\rho$ such that $g' < g$. We similarly define $\tree_0^{\le g}$ with $<$ replaced by $\le$. For $\beta \in [0,1]$, let
$$
G_{\beta}:=\sup\big\{g \in \tree_0^\rho: \hat \nu (\tree_0^{<g})<\beta\big\}.
$$
We will see that $G_\beta:\Omega_1\to\Omega_2$ is a random coalescing geodesic for Lebesgue almost every~$\beta$.

\begin{lemma}\label{google}
The map $G_\beta:\Omega_1\times[0,1]\to\Omega_2$ is measurable.
\end{lemma}

\begin{proof}
We show that $(\omega,\beta)\mapsto\{g \in \tree_0^\rho(\omega): \hat \nu (\tree_0^{<g})(\omega)<\beta\}$ defines a measurable map $\varphi:\Omega_1\times[0,1]\to\Omega_2$. Due to the product structure of $\Omega_2$ it will suffice to show that $\{e\in\varphi\}$ is a measurable event. Let $h_e$ denote the least (that is clockwise-most) geodesic in $\tree_0^\rho$ that goes through the edge $e$, should such a geodesic exist. Then
$$
\{e\in\varphi\}=\{e\in g\text{ for some } g\in\tree_0^\rho\}\cap\{\hat\nu(\tree_0^{<h_e})<\beta\}.
$$
The former of these events is measurable since $\tree_0^\rho:\Omega_1\to\Omega_2$ is measurable. The latter can be written as the countable union
$$
\bigcup_{q\in[0,1]\cap\mathbb{Q}}\{\omega\in\Omega_1:\hat\nu(\tree_0^{<h_e})<q\}\cap\{\beta\in[0,1]:q<\beta\}
$$
of measurable events, and is thus measurable.
\end{proof}

\begin{lemma} 
For $\P$-almost every $\omega\in\Omega_1$ we have for all $g \in \tree_0^\rho(\omega)$ that
\begin{equation} \label{rsvp}
\big\{\beta\in[0,1]: G_\beta \in \tree_0^{\le g}\big\}=\big[0,\hat \nu(\tree_0^{\le g})\big];
\end{equation}
and for any $v\in\Z^2$ and any coalescing pair of geodesics $g\in\tree_0^\rho$ and $g'\in\tree_v^\rho$ that
\begin{equation}\label{rsvp2}
\big\{\beta\in[0,1]:G_\beta\in\tree_0^{\le g}\big\}=\big\{\beta\in[0,1]:G_\beta(v)\in\tree_v^{\le g'}\big\}.
\end{equation}
\end{lemma}

\begin{proof}
We first observe that $\{\beta: G_\beta \in \tree_0^{\le g}\}=\{\beta:G_\beta\le g\}$ by definition of $\tree_0^{\le g}$. We first set out to prove~\eqref{rsvp} and need for this to establish two claims.

We first claim that if $\beta>\hat \nu(\tree_0^{\leq g})$, then $G_\beta >g$. 
To see this, consider the following two cases: Either there is a least element $g'\in\tree_0^\rho$ strictly larger than $g$, or there is a decreasing sequence $(g_k)_{k\ge1}$ approaching $g$. In the former case we have
$$
\hat \nu(\tree_0^{<g'})=\hat \nu(\tree_0^{\leq g})<\beta,
$$
which implies that $G_\beta\ge g'>g$. In the latter case, by continuity of measure, we have that
$$
\hat\nu(\tree_0^{<g_k})\searrow\hat\nu(\tree_0^{\le g})<\beta.
$$
Hence, for some $g_k>g$ we have $\hat\nu(\tree_0^{<g_k})<\beta$, and thus that $G_\beta\ge g_k>g$. This settles the first claim.

Second, we claim that if $\beta \le \hat \nu(\tree_0^{\le g})$, then $G_\beta \le g$. To see this, take $g'>g$. Then
$$
\hat\nu(\tree_0^{<g'})\ge\hat\nu(\tree_0^{\le g})\ge\beta.
$$
That is, no $g'>g$ is contained in the set $\{h\in\tree_0^\rho:\hat\nu(\tree_0^{<h})<\beta\}$, and so $G_\beta\le g$.

The two claims imply~\eqref{rsvp}. Note that it holds for $\P$-almost every $\omega\in\Omega_1$ because $\hat \nu$ is a conditional expectation and is only defined almost surely.

We now turn to~\eqref{rsvp2}. First observe that
$$
\big\{\beta:G_\beta(v)\in\tree_v^{\le g'}(\omega)\big\}=\big\{\beta:G_\beta(\sigma_v\omega)\in\tree_0^{\le\sigma_vg'}(\sigma_v\omega)\big\},
$$
which by~\eqref{rsvp} equals $[0,\hat\nu(\tree_0^{\sigma_vg'})(\sigma_v\omega)]$ almost surely. Since $\nu$ is translation invariant, $\hat\nu(\tree_0^{\sigma_vg'})(\sigma_v\omega)=\hat\nu(\tree_v^{\le g'})(\omega)$ almost surely. To prove~\eqref{rsvp2} it will therefore suffice to show that $\hat\nu(\tree_0^{\le g})=\hat\nu(\tree_v^{\le g'})$ for almost every $\omega$ and all coalescing pairs $g$ and $g'$.

Assume, for a contradiction, that $\hat\nu(\tree_0^{\le g})<\hat\nu(\tree_0^{\le g'})$ for some pair $g$ and $g'$ with positive probability. In this case, due to the total ordering, we must have
$$
\hat\nu\big(\tree_v^{\le g'}\cap[\tree_0^\rho\setminus\tree_0^{\le g}]\big)>0
$$
with positive probability. In this case $\nu$ puts positive mass on a non-coalescing family of geodesics. However, by Theorem~\ref{thm:DH} we know that this can only happen on a null set.
\end{proof}

\begin{lemma} \label{bibi}
For Lebesgue-almost every $\beta\in[0,1]$ the geodesic $G_\beta$ is coalescing, has Busemann function linear to $\rho$ and has its directions contained in $\arc(\rho)$, almost surely.
\end{lemma}

\begin{proof}
We start by putting two measures on $\Omega_1\times\Omega_2$ and show that they are the same. The first measure is the projection of $\nu$. The second is the pushforward of $\P\times\leb$ via the map $\Psi:\Omega_1\times[0,1]\to\Omega_1\times\Omega_2$ given by $(\omega,\beta)\mapsto (\omega,G_\beta(\omega))$.

Observe first that by taking Lebesgue measure on both sides in~\eqref{rsvp} we obtain
$$
\leb(\{\beta:G_\beta\in\tree_0^{\le g}\})=\hat\nu(\tree_0^{\le g})\quad\text{for $\P$-almost every }\omega\in\Omega_1.
$$
As $\tree_0^\rho$ is totally ordered, any measure on $\tree_0^\rho$ is defined by its value for sets of this form. Hence, for any Borel set $B\subseteq\Omega_1\times\Omega_2$, using Fubini's theorem, we have
\begin{equation*}
\nu(B)\,=\,\int_{\Omega_1}\hat\nu(B(\omega))\,d\P\,=\,\int_{\Omega_1}\leb(\{\beta:G_\beta\in B(\omega)\})\, d\P\,=\,(\P\times\leb)(\Psi^{-1}B).
\end{equation*}
In particular, for any event $B$ that $\nu$ assigns full measure, $\P(G_\beta\in B)=1$ for Lebesgue-almost every $\beta\in[0,1]$. So, outside of a null set, $G_\beta$ is a geodesic with Busemann function asymptotically linear to $\rho$ and $\dir(G)\subseteq\arc(\rho)$, almost surely.

It remains to show that $G_\beta$ is coalescing for almost every $\beta$. Assume that $G_\beta(0)$ and $G_\beta(v)$ do not coalesce with positive probability. As $G_\beta$ is supported on geodesics in coalescing families of geodesics (since $\nu$ is) there would then, with positive probability, exist a pair of coalescing geodesics $g\in\tree_0^\rho$ and $g'\in\tree_v^\rho$ such that either $G_\beta\in\tree_0^{\le g}$ but $G_\beta(v)\not\in\tree_v^{\le g'}$, or $G_\beta\not\in\tree_0^{\le g}$ while $G_\beta(v)\in\tree_v^{\le g'}$.
This would contradict~\eqref{rsvp2}, and thus only happen on a null set.
\end{proof}

Theorem~\ref{yahoo} is an immediate consequence of Lemmas~\ref{google} and~\ref{bibi}.

\section{A shift invariant labeling of geodesics}
\label{cola}

In this section we define a flow on the tree $\tree_0=\tree_0(\omega)$ of one-sided geodesics emanating from the origin, and use the resulting flow to label geodesics in a systematic way. For the labeling to be useful it will have to be consistent with some natural notion of order among geodesics. Loosely speaking, we will work with an ordering in which $g\le g'$ if to reach $g'$ from some reference geodesic $g_\star$, in a counterclockwise motion, we first cross $g$.

Since the asymptotic shape has at least four sides, Theorem~\ref{yahoo} grants the existence of at least four random coalescing geodesics. Let $\Gamma_0$ denote one of these. From $\Gamma_0$ we obtain an additional three distinct random coalescing geodesics via right-angle rotation. This again gives a set of four random coalescing geodesics, whose asymptotic properties are related via right-angle rotation. For the rest of this paper we will denote by $\Gamma_0$, $\Gamma_1$, $\Gamma_2$ and $\Gamma_3$ the four random coalescing geodesics obtained in this fashion by right angle rotation. We will further fix one of these four geodesics as our reference geodesic; call this geodesic $\Gamma_\star$. As before, we denote by $\Gamma_i(v)$ the translate of $\Gamma_i$ along the vector $v\in\Z^2$.

\subsection{A total ordering of geodesics}

Recall that $\tree_v$ denotes the tree of one-sided geodesics emanating from the vertex $v\in\Z^2$. There is a natural total ordering among geodesics in $\tree_v$, where $g\le g'$ if we can reach $g'$ from $g$ in a counterclockwise motion without crossing $\Gamma_\star(v)$. Since $\Gamma_\star$ is coalescing, the total ordering on $\tree_v$ extends to a total ordering among all geodesics in $\{g\in\tree_v:v\in\Z^2\}$: Any two geodesics $g$ and $g'$ that intersect either coalesce or intersect in a finite connected set of edges. Given $g\in\tree_u$ and $g'\in\tree_v$ let $S\subset\Z^2$ be finite and connected with the property that $\Gamma_\star(u)$ and $\Gamma_\star(v)$ agree outside $S$ and $g$ and $g'$ either agree or are disjoint outside $S$. Indeed, if some set $S$ has this property, then every set $S'$ containing $S$ has this property too. We say that $g\le g'$ if we can reach $g'$ from $g$ in a counterclockwise motion along the boundary of $S$ without crossing $\Gamma_\star$.
This gives a well-defined total ordering with probability one. Note that if $g$ and $g'$ coalesce, then they are considered equal in the above ordering, and we write $g<g'$ if $g\le g'$ but not $g'\le g$. It is natural to think of the ordering as cyclic, in which $\Gamma_\star$ is not only the minimal, but also the maximal element of the ordering.

Based on the above ordering, we shall say that a geodesic $g\in\tree_v$ is {\bf cw-dense} if there exists a sequence of geodesics $g_1<g_2<g_3 < \cdots$ in $\tree_v$ such that $\lim_{k\to\infty}g_k=g$. Otherwise we say that $g$ is {\bf cw-isolated}. The terms {\bf ccw-dense} and {\bf ccw-isolated} are defined analogously. We note that if a geodesic $g=(v_0,v_1,\dots)$ is cw-dense, then so is the geodesic $g^{n}=(v_n,v_{n+1},\dots)$. Two geodesics $g<g'$ in $\tree_v$ are called {\bf neighbors} if every $g^*\in\tree_v$ satisfies either $g^*\le g$ or $g^*\ge g'$.

\subsection{A single source flow}

In a first step, we define, for every vertex $v\in\Z^2$, a flow from $v$ to $\infty$ along $\tree_v$. This flow has a source of magnitude one at $v$ and no sinks. We define the flow inductively starting at the root. Suppose we have defined the flow into a vertex $w$. The flow splits the mass flowing into $w$ equally among all edges in $\tree_v$ that emanate from $w$. We denote by $M_v(\Gamma_\star,g]$ the mass that flows out along geodesics in $\tree_v$ that lie strictly above $\Gamma_\star$ and below and including $g$, in the counterclockwise ordering. This assigns to each geodesic in $\tree_v$ a value between $0$ and $1$. We make the convention to interpret $M_v(\Gamma_\star,\Gamma_\star]$ as 0, but in consistence with the cyclic ordering, where $\Gamma_\star$ is considered both minimal and maximal, we identify the values 0 and 1.

The cumulative flow $M_v(\Gamma_\star,g]$ will not (necessarily) provide a labeling of geodesics consistent with the ordering. That is, even for $g\in\tree_u$ and $g'\in\tree_v$ that coalesce, and hence are equal as far as the ordering concerns, we may have $M_u(\Gamma_\star,g]\neq M_v(\Gamma_\star,g']$. To obtain a labeling consistent with the ordering we will employ an averaging procedure over equivalence classes of $\Z^2$ and work with subsequential limits.

\subsection{Averaging over equivalence classes}

Let $\{\xi_z\}_{z\in\Z^2}$ be independent $[0,1]$-uniform random variables. For each $i=1,2,\ldots$ let $\{V_i(z):z\in\Z^2\}$ be the partition of $\Z^2$ obtained as follows: Let $S_i$ denote the set of all points in $\Z^2$ with $\xi_z\le1/4^i$ and define $f_i:\Z^2\to S_i$ by mapping each point to the one in $S_i$ at least $\ell_1$-distance. (Choose, say, the one with minimal $\xi$-value in case of a tie.) This induces an equivalence relation on $\Z^2$ in which two sites are equivalent in case they map to the same site in $S_i$. Let $\{V_i(z):z\in\Z^2\}$ be the collection of equivalence classes of this equivalence relation.

\begin{lemma}\label{lma:classes}
For every pair $u,v\in\Z^2$ we have
$$\P\big(V_i(u)=V_i(v)\text{ for all $i$ sufficiently large}\big)=1.$$
\end{lemma}

\begin{proof}
In case $u$ and $v$ belong to different equivalence classes, then there exists $z\in S_i\setminus\{f_i(u)\}$ such that $\|z-u\|\le\|f_i(u)-u\|+2\|u-v\|$. To see this assume the contrary, in which case the triangle inequality, for any $z\in S_i\setminus\{f_i(u)\}$, gives that
$$
\|z-v\|\ge\|z-u\|-\|u-v\|>\|f_i(u)-u\|+\|u-v\|\ge\|f_i(u)-v\|,
$$
and hence $V_i(u)=V_i(v)$.

For each $i\ge1$ the expected distance from $u$ to $S_i$ is of order $2^i$, so with high probability we have $\|f_i(u)-u\|<3^i$. Since $\|y-u\|\ge\|f_i(u)-u\|$ for all $y\in S_i$, there are order $3^i\|u-v\|$ possible choices for the point $z$. Since each has probability $1/4^i$ to belong to $S_i$ we conclude that $V_i(u)\neq V_i(v)$ with probability of order $(3/4)^i\|u-v\|$. The result then follows from Borel-Cantelli.
\end{proof}

We next use the equivalence classes above defined to obtain a labeling of geodesics which is consistent with the ordering among sites within the same equivalence class. Based on the total ordering we may define $M_u(\Gamma_\star,g]$ for any geodesic $g$, not necessarily in $\tree_u$, as the mass (under the flow from $u$) along all geodesics between $\Gamma_\star$ and $g$. That is, let
$$
M_u(\Gamma_\star,g]:=\sup\big\{M_u(\Gamma_\star,g']:g'\in\tree_u\text{ and }g'\le g\big\}.
$$
For each $i=1,2,\ldots$ and $v\in\Z^2$ we define for $g\in\tree_v$
$$
M_v^i(\Gamma_\star,g]:=\frac{1}{|V_i(v)|}\sum_{u\in V_i(v)}M_u(\Gamma_\star,g].
$$
It is straightforward to verify that the (pre-)labeling generated by the averaged cumulative flow $M_v^i$ is consistent with the ordering of geodesics originating from the points in the same equivalence class; we save the details for the proof of Lemma~\ref{lma:label_measure} below.

\subsection{Subsequential weak limits}

The labels produced by $M_v^i:\tree_v\to[0,1]$ can be encoded as an element in $[0,1]^{\Ec^2}$ in the following manner, where $\Ec^2$ denotes the set of (undirected) edges of the square lattice. For each $v\in\Z^2$ and $e\in\Ec^2$ define
$$
\varphi_i(v,e):=\sup\big\{M_v^i(\Gamma_\star,g]:g\in\tree_v\text{ and }e\in g\big\}.
$$
(Supremum of the empty set is interpreted as zero.) This defines, for each $v\in\Z^2$ and $i\ge1$, an element $\varphi_i(v,\cdot)\in[0,1]^{\Ec^2}$. Note further that if $g=(e_1,e_2,\ldots)$ is a ccw-isolated geodesic in $\tree_v$, then we can recover the value of $M_v^i(\Gamma_\star,g]$ from $\varphi_i$ as the limit
\begin{equation}\label{eq:encoding_id}
M_v^i(\Gamma_\star,g]=\lim_{n\to\infty}\varphi_i(v,e_n)=\inf_{e\in g}\varphi_i(v,e),
\end{equation}
as $\varphi_i(v,e_n)$ is decreasing (more precisely, non-increasing) in $n$.

Let $\Omega_1=[0,\infty)^{\Ec^2}$, $\Omega_4=[0,1]^{\Z^2}$ and $\Omega_5=[0,1]^{\Z^2\times\Ec^2}$. For each $i\ge1$ we can exhibit a (measurable) map $\Psi_i:\Omega_1\times\Omega_4\to\Omega_1\times\Omega_5$ as $(\omega,\xi)\mapsto(\omega,\varphi_i)$. The measure $\P\times\leb$ may be pushed forward through the mapping $\Psi_i$ to give a measure $\nu_i$ on $\Omega_1\times\Omega_5$. Via a compactness argument and Prokhorov's theorem we conclude that $(\nu_i)_{i\ge1}$ has a weakly converging subsequence. The next couple of lemmas show that the limit of the converging subsequence is well behaved, and consistent with the total ordering of geodesics.

Given $v\in\Z^2$, let $\tilde\sigma_v$ denote the shift operator on $\Omega_1\times\Omega_5$ for which
$$
[\tilde\sigma_v(\omega,\varphi)](e,(z,f))=(\omega_{e-v},\varphi(z-v,f-v)).
$$

\begin{lemma}
Every subsequential limit $\nu$ of $(\nu_i)_{i\ge1}$ is invariant with respect to $\tilde\sigma_v$.
\end{lemma}

\begin{proof}
Let $\nu$ be a subsequential limit of $(\nu_i)_{i\ge1}$. We first show that $\int f\,d\nu_i=\int f\,d\nu_i\circ\tilde\sigma_v$ for all bounded continuous functions $f:\Omega_1\times\Omega_5\to\R$, and thus that $\nu_i=\nu_i\circ\tilde\sigma_v$. Indeed, this is a straightforward consequence of the product structure of $\P\times\leb$. More precisely, if $\tilde{\tilde\sigma}_v$ denotes the operator on $\Omega_1\times\Omega_4$ for which $[\tilde{\tilde\sigma}_v(\omega,\xi)](e,z)=(\omega_{e-v},\xi_{z-v})$, then
$$
\int f\,d\nu_i\circ\tilde\sigma_v=\int f\circ\tilde\sigma_v\circ\Psi_i\,d(\P\times\leb)=\int f\circ\Psi_i\circ\tilde{\tilde\sigma}_v\,d(\P\times\leb)=\int f\,d\nu_i
$$
for each bounded continuous function $f$, since $\tilde\sigma_v\circ\Psi_i=\Psi_i\circ\tilde{\tilde\sigma}_v$ and $\P\times\leb$ is invariant with respect to $\tilde{\tilde\sigma}_v$. Hence, $\nu_i=\nu_i\circ\tilde\sigma_v$ for every $i\ge1$, and by continuity of $\tilde\sigma_v$ it follows that $\nu=\nu\circ\tilde\sigma_v$ by taking limits.
\end{proof}

\begin{lemma}\label{lma:label_measure}
Every subsequential limit $\nu$ of $(\nu_i)_{i\ge1}$ has the property that for $\nu$-almost every $(\omega,\varphi)\in\Omega_1\times \Omega_5$ and every $u,v\in\Z^2$ we have that
\begin{enumerate}[\quad (a)]
\item $\varphi(u,e_n)$ is decreasing for every geodesic $g=(e_1,e_2,\ldots)$ in $\tree_u$;
\item for any two ccw-isolated  geodesics $g=(e_1,e_2,\ldots)$ in $\tree_u$ and $g'=(e_1',e_2',\ldots)$ in $\tree_v$ with $g\le g'$   we have
$$
\lim_{n\to\infty}\varphi(u,e_n)\le\lim_{n\to\infty}\varphi(v,e_n').
$$
\end{enumerate}
\end{lemma}

\begin{proof}
Since $\Z^2$ is countable it will suffice to prove each of the statements for a fixed pair of vertices $u,v\in\Z^2$. We start with part~\emph{(a)}, and note that it will further suffice to show that $\varphi(u,e)$ is decreasing along the edges of $\geo(u,z)$ for every $z\in\Z^2$.
Let $\gamma=(e_1,e_2,\ldots,e_m)$ be a finite path between $u$ and $z$. Denote by $A_\gamma$ the event that $\gamma$ is a geodesic, and by $B_\gamma$ the event that $\varphi(u,e_n)$ is decreasing. By construction we have $\nu_i((A_\gamma\cap B_\gamma)\cup A_\gamma^c)=1$ and $\nu_i(A_\gamma^c)=\nu(A_\gamma^c)$. Both $A_\gamma$ and $B_\gamma$ are closed events, so the Portmanteau theorem\footnote{The Portmanteau theorem says that $\nu_i\to\nu$ weakly iff $\displaystyle\limsup_{i\to\infty}\nu_i(C)\le\nu(C)$ for any closed event $C$.} gives that
$$
1=\nu_i(A_\gamma\cap B_\gamma)+\nu_i(A_\gamma^c)\le\nu(A_\gamma\cap B_\gamma)+\nu(A_\gamma^c)=\nu\big((A_\gamma\cap B_\gamma)\cup A_\gamma^c\big).
$$
Since the number of finite paths between $u$ and $z$ is countable, this proves part~\emph{(a)}.

We proceed with part~\emph{(b)}, and let $A_{u,v}$ denote the event that
$$
\lim_{n\to\infty}\varphi(u,e_n)\le\lim_{n\to\infty}\varphi(v,e_n')
$$
for every pair of ccw-isolated $g=(e_1,e_2,\ldots)$ in $\tree_u$ and $g'=(e_1',e_2',\ldots)$ in $\tree_v$ such that $g\le g'$. Let $B_{u,v}$ be the event that $\Gamma_\star(u)$ and $\Gamma_\star(v)$ coalesce and $V_i(u)=V_i(v)$. According to the averaging procedure over equivalence classes we have on the event $B_{u,v}$ that
$$
M_u^i(\Gamma_\star,g]=\frac{1}{|V_i(u)|}\sum_{w\in V_i(u)}M_w(\Gamma_\star(u),g]\le\frac{1}{|V_i(u)|}\sum_{w\in V_i(u)}M_w(\Gamma_\star(u),g']=M_v^i(\Gamma_\star,g'].
$$
It follows from the identity in~\eqref{eq:encoding_id} that
$$
\nu_i(A_{u,v})\ge(\P\times\leb)(\{(\omega,\xi):V_i(u)=V_i(v)\}),
$$
which by Lemma~\ref{lma:classes} tends to 1 as $i\to\infty$.

Let $\tree_u^n$ denote the set of ccw-isolated geodesics in $\tree_u$ that diverges from its ccw-neighbor within $n$ steps. Informally speaking, this is the set of geodesics that can be distinguished by observing the first $n$ steps of each geodesic in $\tree_u$. Let $E_{u}(n,m)$  be the set of edges with the property that they are the first edge in some $g\in\tree_u^n$ that connects the box $[-m,m]^2$ to its complement. Define a counterclockwise order on edges that connect $[-m,m]^2$ to its complement by choosing the least element to be the edge in $E_0(n,m)$ that belongs to $\Gamma_\star(0)$.

Let $A'_{u,v}(n,m)$ be the event that for all $\ell\ge m$ and all $e\in E_u(n,\ell)$ and $e'\in E_v(n,\ell)$ such that $e\le e'$ we have
$$
\varphi(u,e)\le\varphi(v,e').
$$
Moreover, let $C_{u,v}(n,m)$ be the event that $\Gamma_\star(u)$ and $\Gamma_\star(v)$ coincide outside $[-m,m]^2$ and that the geodesics of $\tree_u^n$ and $\tree_v^n$ leave $[-m,m]^2$ in their respective order. On $C_{u,v}(n,m)$ we have $A_{u,v}\subseteq A_{u,v}'(n,m)$. Hence, for each $\eps>0$ we have for all large $i$ and $m$ that
$$
\nu_i(A'_{u,v}(n,m))>1-\eps.
$$

We now argue that $A'_{u,v}(n,m)$ is closed. Let $\gamma$ and $\gamma'$ be two ordered sequences of edges connecting $[-\ell,\ell]^2$ to its complement. Let $D_{\gamma,\gamma'}(n,\ell)$ be the event that $E_u(n,\ell)=\gamma$ and $E_v(n,\ell)=\gamma'$. The event $D_{\gamma,\gamma'}(n,\ell)$ is closed, and $A'_{u,v}(n,m)$ can be written as
$$
A'_{u,v}(n,m)=\bigcap_{\ell\ge m}\bigcup_{\gamma,\gamma'}\{\varphi(u,\gamma_i)\le\varphi(v,\gamma'_j)\text{ for }\gamma_i\le\gamma'_j\}\cap D_{\gamma,\gamma'}(n,\ell).
$$
Since finite unions and arbitrary intersections of closed sets is closed, $A'_{u,v}(n,m)$ is closed.

The Portmanteau theorem now implies that for every $\eps>0$ there is $m$ so that
$$
\nu(A'_{u,v}(n,m))>1-\eps.
$$
By continuity of measure it follows that $A'_{u,v}(n,m)$ occurs for some $m\ge1$ with $\nu$-probability one. This concludes part~\emph{(b)} for ccw-isolated geodesics that diverge from their ccw-neighbor within the first $n$ steps. Since any ccw-isolated geodesic is of this form, for some $n$, the result follows.
\end{proof}

\subsection{Global labeling of geodesics}

Finally, given a subsequential limit $\nu$ of $(\nu_i)_{i\ge1}$ we give each geodesic in the plane a label $\alpha\in[0,1]$ through the reconstructed cumulative flow obtained through $\nu$. For each $\omega\in\Omega_1$ we obtain a probability measure $\hat\nu=\hat\nu(\omega)$ on $\Omega_5$ through conditional expectation. For each $\omega\in\Omega_1$ and $v\in\Z^2$, define for each ccw-isolated geodesic $g\in\tree_v(\omega)$ a label through averaging:
\begin{equation}\label{eq:Fdef}
F(g):=\inf\Big\{\lim_{n\to\infty}\int\varphi(v,e_n)\,d\hat\nu(\omega):g'=(e_1,e_2,\ldots)\text{ is ccw-isolated and }g'\ge g\Big\},
\end{equation}
where the infimum is taken over all ccw-isolated geodesics $g'\ge g$ starting anywhere in the plane. Before proceeding, we record a few simple observations regarding the labeling. The first implying that it is well-defined.


\begin{lemma}\label{label pre-props}
For $\P$-almost every $\omega\in\Omega_1$ and every geodesic $g=(e_1,e_2,\ldots)$ in $\tree_v$ the limit
$$
\Phi(g):=\lim_{n\to\infty}\int\varphi(v,e_n)\,d\hat\nu(\omega)
$$
exists and satisfies $F(g)\le\Phi(g)$ in general, with equality whenever $g$ is ccw-isolated.
\end{lemma}

\begin{proof}
Existence of the limit is immediate from part~\emph{(a)} of Lemma~\ref{lma:label_measure} and monotone convergence. For the inequality it will suffice to consider $g$ such that $\Gamma_0\le g<\Gamma_1$. Let $g_n$ denote the counterclockwise-most geodesic in $\tree_v$ containing $e_n$. Then $g_n$ is ccw-isolated and for large $n$ we have $\Gamma_1>g_n\ge g$, so~\eqref{eq:Fdef} and part~\emph{(a)} of Lemma~\ref{lma:label_measure} give
$$
F(g)\le\Phi(g_n)\le\int\varphi(v,e_n)\,d\hat\nu(\omega).
$$
Taking limits yields $F(g)\le\Phi(g)$, whereas equality, for $g$ ccw-isolated, holds since for any ccw-isolated geodesic $g'\ge g$ part~\emph{(b)} of Lemma~\ref{lma:label_measure} gives $\Phi(g')\ge\Phi(g)$.
\end{proof}

We next observe, crucially, that the labeling is consistent with the ordering of geodesics.

\begin{prop}\label{label props}
For $\P$-almost every $\omega\in\Omega_1$ we have for all $g\in\tree_u$ and $g'\in\tree_v$ that
\begin{enumerate}[\quad (a)]
\item if $g\le g'$, then $F(g)\le F(g')$;
\item if $g$ and $g'$ coalesce, then $F(g)=F(g')$.
\end{enumerate}
\end{prop}

\begin{proof}

Part~\emph{(a)} is immediate from~\eqref{eq:Fdef}, since if $g\le g'$ then $F(g)$ is the infimum over a larger set that $F(g')$. Part~\emph{(b)} follows from~\emph{(a)} since if $g$ and $g'$ coalesce, then we have both $g\le g'$ and $g'\le g$.
\end{proof}

\begin{lemma}\label{label post-props}
For $\P$-almost every $\omega\in\Omega_1$ and $v\in\Z^2$ we may for every $g\in\tree_v(\omega)$ such that
\begin{equation}\label{eq:strict}
F(g)<\inf\big\{F(g'):g\in\tree_v\text{ is ccw-isolated and }g'\ge g\big\}
\end{equation}
find $u\in\Z^2$ and a ccw-isolated geodesic $g^\ast\in\tree_u$ such that for any decreasing sequence $(g_k)_{k\ge1}$ in $\tree_v$ of ccw-isolated geodesics converging to $g$ we have $g\le g^\ast<g_k$ for all $k$.
\end{lemma}

\begin{proof}
Suppose the contrary, that with positive probability there exists $g\in\tree_v$ satisfying~\eqref{eq:strict}, and that for every ccw-isolated geodesic $g^\ast\ge g$ there exists a ccw-isolated geodesic $g'\in\tree_v$ such that $g\le g'\le g^\ast$. By Lemmas~\ref{lma:label_measure}\emph{(b)} and~\ref{label pre-props} it follows that with positive probability there exists $g\in\tree_v$ satisfying~\eqref{eq:strict}, and that for every ccw-isolated geodesic $g^\ast\ge g$ there exists a ccw-isolated geodesic $g'\in\tree_v$ such that $\Phi(g^\ast)\ge\Phi(g')=F(g')$, contradicting~\eqref{eq:Fdef}.
\end{proof}

\subsection{Labels of random coalescing geodesics}

Formally we may think of the labeling as a measurable map $F:\Omega_1\to\Omega_5$. The labeling is invariant with respect to translations since $\nu$ is. Together with the coalescence property it follows that for any random coalescing geodesic $G$ and any finite set $V\subseteq\Z^2$ we have
$$
F(G)(\omega)\,=\,\frac{1}{|V|}\sum_{v\in V}F(G(v))(\omega)\,=\,\frac{1}{|V|}\sum_{v\in V}F(G)(\sigma_v\omega).
$$
By the ergodic theorem we then have $F(G)=\E[F(G)]$ almost surely. That is, every random coalescing geodesic has an almost surely constant label. This is true in a strong sense, as we explore in the next couple of lemmas.

\begin{lemma}\label{lma:rcg_label1}
Let $G$ be a random coalescing geodesic. For $\nu$-almost every $(\omega,\varphi)\in\Omega_1\times\Omega_5$, if $(e_1,e_2,\ldots)$ is an enumeration of the edges in $G(\omega)$, then we have
$$
\lim_{n\to\infty}\varphi(0,e_n)=\E\big[M_0(\Gamma_\star,G]\big].
$$
\end{lemma}

\begin{proof}
Let $(e_1,e_2,\ldots)$ is an enumeration of the edges in $G$, and let $(g_k)_{k\ge1}$ be any decreasing sequence of ccw-isolated geodesics in $\tree_0$ converging to $G$. Then, almost surely,
\begin{equation}\label{eq:alpro}
\lim_{n\to\infty}\varphi_i(0,e_n)=\lim_{k\to\infty}M_0^i(\Gamma_\star,g_k]=M_0^i(\Gamma_\star,G],
\end{equation}
where the former equality follows from the definition of $\varphi_i$ and the latter since $G$ is coalescing. We further note that
$$
M_0^i(\Gamma_\star,G](\omega)=\frac{1}{|V_i(0)|}\sum_{v\in V_i(0)}M_v(\Gamma_\star,G(v)](\omega)=\frac{1}{|V_i(0)|}\sum_{v\in V_i(0)}M_0(\Gamma_\star,G](\sigma_v\omega),
$$
which by the ergodic theorem approaches $\E\big[M_0(\Gamma_\star,G]\big]$ as $i\to\infty$ almost surely.

Let $E_m$ denote the set of edges connecting the set $[-m,m]^2$ to its complement, and let $\hat e_m$ denote the first edge in $E_m$ used by $G(\omega)$. For $n\ge1$ and $\eps>0$ let $A(n,\eps)$ denote the event that $\E\big[M_0(\Gamma_\star,G]\big]-\eps\le\varphi(0,\hat e_m)\le\E\big[M_0(\Gamma_\star,G]\big]+\eps$ for all $m\ge n$. That $\nu_i(A(n,\eps))\ge1-\eps$ for all large $i$ and $n$ is immediate from~\eqref{eq:alpro} and the fact that $M_0^i(\Gamma_\star,G]$ approaches $\E\big[M_0(\Gamma_\star,G]\big]$. The event $A(n,\eps)$ is closed, as it can be expressed as
$$
A(n,\eps)=\bigcap_{m\ge n}\bigcup_{e\in E_m}\Big\{\E\big[M_0(\Gamma_\star,G]\big]-\eps\le\varphi(0,e)\le\E\big[M_0(\Gamma_\star,G]\big]+\eps\Big\}\cap\{\hat e_m=e\},
$$
and since finite unions and arbitrary intersections of closed events are closed. Hence, the Portmanteau theorem implies that for every $\eps>0$ there is $n$ so that $\nu(A(n,\eps))\ge1-\eps$. By continuity of measure, for every $\eps>0$ the event $A(n,\eps)$ occurs for some $n$ with $\nu$-probability one. Since $\varphi(0,e)$ is decreasing along geodesics, by Lemma~\ref{lma:label_measure}, the lemma follows.
\end{proof}

\begin{lemma}\label{lma:rcg_label2}
Let $G$ be a random coalescing geodesic. For $\nu$-almost every $(\omega,\varphi)\in\Omega_1\times\Omega_5$ we have for every ccw-isolated geodesic $g=(e_1,e_2,\ldots)$ in $\tree_0(\omega)$ that
\begin{itemize}
\item $\lim_{n\to\infty}\varphi(0,e_n)\ge\E\big[M_0(\Gamma_\star,G]\big]$, if $g\ge G(\omega)$;
\item $\lim_{n\to\infty}\varphi(0,e_n)\le\E\big[M_0(\Gamma_\star,G]\big]$, if $g\le G(\omega)$.
\end{itemize}
Moreover, $F(G)=\E\big[M_0(\Gamma_\star,G]\big]$ almost surely.
\end{lemma}

\begin{proof}
Let $E_m$ denote the set of edges connecting $[-m,m]^2$ to its complement. Let $\tree_0^n$ denote the set of ccw-isolated geodesics in $\tree_0$ that diverge from their ccw-neighbor within $n$ steps. Let $E_m(n)$ denote the subset of $E_m$ of edges with the property that they are the first edge in $E_m$ used by some $g\in\tree_0^n$. We define a counterclockwise order on $E_m$ by declaring the first edge used by $\Gamma_\star$ as the minimal element. Finally, let $\hat e_m$ denote the first edge in $E_m$ used by $G(\omega)$.

For $\eps>0$, let $A(m,n,\eps)$ be the event that for every $\ell\ge m$ and $e\in E_\ell(n)$ we have
\begin{itemize}
\item $\varphi(0,e)\ge\E\big[M_0(\Gamma_\star,G]\big]-\eps$, if $e\ge\hat e_\ell$;
\item $\varphi(0,e)\le\E\big[M_0(\Gamma_\star,G]\big]+\eps$, if $e\le\hat e_\ell$.
\end{itemize}
First note that, $\P$-almost surely, $\varphi_i(0,e)\ge M_0^i(\Gamma_\star,G]$ for all $e\in E_\ell(n)$ such that $e\ge\hat e_\ell$, and that the reversed inequality holds for $e<\hat e_\ell$. For $e=\hat e_\ell$ it follows by~\eqref{eq:alpro} that $\varphi_i(0,e)\le M_0^i(\Gamma_\star,G]+\eps$ for $\ell$ large. Since $M_0^i(\Gamma_\star,G]$ approaches $\E\big[M_0(\Gamma_\star,G]\big]$ as $i\to\infty$, almost surely, we conclude that for all large $i$ and $m$ we indeed have
$$
\nu_i(A(m,n,\eps))\ge1-\eps.
$$

That $A(m,n,\eps)$ is closed is easily verified in a similar fashion as in the proofs of Lemmas~\ref{lma:label_measure} and~\ref{lma:rcg_label1}. By the Portmanteau theorem we therefore conclude that for all large $m$ we have $\nu(A(m,n,\eps))\ge1-\eps$. By continuity of measure it follows that for every $n\ge1$ and $\eps>0$ the event $A(m,n,\eps)$ occurs for some $m$. The first of the two statements of the lemma then follows by monotonicity of $\varphi(0,e)$ along geodesics in $\tree_0$, i.e.\ Lemma~\ref{lma:label_measure}.

We finally argue that $F(G)=\E\big[M_0(\Gamma_\star,G]\big]$ almost surely. That $F(G)$ is at least $\E\big[M_0(\Gamma_\star,G]\big]$ is immediate from the definition of the labeling and the first part of the lemma, since $G$ is coalescing. That $F(G)$ is at most $\E\big[M_0(\Gamma_\star,G]\big]$ follows from Lemma~\ref{lma:rcg_label1}, as it shows that for every $\eps>0$ there are ccw-isolated geodesics counterclockwise of $G$ (possibly $G$ itself) that have label at most $\E\big[M_0(\Gamma_\star,G]\big]+\eps$ with $\nu$-probability one.
\end{proof}

We remark, on the side, that since $F(\Gamma_i)=\E\big[M_0(\Gamma_\star,\Gamma_i]\big]$ for $i=0,1,2,3$, it follows by symmetry that the four $\Gamma_i$-geodesics have labels $0$, $1/4$, $1/2$ and $3/4$.

\subsection{Multiplicity of labels}

We end this section by showing, in a couple of lemmas, that the labeling does a good job of distinguishing distinct geodesics.

\begin{lemma}\label{lma:nu_mult}
Let $\alpha\in[0,1]$ be fixed. For $\nu$-almost every $(\omega,\varphi)\in\Omega_1\times\Omega_5$ there are no two ccw-isolated geodesics $g=(e_1,e_2,\ldots)$ and $g'=(e_1',e_2',\ldots)$ in $\tree_0(\omega)$ for which
$$
\lim_{n\to\infty}\varphi(0,e_n)=\lim_{n\to\infty}\varphi(0,e_n')=\alpha.
$$
\end{lemma}

\begin{proof}
Below, we shall call a geodesic $g\in\tree_v$ \emph{$k$-good} if it is ccw-isolated and splits from its ccw-neighbor within $k$ steps. Given an open interval $I\subset[0,1]$ and integer $k\ge1$ let $A_i(I,k)$ be the set of vertices $v\in\Z^2$ for which there are two $k$-good geodesics $g$ and $g'$ in $\tree_v$ with $M_v^i(\Gamma_\star,g]$ and $M_v^i(\Gamma_\star,g']$ in $I$.

\begin{claim}
For every $k\ge1$ and $i\ge1$ we have
$$
\P\big(|A_i(I,k)\cap V_i(0)|<|I|4^k|V_i(0)|\text{ for all }I\subseteq[0,1]\big)=1.
$$
\end{claim}

\begin{proof}[Proof of claim]
For every $v\in A_i$ let $a_v$ and $b_v$ denote the clockwise- and counterclockwise-most of the two $k$-good geodesics, and let $a_i$ and $b_i$ denote the least and largest elements among all $k$-good geodesics in $A_i\cap V_i(0)$ (in the total ordering). We then observe that for each $v$ mass of at least $4^{-k}$ escapes along $b_v$. Consequently, we obtain
$$
|I|\,>\,M_v^i(\Gamma_\star,b_i]-M_v^i(\Gamma_\star,a_i]\,>\,4^{-k}\frac{|A_i\cap V_i|}{|V_i|},
$$
or that $|A_i\cap V_i|<|I|4^k|V_i|$.
\end{proof}

From the claim we conclude that for each fixed interval $I$ we have
\begin{equation}\label{to be referred to below}
\P\big(0\in A_i(I,k)\big)<|I|4^k.
\end{equation}
Let $B(I,k)$ denote the event that there are edges $e_k,e_k'$ in $\tree_0(\omega)$ at distance $k$ from the origin for which both $\varphi(0,e_k)$ and $\varphi(0,e_k')$ are in $I$. We note that~\eqref{to be referred to below} implies that
$$
\nu_i(B(I,k))<|I|4^k.
$$
Given a finite set of edges $\gamma$ we let $C_\gamma$ be the event that $\varphi(0,e)\in I$ for at least two $e\in\gamma$, and let $D_\gamma$ be the event that $\gamma$ are precisely the edges of $\tree_0$ at distance $k$ from the origin. The set $C_\gamma$ is open and one can show that $D_\gamma$ is a $\nu$-continuity set. Hence, the Portmanteau theorem gives that
$$
\liminf_{i\to\infty}\nu_i(B(I,k))\,=\,\liminf_{i\to\infty}\sum_\gamma\nu_i(C_\gamma\cap D_\gamma)\,\ge\,\sum_\gamma\nu(C_\gamma\cap D_\gamma)\,=\,\nu(B(I,k)),
$$
and hence that $\nu(B(I,k))\le |I|4^k$. Since $I\subseteq[0,1]$ was arbitrary, then lemma follows.
\end{proof}

As a consequence of the next lemma, if there exists a random coalescing geodesic with label $\alpha$, then there are at most two geodesics in $\tree_0$ with label $\alpha$ almost surely.

\begin{lemma}\label{lma:multiplicity}
Let $G$ be a random coalescing geodesic with label $\alpha$. Then,
$$
\P\big(\exists\text{ two ccw-isolated geodesics in $\tree_0$ with label }\alpha\big)=0.
$$
\end{lemma}

\begin{proof}
Assume, for a contradiction, that with positive probability we may find two ccw-isolated geodesics in $\tree_0$ with label $\alpha$. That is, suppose that with positive probability there are two ccw-isolated geodesics $g=(e_1,e_2,\ldots)$ and $g'=(e_1',e_2',\ldots)$ in $\tree_0$ for which
$$
\lim_{n\to\infty}\int\varphi(0,e_n)\,d\hat\nu(\omega)=\lim_{n\to\infty}\int\varphi(0,e_n')\,d\hat\nu(\omega)=\alpha.
$$
According to Lemma~\ref{lma:rcg_label2}, almost surely, for every ccw-isolated geodesic in $\tree_0(\omega)$ the limit $\lim_{n\to\infty}\varphi(0,e_n)$ is supported on either $[0,\alpha]$ or $[\alpha,1]$ with $\hat\nu(\omega)$-probability one, depending on its relation to $G$. Consequently, if with positive probability we may find two ccw-isolated geodesics in $\tree_0$ with label $\alpha$, then with positive probability we find ccw-isolated geodesics $g$ and $g'$ such that
$$
\hat\nu(\omega)\Big(\lim_{n\to\infty}\varphi(0,e_n)=\lim_{n\to\infty}\varphi(0,e_n')=\alpha\Big)=1.
$$
This contradicts Lemma~\ref{lma:nu_mult}.
\end{proof}

\section{A central geometric argument}
\label{geometric}

In this section we present a central geometric argument. This argument will effectively function as a 0-1 law, and will be used repeatedly for constructing geodesics that starts at some vertex $v$ and have certain desired properties. Recall that a random coalescing geodesic has an almost surely constant label due to the coalescence property. We demonstrate the use of our geometric argument below and show that random non-crossing geodesics, defined next, have constant label.

\begin{define} 
A measurable map $G:\Omega_1\to\Omega_2$ is a {\bf random non-crossing geodesic} if, almost surely, $G(\omega)\in\tree_0(\omega)$ and for all pairs of points $G(u)$ and $G(v)$ are non-crossing.
\end{define}

Recall that we since Section~\ref{cola} have fixed a set of four random coalescing geodesics $\Gamma_i$, for $i=0,1,2,3$, obtained from one another through right-angle rotation. The random non-crossing geodesics $G$ that we shall encounter will almost surely be contained counterclockwise between $\Gamma_{i}$ and $\Gamma_{i+1}$, meaning that $\Gamma_i\le G\le\Gamma_{i+1}$, for some $i$. By relabeling the geodesics $\Gamma_i$ we may assume that $G$ lies counterclockwise between $\Gamma_0$ and $\Gamma_1$ almost surely.

We first illustrate the use of the geometric argument. We start with the cone determined by moving counterclockwise from $\Gamma_{0}(0)$ to $\Gamma_{1}(0).$ Then we find two geodesics $g\in\tree_u$ and $g'\in\tree_v$ such that (see Figure~\ref{fig:geom1})
\begin{itemize}
\item $u \in \Gamma_{0}$ and $g$ is in the cone counterclockwise between $\Gamma_{0}$ and $\Gamma_{1}$;
\item $v \in \Gamma_{1}$ and $g'$ is in the cone counterclockwise between $\Gamma_{0}$ and $\Gamma_{1}$;
\item $g'$ is counterclockwise of $g$.
\end{itemize}
\begin{figure}[htbp]
\begin{center}
\includegraphics{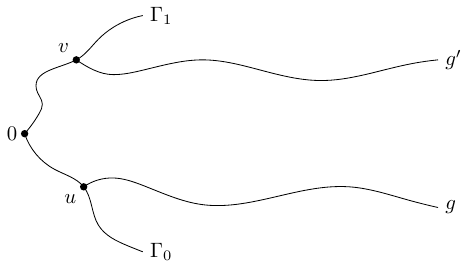}
\end{center}
\caption{Description of the central geometric argument.}
\label{fig:geom1}
\end{figure}
In this setup we are ensured that there exists a geodesic $h\in\tree_0$ which is contained counterclockwise between $g$ and $g'$: Let $(u,u_1,u_2,\ldots)$ be an enumeration of the vertices in $g$ and consider the limit of $\geo(0,u_k)$ as $k$ increases. Since $\geo(0,u_k)$ has to lie within the region confined by $\Gamma_0$, $\Gamma_1$ and $g$, $g'$, due to unique passage times, and since $\geo(0,u_{k+1})$ has to lie counterclockwise of $\geo(0,u_k)$, this limit exists and is shielded off by $g$ and $g'$. In particular, the label of the limiting geodesic $h$ is then contained between those of $g$ and $g'$, and if $g$ and $g'$ coalesce then all three labels coincide.

We will next formulate a statement which will allow us to draw the above picture. We will apply this proposition a number of times in settings which differ somewhat one from another. In order to provide a result that encompasses these different settings we will phrase the statement in terms of translation invariant subfamilies $\mathscr{F}$ of geodesics contained between $\Gamma_0$ and $\Gamma_1$, and ask for the existence of a geodesic with a given asymptotic geodesic property. Below, an {\bf asymptotic geodesic property} is a property such that if $g=(v_0,v_1,v_2,\ldots)$ has the property, then $(v_k,v_{k+1},\ldots)$ does too for all $k\ge1$. Having a certain label, or being cw-/ccw-dense, are examples of asymptotic geodesic properties.


\begin{prop} \label{cycle}
Let $\mathscr{F}(v)$ denote a translation invariant subfamily of geodesics in $\tree_v$ contained almost surely counterclockwise between $\Gamma_0(v)$ and $\Gamma_1(v)$. Let $I$ be an asymptotic geodesic property and let $A(v)$ denote the event that $\mathscr{F}(v)$ contains a geodesic with property $I$. Let $C(v)$ denote the event that $\mathscr{F}(v)$ is contained entirely in the cone counterclockwise between $\Gamma_0(0)$ and $\Gamma_1(0)$.
If $\P\big(A(0)\big)>0$, then, almost surely,
\begin{itemize}
\item there exists $u \in \Gamma_{0}$ for which $A(u)\cap C(u)$ occurs, and
\item there exists $v \in \Gamma_{1}$ for which $A(v)\cap C(v)$ occurs.
\end{itemize}
\end{prop}

We remark that by applying the above proposition twice, we may obtain $u\in\Gamma_0$ and $v\in\Gamma_1$ and geodesics starting at $u$ and $v$ with different properties.

\subsection{Almost sure properties of random geodesics}

Before presenting a proof we give a few typical application of Proposition~\ref{cycle}; in Section~\ref{label of sections} we shall see several more.

\begin{lemma} \label{constant}
Let $G$ be a random non-crossing geodesic which with probability one is contained counterclockwise between $\Gamma_0$ and $\Gamma_1$. Then, $F(G)$ is almost surely constant.
\end{lemma}

\begin{proof}
Suppose the lemma is not true. Then there are two disjoint intervals $[a,b]$ and $[c,d]$  such that $b<c$ and the probability that $F(G)$ is in either of those intervals is positive. We let $A(v)$ be the event that $G(v)$ has label in $[c,d]$ and $B(v)$ be that $G(v)$ has label in $[a,b]$.

Then by Proposition~\ref{cycle} we get a $u$ on $\Gamma_{0}$ such that $A(u)$ occurs. We also get a $v$ on $\Gamma_{1}$ such that $B(v)$ occurs. As the label of $G(u)$ is greater than the label of $G(v)$ we must have that $G(v)<G(u)$. But this means that $G(u)$ and $G(v)$ must cross. This is a contradiction.
\end{proof}

\begin{lemma}\label{rng direction}
Let $G$ be a random non-crossing geodesic almost surely contained counterclockwise between $\Gamma_0$ and $\Gamma_1$. Then there is a set $D$ such that $\dir(G)=D$ almost surely.
\end{lemma}

\begin{proof}
We first show that for any interval $I$ we have $\P\big(\dir(G)\cap I\neq\emptyset\big)\in\{0,1\}$. Assume $\dir(G)\cap I\neq\emptyset$ occurs with positive probability. Applying Proposition~\ref{cycle} we find $u\in\Gamma_0$ and $v\in\Gamma_1$ for which $\dir(G(u))$ and $\dir(G(v))$ have nonempty intersection with $I$. Since $G$ is non-crossing, we have $G(u)\le G(0)\le G(v)$, so via a sandwiching argument it follows that $\dir(G(0))$ has non-empty intersection with $I$.

Now, let $I_0$ be an interval of length $\pi$ that almost surely contains $\dir(\Gamma_0)$ and $\dir(\Gamma_1)$, and hence $\dir(G)$. Define a nested sequence of closed intervals $I_0\supseteq I_1\supseteq I_2\supseteq\ldots$ as follows: In each step split $I_{k-1}$ in half and let $I_k$ denote the cw-most of the two intervals that has nonempty intersection with $\dir(G)$ with positive probability. Since nonempty intersection is a 0-1 event, it follows that $\dir(G)\cap I_k\neq\emptyset$ with probability one, for all $k\ge1$. The intersection $\bigcap_{k\ge1}I_k$ contains a unique point $a$, which is the almost sure infimum of $\dir(G)$. Since $\dir(G)$ is closed it contains $a$. We similarly find a (deterministic) point $b$ so that $\dir(G)=[a,b]$ almost surely.
\end{proof}

\begin{lemma} \label{ccw-dense}
For a random coalescing geodesic $G$, being cw- or ccw-dense are 0-1 events. As a consequence, at least one of the following statements occur with probability one:
\begin{itemize}
\item $G$ is the cw-most geodesic with label $F(G)$;
\item $G$ is the ccw-most geodesic with label $F(G)$.
\end{itemize}
\end{lemma}

\begin{proof}
We focus on the ccw-property here as the other is treated analogously. Let $G$ be a random coalescing geodesic. By relabeling the geodesics if necessary, we may assume that $\Gamma_0\le G<\Gamma_1$.\footnote{Note that for any two random coalescing geodesics $G$ and $G'$, that the events $\{G=G'\}$ and $\{G<G'\}$ are 0-1 events follows from the coalescence property and the ergodic theorem.} Assume further that $G$ is ccw-dense with positive probability. We may then apply Proposition~\ref{cycle} to find $v\in\Gamma_1$ such that $G(v)$ is ccw-dense. Let $g=(v,v_1,v_2,\ldots)$ be a geodesic counterclockwise of $G(v)$ which coincides with $G(v)$ until some point $v_k$ on $G(v)\cap G(0)$. We notice that for $\ell\ge k$ the concatenation of $\geo(0,v_k)$ and $(v_k,v_{k+1},\ldots,v_\ell)$ is a geodesic. Hence, the concatenation of $\geo(0,v_k)$ and $(v_k,v_{k+1},\ldots)$ is an infinite geodesic, and we conclude that $G(0)$ is ccw-dense.

For the second statement of the corollary, recall that $\alpha:=F(G)$ is almost surely constant. By Lemma~\ref{lma:multiplicity} there are at most two geodesics in $\tree_0$ with label $\alpha$ almost surely. In case there is only one, then $G$ is both the cw- and ccw-most geodesic in $\tree_0$ with label $\alpha$. If there are two, then the cw-most has to be ccw-isolated and the ccw-most has to be ccw-dense. Which is the case for $G$ depends on whether it is almost surely ccw-dense or not.
\end{proof}

\subsection{Proof of Proposition~\ref{cycle}}

We will divide the proof into two cases, depending on whether the set of directions $\dir(\Gamma_i)$ for the four geodesics $\Gamma_i$ has width $\pi/2$ or not. The width cannot be larger than $\pi/2$, and if it indeed is as large as $\pi/2$ then the asymptotic shape is necessarily either a square or a diamond. The case when the width is strictly smaller than $\pi/2$ is easier as we in this case can find a half-plane $H$ which both $\Gamma_0$ and $\Gamma_1$ eventually move into. In the remaining case we do not know that this is true, and we will require some additional arguments.\\

\emph{Case 1: Width less than $\pi/2$}.
Consider first the case that the width is strictly smaller than $\pi/2$. We may in this case find two half-planes $H$ and $H'$, both containing the origin as a boundary point, and such that $\Gamma_0$ and $\Gamma_1$ visits the complement of $H\cap H'$ at most finitely many times almost surely. Fix $m\ge1$ and let $B_m(u)$ be the event that $\Gamma_0(u)$ and $\Gamma_1(u)$ visits at most $m$ points in $(H\cap H')^c+u$.\footnote{Here and below $S+u$ denotes the translate of the set $S$ along the vector $u$.} We may make the probability of $B_m(u)$ as close to 1 as we wish by increasing $m$ if necessary. In particular, there exist $\eps>0$ and $m\ge1$ such that
$$
\P\big(A(u)\cap B_m(u)\big)>\eps.
$$

According to the ergodic theorem we may find a density of sites in the symmetric difference $H\Delta H'$ for which $A(u)\cap B_m(u)$ occurs. Since $\Gamma_0$ and $\Gamma_1$ visits $(H\cap H')^c$ at most finitely many times, almost surely, we may find $u$ and $v$ in $H\Delta H'$, sufficiently far from the origin, such that (see Figure~\ref{fig:geom2}):
\begin{itemize}
\item $\Gamma_0(u)$ and $\Gamma_1(u)$ intersect $\Gamma_0(0)$ and does not contain the origin ccw between them;
\item $\Gamma_0(v)$ and $\Gamma_1(v)$ intersect $\Gamma_1(0)$ and does not contain the origin ccw between them;
\item there exists a geodesic $g\in\mathscr{F}(u)$ ccw between $\Gamma_0(u)$ and $\Gamma_1(u)$ with property $I$;
\item there exists a geodesic $g'\in\mathscr{F}(v)$ ccw between $\Gamma_0(v)$ and $\Gamma_1(v)$ with property $I$.
\end{itemize}
\begin{figure}[htbp]
\begin{center}
\includegraphics{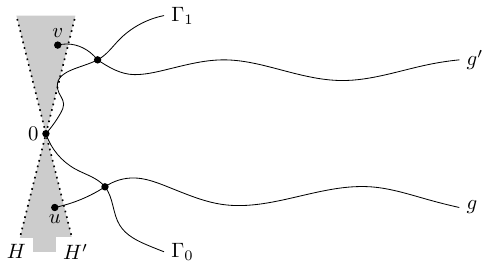}
\end{center}
\caption{$\Gamma_0$ and $\Gamma_1$ leave $H\cap H'$ at most finitely many times. The symmetric difference $H\Delta H'$ has been shaded.}
\label{fig:geom2}
\end{figure}
The geodesics $g$ and $g'$ also intersect $\Gamma_0$ and $\Gamma_1$ respectively, and their subpaths, from these intersections and onwards, are contained between $\Gamma_0$ and $\Gamma_1$. We have thus found $u'\in\Gamma_0$ and $v'\in\Gamma_1$ and geodesics in $\tree_u$ and $\tree_v$ with the required properties.\\

\emph{Case 2: Width equal to $\pi/2$}.
We proceed with the proof in the second case, where the width of the sets of directions of $\Gamma_0$ and $\Gamma_1$ span an angle $\pi/2$. In this case the asymptotic shape is necessarily either a square or a diamond. In either case the proof is the same, so we assume in the following that the asymptotic shape is a diamond, and hence strictly convex in the coordinate directions. By rotational invariance, we may further assume that $\dir(\Gamma_0)$ and $\dir(\Gamma_1)$ intersect at ${\bf e}_1$.

The difficulty that arises in the case the set of directions has width $\pi/2$ is that we cannot guarantee that $\Gamma_0$ and $\Gamma_1$ are both contained in a half-plane. We may nevertheless pick two half-planes $H$ and $H'$ such that $\Gamma_0$ and $\Gamma_1$ visits the complement of $H\cup H'$ at most finitely many times with probability 1. Moreover, we may assume that the half-planes are chosen so that their intersection is a sector symmetric around the first coordinate axis and spans an angle at least $3\pi/4$.

If we choose $m\ge1$ large and let $B_m(u)$ denote the event that $\Gamma_0(u)$ and $\Gamma_1(u)$ visits $(H\cup H')^c+u$ at most $m$ times, then we may again appeal to the ergodic theorem to obtain a density of points in $H\cap H'$ for which $A(u)\cap B_m(u)$ occurs.
By the choice of the half-planes, if $A(u)\cap B_m(u)$ occurs and $u$ is at distance at least $m$ from the origin, then the origin is not contained in the region counterclockwise between $\Gamma_0(u)$ to $\Gamma_1(u)$; compare with Figure~\ref{fig:geom2}. In order to conclude that there exist $u$ and $v$ such that $\Gamma_0(u)$ and $\Gamma_1(u)$, and $\Gamma_0(v)$ and $\Gamma_1(v)$, intersect $\Gamma_0$ and $\Gamma_1$ respectively, 
we need to control the structure of $\Gamma_0$ and $\Gamma_1$ further.

\begin{claim}\label{claim2}
Assume that the asymptotic shape is not flat in the first coordinate direction. Then, there exists an almost surely finite $N\ge1$ such that if $\Gamma_i$, for some $i$, visits $[-\eps n,\eps n]^2+n{\bf e}_1$ for some $n\ge N$, then it does not visit the intersection of $(H\cap H')^c+\eps n{\bf e}_1$ and $[n/2,n/2]^2\setminus[-10\eps n,10\eps n]^2$.
\end{claim}

\begin{proof}[Proof of claim]
Let $C(\eps,z)$ be the event from the proof of Proposition~\ref{extended shape}. Given $\delta>0$, let $N\ge1$ be large so that both $|T(0,z)-\mu(z)|\le\delta|z|$ and $C(\delta,z)$ hold for $|z|\ge N$. If $\delta>0$ is small enough, then for $n\ge N$ we have
$$
T(z,0)<T\big(z,[(H\cap H')^c+\eps n{\bf e}_1]\setminus[-\eps n,\eps n]^2\big)\quad\text{for all }z\in [-\eps n,\eps n]^2+n{\bf e}_1,
$$
and $T(0,y)<T(z,y)$ for all $y$ in the intersection of $(H\cap H')^c+\eps n{\bf e}_1$ and $[n/2,n/2]^2\setminus[-10\eps n,10\eps n]^2$. In particular, no geodesic in $\tree_0$ that visits $[-\eps n,\eps n]^2+n{\bf e}_1$ can also visit the intersection of $(H\cap H')^c+\eps n{\bf e}_1$ and $[n/2,n/2]^2\setminus[-10\eps n,10\eps n]^2$, when $n\ge N$.
\end{proof}

By assumption, there will be arbitrarily large $n$ for which $\Gamma_0$ will visit $[-\eps n,\eps n]^2+n{\bf e}_1$, and similarly for $\Gamma_1$. For these values of $n$ (except for possibly finitely many) Claim~\ref{claim2} says that $\Gamma_0$ cannot visit the intersection of $(H\cap H')^c+\eps n{\bf e}_1$ and $[n/2,n/2]^2\setminus[-10\eps n,10\eps n]^2$. According to the ergodic theorem, each such region will have to contain a density of sites $u$ for which $A(u)\cap B_m(u)$ occurs. That is, Claim~\ref{claim2} and the ergodic theorem together give the existence of points $u$ and $v$ in $H\cap H'$ at distance at least $m$ from the origin, that are not contained counterclockwise between $\Gamma_0$ and $\Gamma_1$, for which $A(u)\cap B_m(u)$ and
$A(v)\cap B_m(v)$ occur. We then find geodesics $g\in\tree_u$ and $g'\in\tree_v$ with property $I$, intersecting $\Gamma_0$ and $\Gamma_1$ respectively. This completes the proof of Proposition~\ref{cycle}.

\section{Labels and non-crossing geodesics}
\label{label of sections}

We examine in this section the ergodic properties of the shift invariant labeling constructed in Section~\ref{cola}. We shall also pay special interest in the clockwise- and counterclockwise-most geodesics with a given label. The reason for this is that every random coalescing geodesic is of this form (recall Lemma~\ref{ccw-dense}), which makes makes them natural candidates for constructing coalescing geodesics. We introduce the notation
\begin{equation*}
\begin{aligned}
\labels&:=\{\alpha\in[0,1]:F(g)=\alpha\text{ for some }g\in\tree_0\},\\
\labels_\star&:=\{\alpha\in[0,1]:\P(\alpha\in\labels)=1\}.
\end{aligned}
\end{equation*}
Since random coalescing geodesics exist and have constant label, the set $\labels_\star$ is non-empty.

\begin{theorem}\label{thm:labels}
The set $\labels_\star$ is closed and $\P(\labels=\labels_\star)=1$.
Moreover, for every $\alpha\in\labels_\star$
\begin{equation*}
G_\alpha^{cw}:=\inf\{g\in\tree_0:F(g)\ge\alpha\}\quad\text{and}\quad G_\alpha^{ccw}:=\sup\{g\in\tree_0:F(g)\le\alpha\}
\end{equation*}
define random non-crossing geodesics whose labels almost surely equal $\alpha$.
\end{theorem}

The proof of the theorem will to a large extent exploit the monotonicity of the labeling together with the geometric argument described in the previous section to imply the existence of geodesics with certain properties. In addition we shall require a result, described next, that excludes the occurrence of certain configurations of geodesics.

\subsection{A weight continuity argument}\label{sec:lowering}

Below we present a key result that will be used to rule out certain configurations of geodesics that correspond to events in the probability space with empty interior. For instance, we shall apply the result in order to prevent geodesics from crossing. The statement is more general than that, however, and is phrased in terms of the order of limits along certain geodesics.

\begin{prop} \label{neymar}
For almost every $\omega\in\Omega_1$ we have for every $v\in\Z^2$ and every geodesic $g=(v_1,v_2,\ldots)$ in $\tree_v$ for which the limit $g^*:=\lim_{k\to\infty}\Geo(0,v_k)$ exists that
\begin{itemize}
\item if $g$ is cw-isolated, then $g^*\le g$;
\item if $g$ is ccw-isolated, then $g^*\ge g$.
\end{itemize}
\end{prop}

Before we proceed to the proof of the proposition, let us describe a typical setting in which it will be applied. The proposition shows that, almost surely, for any pair of neighboring geodesics $g'<g''$ in $\tree_0$ we cannot find a vertex $v$ on either of the two and a geodesic in $\tree_v$ that lies strictly counterclockwise between $g'$ and $g''$. If we could, for $v\in g'$ say, then we could find a cw-isolated geodesic $g=(v_1,v_2,\ldots)$ of this kind, for which $\lim_{k\to\infty}\Geo(0,v_k)$ has to equal $g''$, contradicting Proposition~\ref{neymar}.

Moving on to the proof, we note that the two statements in the proposition follow from one another due to symmetry. It will therefore suffice to prove the former. We mention that an argument similar to the one we present below has in parallel been devised by Nakajima~\cite{nakajima}. While our argument will exploit the planarity of $\Z^2$, the argument in~\cite{nakajima} applies also in higher dimensions.

Given $v\in\Z^2$, let $P_1$ be a finite path starting at $v$, let $P_2$ be a finite path starting at the origin which is disjoint from $P_1$, and let $P_3$ be a finite path connecting some point $x\in P_2$ to some point $y\in P_1$. Subject to their existence, we shall let $g_1=(v_1,v_2,\ldots)$ denote the clockwise-most geodesic in $\tree_v$ that contains $P_1$, and let $g_2$ denote the limit $\lim_{k\to\infty}\Geo(0,v_k)$. We also define a (possibly suboptimal) path to far out vertices on $g_1$ as follows: Let $\gamma_k$ be the concatenation of the segment of $P_2$ from the origin to $x$, $P_3$ from $x$ to $y$ and the segment of $g_1$ from $y$ to $v_k$.
For any $\delta>0$ let $A=A(v,P_1,P_2,P_3,\delta)$ denote the event that (see Figure~\ref{fig:rng}, left)
\begin{itemize}
\item there exists $g\in\tree_v$ that contains $P_1$, and hence $g_1=(v_1,v_2,\ldots)$ is well-defined;
\item the limit $g_2:=\lim_{k\to\infty}\Geo(0,v_k)$ exists, contains $P_2$ and satisfies $g_2>g_1$;
\item $T(\gamma_k)-T(0,v_k)<\delta/2$ for all large $k$.
\end{itemize}
\begin{figure}[htbp]
\begin{center}
\includegraphics{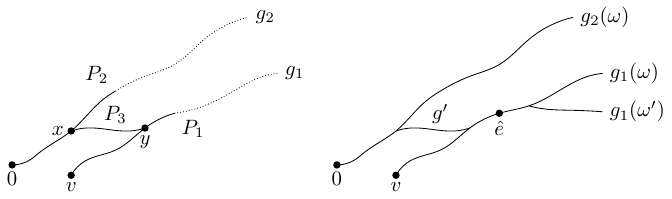}
\end{center}
\caption{Display of the event $A$ (left) and the premise described in the proof of Lemma~\ref{lma:macklemore} (right).}
\label{fig:rng}
\end{figure}
As we shall see next, if Proposition~\ref{neymar} were false, then $A$ would occur with positive probability for some choice of parameters.

\begin{lemma}\label{pre-plie}
Assume that with positive probability there exists a cw-isolated geodesic $g=(v_1,v_2,\ldots)$ such that the limit $g^*:=\lim_{k\to\infty}\Geo(0,v_k)$ exists and satisfies $g^*>g$. Then, for every $\delta>0$, there exist $v$ and $P_1$, $P_2$, $P_3$ such that
$$\P\big(A(v,P_1,P_2,P_3,\delta)\big)>0.$$
\end{lemma}

\begin{proof}
Denote by $A_0$ the event that there exists a cw-isolated geodesic $g=(v_1,v_2,\ldots)$ for which the limit $g^*:=\lim_{k\to\infty}\Geo(0,v_k)$ exists and lies strictly counterclockwise of $g$, and assume that $\P(A_0)>0$. On the event $A_0$, let $\gamma'_k$ denote the path obtained by concatenating $\Geo(0,v)$ and the segment of $g$ from $v$ to $v_k$. Since $g$ is a geodesic
\begin{equation}\label{eq:gamma_alpha}
T(\gamma_k')-T(0,v_k)
\end{equation}
is non-decreasing. It is bounded from above by $T(0,v)$, and hence convergent. Thus for any $\delta>0$ we can find $k$ such that~\eqref{eq:gamma_alpha} is within $\delta/2$ of its limit. It follows, in particular, that for all $\ell\ge k$ we have
$$
T(0,v_k)+T(v_k,v_\ell)-T(0,v_\ell)<\delta/2.
$$
Since $g$ is cw-isolated, countable additivity has that there exists a choice of $v$ and $P_1$, $P_2$, $P_3$ such that with positive probability $A_0$ occurs and $g$ is the cw-most geodesic in $\tree_v$ that contains $P_1$, $g^*$ contains $P_2$, and $\Geo(0,v_k)$ coincides with the path obtained by concatenating a segment of $P_2$ (from the origin to its intersection with $P_3$) with $P_3$. That is, $A(v,P_1,P_2,P_3,\delta)$ has positive probability to occur.
\end{proof}

The rest of the argument will consist in showing that by lowering the weight of an edge along $g_1$ by $\delta$, we can make the event $A$ go from occurring to not occurring. Since we will be able to do this regardless of the state of the weights in any finite box around the origin, the Lebesgue density theorem will imply that $A$ occurs with probability zero.

\begin{lemma} \label{plie}
Suppose that $\hat e$ is an edge and $\delta>0$. Then, for any $\omega,\omega' \in \Omega_1$ such that
\begin{itemize}
\item $\omega'_e=\omega_e$ for all edges $e \neq \hat e$; and
\item $\omega'_{\hat e}< \omega_{\hat e}-\delta$;
\end{itemize}
we have that any (finite or infinite) geodesic in $\omega'$ is either a geodesic in $\omega$ or it contains $\hat e$. 
\end{lemma}

\begin{proof} 
Fix vertices $x$ and $y$. Break up all paths from $x$ to $y$ into those that contain $\hat e$ and those that don't. If the geodesic from $x$ to $y$ in $\omega'$ falls in the first category, then it contains $\hat e$. If not, then it must be a geodesic in $\omega$ too.
\end{proof}

\begin{lemma}\label{lma:macklemore}
Fix $v$, $P_1$, $P_2$, $P_3$ and $\delta$.
Suppose $\omega \in  A(v,P_1,P_2,P_3,\delta)$
and that $\hat e$ is an edge in $g_1 \setminus P_1$. Then for any
point $\omega'$ such that
\begin{itemize}
\item $\omega'_e=\omega_e$ for all edges $e \neq \hat e$; and
\item $\omega'_{\hat e}< \omega_{\hat e}-\delta$;
\end{itemize}
we have $\omega' \not \in A(v,P_1,P_2,P_3,\delta)$.
\end{lemma}

\begin{proof}
The proof is by contradiction. So, suppose both $\omega$ and $\omega'$ belong to $A$. The configurations $\omega$ and $\omega'$ differ only at the edge $\hat e$, where the latter of the two is lower. Hence, since $\hat e$ is contained in $g_1(\omega)$, $g_1(\omega)$ is a geodesic also in $\omega'$ (although possibly different from $g_1(\omega')$). Let $(v_1,v_2,\ldots)$ be an enumeration of the vertices in $g_1(\omega)$ and consider, in $\omega'$, the finite geodesics $\Geo(0,v_k)$. By Lemma~\ref{plie}, either $\Geo(0,v_k)$ contains $\hat e$, or it does not but coincides with the geodesic from the origin to $v_k$ in $\omega$. For large $k$ the latter cannot happen, since $\gamma_k$ would be a faster path to $v_k$. It follows that the limit $g':=\lim_{k\to\infty}\Geo(0,v_k)$ exists (in $\omega'$), contains $\hat e$, and thus coalesces with $g_1(\omega)$ (see Figure~\ref{fig:rng}, right).

There is no restriction to assume that $g_1(\omega')$ is contained counterclockwise between $\Gamma_0(\omega')$ and $\Gamma_1(\omega')$, and hence that $g_1(\omega')\le g_1(\omega)\le g'$. Since $g_2(\omega')$ and $g'$ are obtained as limits from the origin to far out vertices of $g_1(\omega')$ and $g_1(\omega)$ respectively, and since $g_1(\omega')\le g_1(\omega)$, we must have $g_2(\omega')\le g'$. Consequently, since $\omega'\in A$ and thus $g_2(\omega')>g_1(\omega')$, we have $g_2(\omega')$ sandwiched counterclockwise between $g_1(\omega')$ and $g'$. Since both $g_1(\omega')$ and $g'$ contain $\hat e$ we also have that $g_2(\omega')$ contains $\hat e$. Then $g_2(\omega')$ and $g_1(\omega')$ coalesce, contradicting the assumption that $g_2(\omega')>g_1(\omega')$.
\end{proof}

\begin{lemma} \label{brassica}
Let $\mathcal{F}_n$ denote the sigma-algebra generated by the restriction of $\omega$ to edges in $[-n,n]^2$. There exists $\eps,\delta>0$ such that for any $v$, $P_1$, $P_2$ and $P_3$ we have for all large $n$
$$
\P\Big(\P\big(A(v,P_1,P_2,P_3,\delta) \big| \mathcal{F}_n\big)>1-\eps\Big)\le\frac12\P\big(A(v,P_1,P_2,P_3,\delta)\big).
$$
\end{lemma}

\begin{proof}
We may assume that $v$ and $P_1$, $P_2$, $P_3$ are all contained in $[-n,n]^2$. Let $L$, $t$, $\delta$ and $\gamma$ be as defined in assumption {\bf A2} (vi)-(vii). As usual we write $A=A(v,P_1,P_2,P_3,\delta)$.
We introduce the events
\begin{align*}
C_1&:=\big\{|\geo(x,y)|\le Ln\text{ for all }x,y\in[-2n,2n]^2\big\},\\
C_2&:=\big\{|\{e\in\geo(x,y):\omega_e> t\}|>\gamma|\geo(x,y)|\text{ for all }x\in\partial[-n,n]^2,y\in\partial[-2n,2n]^2\big\}.
\end{align*}
Moreover, let $C_1'$ denote the event obtained from $C_1$ by replacing the constant $L$ by $3L$. Note that by {\bf A2} (vi)-(vii) the probability of the events above tends to 1 as $n\to\infty$. (Recall that (vi)-(vii) hold also for any measure satisfying {\bf A1}.)

Let $E_n$ denote the set of edges with an endpoint in $[-2n,2n]^2\setminus[-n,n]^2$. We claim that for every $e\in E_n$ we have almost surely
\begin{equation}\label{eq:twodistr}
\delta\,\P\big(A\cap C_1\cap\{\omega_e\ge t\}\cap\{e\in g_1\}\big|\mathcal{F}_n\big)\le\P(A^c\cap C_1'\cap\{e\in g_1\}|\mathcal{F}_n).
\end{equation}
To see this, let $(\omega,\omega')$ be obtained by first sampling $\omega=\omega'$ on $\{e\}^c=\mathcal{E}^2\setminus\{e\}$, and then sampling $\omega_e$ and $\omega'_e$ independently according to $\P(\,\cdot\,|\,\omega_{\{e\}^c})$. Suppose that $\omega'_e<\omega_e-\delta$. Then, if $\omega\in C_1$ and that there exist $x$ and $y$ such that the geodesic between $x$ and $y$ differ in $\omega$ and $\omega'$, by Lemma~\ref{plie} the geodesic in $\omega'$ must contain $e$. Since the segment from $x$ to $e$, and the segment from $e$ to $y$, has length at most $Ln$, we conclude that $\omega'\in C_1'$. Moreover, on the event that $\omega\in A$, $\omega_e'<\omega_e-\delta$ and $e\in g_1(\omega)$, $g_1(\omega)$ is a geodesic also in $\omega'$. Hence $g_1(\omega')$ exists, and by Lemma~\ref{plie} has to contain $e$. Thus, using Lemma~\ref{lma:macklemore}, we obtain
$$
\{\omega\in A\cap C_1\}\cap\{\omega_e\ge t\}\cap\{e\in g_1(\omega)\}\cap\{\omega'_e<t-\delta\}\subseteq\{\omega'\in A^c\cap C'_1\}\cap\{e\in g_1(\omega')\}.
$$
Due to the relation between $\omega$ and $\omega'$ it follows that
$$
\P\big(A\cap C_1\cap\{\omega_e\ge t\}\cap\{e\in g_1\}\big|\,\omega_{\{e\}^c}\big)\P\big(\omega_e<t-\delta\big|\,\omega_{\{e\}^c}\big)\le\P\big(A^c\cap C_1'\cap\{e\in g_1\}\big|\,\omega_{\{e\}^c}\big),
$$
which by uniform downward finite energy gives~\eqref{eq:twodistr}.

On the event $C_2$, the number of edges in $g_1\cap E_n$ with weight at least $t$ is at least $\gamma n$, so
\begin{equation*}
\begin{aligned}
\P\big(A\cap C_1\cap C_2\big|\mathcal{F}_n\big)\,&\le\,\frac{1}{\gamma n}\E\big[|\{e\in E_n:\omega_e\ge t,e\in g_1\}|{\bf 1}_{A\cap C_1}\big|\mathcal{F}_n\big]\\
&=\,\frac{1}{\gamma n}\sum_{e\in E_n}\P\big(A\cap C_1\cap\{\omega_e\ge t\}\cap\{e\in g_1\}\big|\mathcal{F}_n\big).
\end{aligned}
\end{equation*}
By~\eqref{eq:twodistr} and the fact that on $C_1'$ the geodesic $g_1$ visits at most $3Ln$ edges in $E_n$, we obtain the further upper bound
$$
\frac{1}{\delta\gamma n}\sum_{e\in E_n}\P\big(A^c\cap C_1'\cap\{e\in g_1\}\big|\mathcal{F}_n\big)\le\frac{3L}{\delta\gamma}\P(A^c|\mathcal{F}_n).
$$

In conclusion, setting $M=3L/\delta\gamma$ we get that with probability one
$$
\P(A|\mathcal{F}_n)\le\P(C_1^c|\mathcal{F}_n)+\P(C_2^c|\mathcal{F}_n)+M\P(A^c|\mathcal{F}_n),
$$
implying that
$$
\P(A|\mathcal{F}_n)\le\frac{M+\P(C_1^c|\mathcal{F}_n)+\P(C_2^c|\mathcal{F}_n)}{M+1}.
$$
If $\P(A)=0$ the lemma is trivial, and if $\P(A)>0$ we obtain for large $n$ that $\P(C_1^c|\mathcal{F}_n)<1/4$ and $\P(C_2^c|\mathcal{F}_n)<1/4$ with probability at least $1-\P(A)/2$, as required.
\end{proof}

\begin{proof}[Proof of Proposition~\ref{neymar}]
Proof by contradiction, so suppose that with positive probability there exists a cw-isolated geodesic $g=(v_1,v_2,\ldots)$ such that the limit $g^*:=\lim_{k\to\infty}\Geo(0,v_k)$ exists and satisfies $g^*>g$. Let $\eps,\delta>0$ be as in Lemma~\ref{brassica}. By Lemma~\ref{pre-plie} there exist $v$ and $P_1$, $P_2$, $P_3$ such that 
$$
\P\big(A(v,P_1,P_2,P_3,\delta)\big)>0.
$$
By the Lebesgue density theorem there exist arbitrarily large $n$ for which
$$
\P\big(A(v,P_1,P_2,P_3,\delta) \big| \mathcal{F}_n\big)>1-\eps
$$
with probability at least $\frac23\P(A(v,P_1,P_2,P_3,\delta))$. This contradicts Lemma~\ref{brassica}.
\end{proof}

\subsection{Label lemmata}

In order to prove Theorem~\ref{thm:labels} we proceed with a series of lemmas, several of which are based on combinations of Propositions~\ref{cycle} and~\ref{neymar}.

\begin{lemma} \label{copa}
For any $\alpha \in [0,1]$ we have $\P(\alpha \in \labels)\in \{0,1\}$.
\end{lemma}

\begin{proof}
First of all we recall that a random coalescing geodesic has constant label, so if $\alpha$ coincides with the label of one of the $\Gamma_i$-geodesics, then there is nothing more to prove. It will thus suffice to consider the case that $\P(\alpha\in\labels)>0$ for some $\alpha\in(F(\Gamma_0),F(\Gamma_1))$.

By Proposition~\ref{cycle} we may find $u\in\Gamma_0$ and a geodesic $g\in\tree_u$ with label $\alpha$. Let
\begin{equation}\label{g_minus}
g_{\alpha-}:=\sup\{g\in\tree_0:F(g)<\alpha\}.
\end{equation}
Another application of Proposition~\ref{cycle} finds $v\in\Gamma_1$ and a geodesic in $\tree_v$ with label $\alpha$. Since $g_{\alpha-}$ lies clockwise of this geodesic, it follows that $g_{\alpha-}$ has label at most $\alpha$. Then, either $g_{\alpha-}$ has label $\alpha$, in which case we are done, or, due to the consistency between the labeling and the ordering, $g_{\alpha-}$ and $g$ cross. Suppose the latter occurs and let $v$ denote their last common point. Let $e$ denote the edge out of $v$ used by $g$, and let $g'$ denote the clockwise-most geodesic in $\tree_v$ containing $e$. Let $(v_1,v_2,\ldots)$ be an enumeration of the sites on $g'$. Since $v$ is a common point of $g'$ and $g_{\alpha-}$ the limit $g^*:=\lim_{k\to\infty}\Geo(0,v_k)$ exists and satisfies $g^*\ge g'>g_{\alpha-}$. Since $g^*>g_{\alpha-}$ it has label at least $\alpha$. If its label exceeds $\alpha$, then $g^*>g'$. The latter contradicts Proposition~\ref{neymar}, so we conclude that there exists a geodesic in $\tree_0$ with label $\alpha$, almost surely.
\end{proof}

\begin{lemma}\label{lma:exists_label}
Let $I\subseteq[0,1]$ be an (open or closed) interval. Then 
$$\P(\labels\cap I\neq\emptyset)\in\{0,1\}.$$
\end{lemma}

\begin{proof}
If $I$ contains $F(\Gamma_i)$ for some $i$, then there is nothing more to prove. We may therefore assume that $I$ is a subset of $(F(\Gamma_0),F(\Gamma_1))$.
If for some $\alpha \in I$ we have
$$\P(\alpha \in \labels)>0,$$
then by Lemma~\ref{copa} we are done. If not we may split $I$ into two disjoint intervals $I_1$ and $I_2$ (with $I_1<I_2$) such that there is positive probability to find a geodesic with label in either of the two. We apply Proposition~\ref{cycle} to find a site $u\in\Gamma_0$ and a geodesic $g\in\tree_u$ with label in $I_1$, and a site $v\in\Gamma_1$ and a geodesic $g'\in\tree_v$ with label in $I_2$. We then have
$$
F(\Gamma_0)<F(g)< F(g')<F(\Gamma_1),
$$
and hence recover the picture in Figure~\ref{fig:geom1}. Let $(u,u_1,u_2,\ldots)$ be an enumeration of the vertices in $g$. We may take the limit of the sequence of finite geodesics $\geo(0,u_k)$ to obtain a geodesic in $\tree_0$ shielded of by $g$ and $g'$. This geodesic must have a label in $[F(g),F(g')]$, which is a subset of $I$, as required.
\end{proof}

\begin{lemma}\label{lma:exists_label_wp1}
Let $I\subseteq[0,1]$ be an (open or closed) interval. If $\P(\labels\cap I\neq\emptyset)=1$, then there exists $\alpha\in I$ such that $\P(\alpha\in\labels)=1$.
\end{lemma}

\begin{proof}
Observe first that if $I$ is open, then we may pick a closed subinterval $I'\subseteq I$ for which $\P(\labels\cap I'\neq\emptyset)>0$. According to Lemma~\ref{lma:exists_label} this probability is then 1, so it will suffice to prove the lemma for $I$ closed.
Assume that $I=[a,b]$. Either $\P(b\in\labels)>0$, or not. In the former case Lemma~\ref{copa} completes the proof. In the latter let
$$
c:=\sup\big\{\alpha\in I:\P(\labels\cap(\alpha,b]\neq\emptyset)>0\big\}
$$
and note that $\P(\labels\cap(c-\eps,c]\neq\emptyset)=1$ for every $\eps>0$. Since $c\in[a,b]$, either $\P(c\in\labels)>0$, or we can find $c'\in[a,c)$ such that
$$
\P(\labels\cap(c'-\eps,c')\neq\emptyset)=1
$$
for all $\eps>0$. In the latter case, with probability one, the geodesic $g_{c'-}$ defined in~\eqref{g_minus} is cw-dense, has label at least $c'$, and every geodesic in $\tree_0$ clockwise of $g_{c'-}$ has label strictly below $c'$. It follows that $g_{c'-}$ is non-crossing, since a crossing (i.e., a non-coalescing intersection) of two representatives would contradict the monotonicity of the labeling. By Lemma~\ref{constant} the label of $g_{c'-}$ is almost surely constant. Since the label is contained in $[c',c]\subseteq I$, we are done.
\end{proof}

\begin{lemma}\label{L_star}
The set $\labels_\star$ is closed.
\end{lemma}

\begin{proof}
We shall prove that the limit of every (strictly) monotone sequence of elements in $\labels_\star$ is contained in $\labels_\star$. Assume first that $(\alpha_k)_{k\ge1}$ is a strictly decreasing sequence of elements in $\labels_\star$ and let $\alpha$ be its limit. Then, almost surely, there is a decreasing sequence $(g_k)_{k\ge1}$ of geodesics in $\tree_0$ with labels strictly larger than, but converging to, $\alpha$. We may without restriction assume that each $g_k$ is ccw-isolated. Let
$$
g_{\alpha+}:=\{g\in\tree_0:F(g)>\alpha\}.
$$
By assumption, $g_{\alpha+}$ is almost surely the limit of a decreasing sequence of geodesics with labels strictly larger than $\alpha$. Since $g_{\alpha+}$ has label at most $\alpha$ it is non-crossing. Then, by Lemma~\ref{constant}, the label of $g_{\alpha+}$ is almost surely constant.

The label of $g_{\alpha+}$ either equals $\alpha$, in which case there is nothing left to prove, or it is strictly smaller than $\alpha$, almost surely. Suppose the latter is true. Since $g_{\alpha+}$ is ccw-dense, it follows by Lemma~\ref{label post-props} that there exists $v\in\Z^2$ and a ccw-isolated geodesic $g\in\tree_v$ such that $g_{\alpha+}\le g<g_k$ for all $k\ge1$. As a consequence, the limit $\lim_{\ell\to\infty}\Geo(0,v_\ell)$, where $(v_1,v_2,\ldots)$ is an enumeration of $g$, exists and equals $g_{\alpha+}$. By Proposition~\ref{neymar}, $g$ and $g_{\alpha+}$ coalesce. Since also $g_{\alpha+}(v)$ is ccw-dense we can find $u\in\Z^2$ and a ccw-isolated geodesic $g'\in\tree_u$ such that $g_{\alpha+}(v)\le g'<g_k$ for all $k$. By Proposition~\ref{neymar} we have again that $g'$ and $g_{\alpha+}$ coalesce, contradicting the fact that $g_{\alpha+}\le g<g_{\alpha+}(v)\le g'$.

Assume next that $(\alpha_k)_{k\ge1}$ instead is a strictly increasing sequence of elements of $\labels_\star$, and denote by $\alpha$ its limit. This time, almost surely, there is an increasing sequence $(g_k)_{k\ge1}$ of geodesics in $\tree_0$ with labels strictly below but converging to $\alpha$. We may without restriction assume they are ccw-isolated. The geodesic $g_{\alpha-}$, defined in~\eqref{g_minus}, is by assumption cw-dense and has label at least $\alpha$. It is then non-crossing, so by Lemma~\ref{constant} the label of $g_{\alpha-}$ is almost surely constant.

Suppose the label of $g_{\alpha-}$ is strictly larger than $\alpha$. Then a positive amount of mass in the flow on $\tree_v$ has to escape along geodesics $g\in\tree_v$ satisfying $g_{\alpha-}\ge g>g_k$ for all $k$, for a density of sites $v\in\Z^2$. For each such $v$ there is at least one such cw-isolated geodesic $g$ in $\tree_v$. The limit from the origin to the tail of that geodesic exists and equals $g_{\alpha-}$. By Proposition~\ref{neymar}, $g$ has to coalesce with $g_{\alpha-}$. Similarly we find a vertex $u$ and a cw-isolated geodesic $g'\in\tree_u$ coalescing with $g_{\alpha-}(v)$. In particular $g_{\alpha-}>g'>g_k$ for all $k$, contradicting Proposition~\ref{neymar}.
\end{proof}

\begin{lemma}
$G_\alpha^{cw}$ and $G_\alpha^{ccw}$, for every $\alpha\in\labels_\star$, define random non-crossing geodesics.
\end{lemma}

\begin{proof}
Since the proofs for $G_\alpha^{cw}$ and $G_\alpha^{ccw}$ are analogous, we consider only the former.
Assume the contrary, that there exists $v\in\Z^2$ for which $G_\alpha^{cw}$ and $G_\alpha^{cw}(v)$ cross with positive probability. On the event of a crossing, let $z$ denote the last common vertex in the intersection, and assume that $G_\alpha^{cw}(v)<G_\alpha^{cw}$. Let $e$ denote the edge out of $z$ used by $G_\alpha^{cw}(v)$, and note that $\tree_z$ contains at least one geodesic $g$ with label $\alpha$ containing $e$ (possibly $G_\alpha^{cw}(v)$ itself). We note that $g$ may be assumed ccw-isolated, since either $G_\alpha^{cw}(v)$ is ccw-isolated with label $\alpha$, or $G_\alpha^{cw}(v)$ is ccw-dense and a decreasing limit of geodesics labeled $\alpha$. Let $(v_1,v_2,\ldots)$ be an enumeration of the vertices in $g$. Since $z$ is a vertex on both $g$ and $G_\alpha^{cw}$, the limit $g'$ of $\Geo(0,v_k)$ as $k\to\infty$ exists and satisfies $g'\le g<G_\alpha^{cw}$. Since $g$ has label $\alpha$ we have to have $g'<g$. This contradicts Proposition~\ref{neymar}.
\end{proof}

\begin{lemma}\label{cw/ccw-most label}
For every $\alpha\in\labels_\star$ the labels of $G_\alpha^{cw}$ and $G_\alpha^{ccw}$ almost surely equal $\alpha$.
\end{lemma}

\begin{proof}
Since the proofs for $G_\alpha^{cw}$ and $G_\alpha^{ccw}$ are similar, we consider only the former.
We argue in a similar fashion as in the proof of Lemma~\ref{L_star}.
Either $G_\alpha^{cw}$ is ccw-isolated with positive probability, or it is almost surely ccw-dense. In the former case $G_\alpha^{cw}$ is the least geodesic with label $\alpha$, with positive probability, and since the label of $G_\alpha^{cw}$ is almost surely constant it has to equal $\alpha$. In the latter case $G_\alpha^{cw}$ is the limit of a decreasing sequence $(g_k)_{k\ge1}$ of ccw-isolated geodesics with labels larger than or equal to $\alpha$. Suppose that $F(G_\alpha^{cw})<\alpha$. Then there exists $v\in\Z^2$ and a ccw-isolated geodesic $g\in\tree_v$ such that $G_\alpha^{cw}\le g<g_k$ for all $k\ge1$. Consequently, the limit from the origin along $g$ exists and equals $G_\alpha^{cw}$. By Proposition~\ref{neymar} the two coalesce. By an analogous argument we can find $u\in\Z^2$ and a ccw-isolated geodesic $g'\in\tree_u$ coalescing with $G_\alpha^{cw}(v)$. Since $g<G_\alpha^{cw}(v)$ we find that $G_\alpha^{cw}<g'<g_k$ for all $k$, contradicting Proposition~\ref{neymar}.
\end{proof}

\subsection{Proof of Theorem~\ref{thm:labels}}


Due to Lemmas~\ref{L_star}-\ref{cw/ccw-most label}, it remains only to show that $\labels=\labels_\star$ almost surely. Since $\labels_\star$ is a closed subset of $[0,1]$, by Lemma~\ref{L_star}, its complement $[0,1]\setminus \labels_\star$ is the union of at most countably many disjoint open intervals $(a_m,b_m)$, for $m\ge1$. For each $m$ it follows by Lemma~\ref{lma:exists_label_wp1} that $\P(\labels\cap (a_m,b_m)\neq\emptyset)=0$. Since there are at most countably many such sets, we conclude that $\P(\labels\subseteq \labels_\star)=1$. Below we shall call the (at most countably many) points $a_m$ and $b_m$ \emph{exposed} points of $\labels_\star$, and the remaining points of $\labels_\star$ we call \emph{constrained} points. 

Next, let $L_n$ be the subset of $\labels_\star$ obtained as follows: Split $[0,1]$ into $2^n$ closed intervals $I_k$ of length $2^{-n}$, and include in $L_n$ the maximal element of $\labels_\star\cap I_k$ for each $k$. The set $L_\infty=\bigcup_{n\ge1}L_n$ is then a countable and dense subset of $\labels_\star$. In particular $\P(L_\infty\subseteq\labels)=1$.

Assume that $\labels_\star\setminus\labels\neq\emptyset$ with positive probability. Since there are at most countably many exposed points, this would then have to occur at a constrained point in $\labels_\star$. That is, there would be $\eps>0$ such that with positive probability we can find an constrained point $\alpha\in\labels_\star$ and an increasing sequence $(g_k)_{k\ge1}$ with labels converging to $\alpha$, whereas $F(\lim_{k\to\infty}g_k)>\alpha+\eps$. In this case no label in $[\alpha,\alpha+\eps]$ is represented in $\labels$, which contradicts the fact that $L_\infty$ is contained in $\labels$ with probability one.

\subsection{Continuity of the labeling}

While we have already seen that the cw- and ccw-most geodesics with a given label are non-crossing, and that a the limit of a sequence of geodesics with labels converging to a given label has to have that label, we obtain as a corollary to Theorem~\ref{thm:labels} the following stronger statement showing that these properties hold simultaneously for all $\alpha\in\labels_\star$.

\begin{lemma}\label{continuity_labels}
With probability one we have
\begin{enumerate}[\quad (a)]
\item for all $\alpha\in\labels_\star$ and $u,v\in\Z^2$ that $G_\alpha^{cw}(u)$ and $G_\alpha^{cw}(v)$ are non-crossing;
\item for all $\alpha\in\labels_\star$ and $u,v\in\Z^2$ that $G_\alpha^{ccw}(u)$ and $G_\alpha^{ccw}(v)$ are non-crossing;
\item for every monotone sequence $(g_k)_{k\ge1}$ in $\tree_0$ that
$$
\lim_{k\to\infty}F(g_k)=F\Big(\lim_{k\to\infty}g_k\Big).
$$
\end{enumerate}
\end{lemma}

\begin{proof}
By Theorem~\ref{thm:labels} $G_\alpha^{cw}$ is almost surely non-crossing and $\labels_\star$ have at most countably many exposed points. It follows that $G_\alpha^{cw}$ are non-crossing for all exposed points of $\labels_\star$ almost surely. Moreover, since $\labels=\labels_\star$, almost surely, it follows that for every $\alpha$ which is a constrained point of $\labels_\star$ that $G_\alpha^{cw}$ is cw-dense. Consequently, with probability one, there cannot be $\alpha$ and $u,v\in\Z^2$ for which $G_\alpha^{cw}(u)$ and $G_\alpha^{cw}(v)$ cross, since that would contradict the monotonicity of the labeling. This proves part~\emph{(a)}, and since~\emph{(b)} is analogous, we proceed with~\emph{(c)}.

Assume that with positive probability there exists a monotone (decreasing, say) sequence of geodesics $(g_k)_{k\ge1}$ whose limit $g$ satisfies
$$
F(g)<\lim_{k\to\infty}F(g_k).
$$
Since $\labels=\labels_\star$ almost surely, $F(g)$ cannot with positive probability attain an constrained point of $\labels_\star$. Instead, there would have to exist exposed points $\beta<\alpha$ in $\labels_\star$ such that $F(g)=\beta$ and $F(g_k)=\alpha$ for all large $k$. In that case $g=G_\alpha^{cw}$, so that $F(G_\alpha^{cw})=\beta$. Since there are at most countably many exposed points of $\labels_\star$, this is inconsistent with the fact that $G_\alpha^{cw}$ has label $\alpha$ almost surely.
\end{proof}

\section{Coalescence}
\label{sec:coalescence}

Let $\rcg$ denote the class of random geodesics of the form $G_\alpha^{ccw}$ and $G_\alpha^{cw}$ for $\alpha\in\labels_\star$. In the previous section we saw that these geodesics are non-crossing. In this section we prove that they are coalescing, and that there are no random coalescing geodesics outside of this class.

\begin{theorem}\label{thm:coalescence}
Every member of $\rcg$ is a random coalescing geodesic, and for every random coalescing geodesic $G$ there exists $G'\in\rcg$ such that $G=G'$ almost surely. Moreover, for every $\alpha\in\labels_\star$ we have
\begin{enumerate}[\quad (a)]
\item $\P\big(\exists\text{ at least two geodesics in $\tree_0$ with label }\alpha\big)\in\{0,1\}$;
\item $\P\big(\exists\text{ at least three geodesics in $\tree_0$ with label }\alpha\big)=0$.
\end{enumerate}
\end{theorem}

In order to prove Theorem~\ref{thm:coalescence} we primarily need to show that each member of $\rcg$ is coalescing. The outline of our proof goes as follows.
We first show that every non-crossing geodesic eventually moves into a half-plane. We next show that if $\omega'$ is an alteration of $\omega$ on finitely many edges, and if $(v_1,v_2,\ldots)$ is an enumeration of $G(\omega)$, then $G(v_k)(\omega')=(v_k,v_{k+1},\ldots)$ for all large $k$. Then we let $R$ be any shift invariant order on $\Z^2$.
We show by the mass-transport principle that for any coalescence class of $G$ there is no least element according to the order $R$.  However, by the upward finite energy property we show that if $G$ is not coalescing then there exists a coalescence class with a least element. This is a contradiction and $G$ has to be coalescing.


Just as for coalescing geodesics, we say that a random non-crossing geodesic $G$ {\bf eventually moves into a half-plane} $H$ if for every parallel half-plane $H' \subseteq H$ and $v\in\Z^2$ we have
$$
\P\big(|G(v) \cap H'^c|<\infty\big)=1.
$$
Moreover, a random non-crossing geodesic $G$ defines an equivalence relation on $\Z^2$ where $u\sim v$ if $G(u)$ and $G(v)$ coalesce. The equivalence classes of this relation will be referred to as the {\bf coalescence classes} of $G$. If $G$ almost surely has a unique coalescence class then $G$ is coalescing.

\subsection{Non-crossing geodesics eventually move into a half-plane}

Recall that by $H_i$, for $i=0,1,\ldots,7$, we denote the eight half-planes with normal vector in direction $i\pi/4$.
We shall prove that every geodesic in the class $\rcg$ eventually moves into one of the eight half-planes $H_i$.

\begin{prop}\label{riders on the storm}
For every random non-crossing geodesic $G$ of the form $G_\alpha^{cw}$ or $G_\alpha^{ccw}$, for some $\alpha\in\labels_\star$, there exists $i$ such that $G$ eventually moves into $H_i$.
\end{prop}

The argument is easy in the case that the limiting shape is neither a square nor a diamond. In this case the shape has at least eight extreme points: two corresponding to angles in $[0,\pi/2)$, and another six are produced by rotation by right-angles. For each consecutive pair of extreme points there is, by Theorem~\ref{yahoo}, a random coalescing geodesic directed in between. This gives a set of eight random coalescing geodesics, and each consecutive pair of geodesics are together directed in a sector of width at most $\pi/2$. Since this sector contains at most three angles of the form $i\pi/4$, we can produce a $H_i$ half-plane with five of the remaining points. Since every other geodesic is contained in between two consecutive geodesics of the above form, each geodesic 
will 
eventually move into one of the eight half-planes $H_i$.

In the general case, where we make no assumption on the limiting shape, we will argue via a subsequential limiting argument of Busemann measures, similar to that of Damron and Hanson~\cite{damhan14}.


\begin{proof}[Proof of Proposition~\ref{riders on the storm}]
Let $G$ be a random non-crossing geodesic of the form $G_\alpha^{cw}$ or $G_\alpha^{ccw}$. Recall that the Busemann function of $G(y)$, for all $x,y,z\in\Z^2$, is defined as the limit
$$
B_{G(y)}(x,z)=\lim_{k\to\infty}\big[T(x,y_k)-T(z,y_k)\big],
$$
where $(y,y_1,\ldots)$ in an enumeration of the sites of $G(y)$. That the limit exists, almost surely and in $L^1$, follows from Lemma~\ref{rioja}. Moreover, $B_{G(y)}(x,z)$ is additive and hence completely determined by the configuration $\theta^y=(\theta^y(x))_{x\in\Z^2}$ where
$$
\theta^y(x):=\big(B_{G(y)}(x,x+{\bf e}_1),B_{G(y)}(x,x+{\bf e}_2)\big).
$$

Next, recall the sequence of equivalence relations on $\Z^2$ introduced in Section~\ref{cola}: Let $\{\xi_z\}_{z\in\Z^2}$ be independent $[0,1]$-uniform random variables, and let $S_i=\{z\in\Z^2:\xi_z\le1/4^i\}$. The function $f_i:\Z^2\to S_i$ that maps each point to the point in $S_i$ at least $\ell_1$-distance induces an equivalence relation on $\Z^2$, and we denote the associated family of equivalence classes by $(V_i(x))_{x\in\Z^2}$, for $i=1,2,\ldots$. Based on this (sequence of) equivalence classes we define $\theta_i=(\theta_i(x))_{x\in\Z^2}$ as $\theta_i(x)=\theta^{f_i(x)}(x)$, and note that $\theta_i$ encodes the same Busemann function at $x$ and $y$ if the two points belong to the same equivalence class. Moreover, we consider the sequence $(\geo(x,y_k))_{k\ge1}$ of finite geodesics along $G(f_i(x))$. Either this sequence converges to a limit, which we then denote by $\gamma_x^i$, or there are two subsequential limits (see Figure~\ref{fig:coal3}).
\begin{figure}[htbp]
\begin{center}
\includegraphics{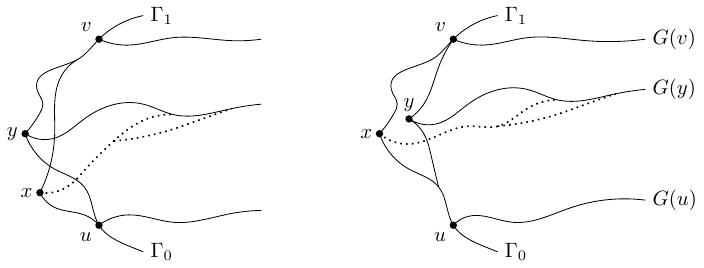}
\end{center}
\caption{In the left picture there is a unique subsequential limit. In the right there are potentially two, one below and another above $G(y)$.}
\label{fig:coal3}
\end{figure}
In the latter case we let $\gamma_x^i$ denote the counterclockwise-most of these limits. Finally we define $\eta\in\{0,1\}^{\Z^2\times\mathcal{E}^2}$ as follows:
\begin{equation*}
\eta(x,e):=\left\{
\begin{aligned}
& 1 && \text{if $x+e$ is used by }\gamma_x^i;\\
& 0 && \text{otherwise}.
\end{aligned}
\right.
\end{equation*}

Let $\Omega_1=[0,\infty)^{\mathcal{E}^2}$, $\Omega_4=[0,1]^{\Z^2}$, $\Omega_3=(\R^2)^{\Z^2}$ and $\Omega_5=\{0,1\}^{\Z^2\times\mathcal{E}^2}$. For each $i$ we may exhibit a measurable map $\Psi_i:\Omega_1\times\Omega_4\to\Omega_1\times\Omega_3\times\Omega_5$ as $(\omega,\xi)\mapsto(\omega,\theta_i,\eta_i)$. The measure $\P\times\leb$ may be pushed forward through the map $\Psi_i$ to give a measure $\nu_i$ on $\Omega_1\times\Omega_3\times\Omega_5$. By tightness, the sequence $(\nu_i)_{i\ge1}$ will have a subsequential limit, which we denote by $\nu$.




We next use $\nu$ to reconstruct a Busemann-like function $\hat B:\Z^2\times\Z^2\to\R$ associated to $G$. For $\theta=(\theta_1,\theta_2)$ in $\Omega_3$ we define for any neighboring pair of vertices $x$ and $y$
\begin{equation*}
b(x,y):=\left\{
\begin{aligned}
& \theta_1(x) && \text{if }y=x+{\bf e}_1,\\
& \theta_2(x) && \text{if }y=x+{\bf e}_2,\\
& -\theta_1(y) && \text{if }y=x-{\bf e}_1,\\
& -\theta_2(y) && \text{if }y=x-{\bf e}_2.
\end{aligned}
\right.
\end{equation*}
Let $\{\pi(x,z):x,z\in\Z^2\}$ be some predefined family of finite paths. For $x,z\in\Z^2$, let $x_0=x,x_1,\ldots,x_m=z$ be the enumeration of the sites of $\pi(x,z)$, and define
\begin{equation*}
\hat B(x,z):=\sum_{i=0}^{m-1}b(x_i,x_{i+1}).
\end{equation*}

We now make two claims. First that with $\nu$-measure 1 the configuration $\eta$ in $\Omega_5$ encodes a unique geodesic in $\tree_x$ for each $x\in\Z^2$, and second that $\hat B$ functions as a Busemann function for this family of geodesics.

\begin{claim}\label{tab}
For $\nu$-almost every $(\omega,\theta,\eta)$ the graph encoded by $(\eta(x,e))_{e\in\mathcal{E}^2}$ represents an infinite geodesic $\gamma_x$ starting at $x$. Moreover, $\nu(\gamma_x=G(x))\ge1/2$ for all $x\in\Z^2$.
\end{claim}

\begin{proof}[Proof of claim]
That $\nu_i$, for each $i\ge1$, is supported on configurations such that the graph encoded by $(\eta(x,e))_{e\in\mathcal{E}^2}$ represents an infinite geodesic originating at $x$ is immediate from construction. Since this event represents a closed subset of the sample space, the first part of the claim follows from the Portmanteau theorem.

For the latter part of the claim we note that for every $y\in\Z^2$ that
$$
1\le\P\big(G(0)\le G(y)\big)+\P\big(G(y)\le G(0)\big)=\P\big(G(0)\le G(y)\big)+\P\big(G(0)\le G(-y)\big).
$$
Consequently, due to symmetry, we have for all $i\ge1$ that
\begin{equation*}
\begin{aligned}
\P\big(G(0)\le G(f_i(0))\big)&=\sum_{y\in\Z^2}\P\big(G(0)\le G(y)\big)\P(f_i(0)=y)\\
&=\frac12\sum_{y\in\Z^2}\Big[\P\big(G(0)\le G(y)\big)+\P\big(G(0)\le G(-y)\big)\Big]\P(f_i(0)=y),
\end{aligned}
\end{equation*}
which by the previous calculation is at least $1/2$. By an analogous argument we also obtain $\P\big(G(0)\ge G(f_i(0))\big)\ge1/2$.

Since random non-crossing geodesics have constant labels we may without restriction, due to Lemma~\ref{ccw-dense}, assume that $G$ is almost surely contained counterclockwise between $\Gamma_0$ and $\Gamma_1$. Since $\Gamma_0$ and $\Gamma_1$ are coalescing, we may due to Proposition~\ref{cycle} find $u\in\Gamma_0(0)\cap\Gamma_0(f_i(0))$ and $v\in\Gamma_1(0)\cap\Gamma_1(f_i(0))$ such that $G(u)$ and $G(v)$ are contained entirely in the cone counterclockwise between $\Gamma_0$ and $\Gamma_1$. Since $G$ is non-crossing and contained between $\Gamma_0$ and $\Gamma_1$ it follows that $G(f_i(0))$ and any subsequential limit of $(\geo(0,y_k))_{k\ge1}$ is sandwiched between $G(u)$ and $G(v)$ (see Figure~\ref{fig:coal3}). Consequently, each subsequential limit has label $\alpha=F(G)$. Moreover, on the event $\{G(0)\le G(f_i(0))\}$ we necessarily have $\gamma_0^i\ge G(0)$, and we similarly have $\gamma_0^i\le G(0)$ on the event $\{G(0)\ge G(f_i(0))\}$. In the case that $G=G_\alpha^{ccw}$, then the former implies that $\gamma_0^i=G(0)$. The same is implied by the latter in the case that $G=G_\alpha^{cw}$. In either case the conclusion is that $\nu_i$ puts mass at least $1/2$ on the event that $\eta$ encodes $G(0)$. Since this represents a closed event, the second statement of the claim follows from the Portmanteau theorem.
\end{proof}

\begin{claim}\label{tac}
The function $\hat B:\Z^2\times\Z^2\to\R$ has the following almost sure properties:
\begin{enumerate}[\quad (a)]
\item $\hat B(x,z)=\hat B(x,y)+\hat B(y,z)$ for all $x,y,z\in\Z^2$;
\item $|\hat B(x,z)|\le T(x,z)$ for all $x,z\in\Z^2$;
\item $\hat B(x,z)=T(x,z)$ for all $x$ and $z\in\gamma(x)$.
\end{enumerate}
\end{claim}

\begin{proof}[Proof of claim]
First recall that the properties in~\emph{(a)} and~\emph{(b)} are satisfied, almost surely, by $B_{G(v)}$ for any $v\in\Z^2$. Second note that, almost surely, $B_{G(v)}(x,z)=T(x,z)$ for all $z$ contained in $\geo(x,v_k)$ for infinitely many $k$, where $v_1,v_2,\ldots$ is an enumeration of the sites in $G(v)$.
Now, fix $x$, $y$ and $z$, and let $E_i$ denote the event that the paths $\pi(x,y)$, $\pi(y,z)$ and $\pi(x,z)$ are contained in the same equivalence class induced by $(V_i(x))_{x\in\Z^2}$. On the event $E_i$ we have that the values of $(\theta_i,\eta_i)$ on sites on $\pi(x,y)$, $\pi(y,z)$ and $\pi(x,z)$ encode the Busemann function and geodesics corresponding to the same representative of $G$, which thus satisfy the properties~\emph{(a)},~\emph{(b)} and~\emph{(c)}. Since $E_i$ occurs with probability tending to 1 as $i\to\infty$, we conclude that the three properties occur with $\nu_i$-probability tending to 1. 
As the three properties correspond to closed events of the sample space, we conclude that they are given probability one by the limiting measure $\nu$.
\end{proof}

It is immediate from construction that the limiting measure $\nu$ is translation invariant, and hence that $\hat B(x,z)$ equals $\hat B(x+y,z+y)$ in distribution, for all $x$, $y$ and $z$. Thus, by the ergodic theorem, for all $z\in\Z^2$ the (possibly random) limit
$$
\hat\rho(z):=\lim_{n\to\infty}\frac1n\hat B(0,nz)=\lim_{n\to\infty}\frac1n\sum_{j=1}^n\hat B((j-1)z,jz)
$$
exists $\nu$-almost surely and in $L^1$. Moreover, by additivity of $\hat B$ and the characterization of $\hat\rho(z)$ as the conditional expectation $\E[\hat B(0,z)|\mathcal{I}]$ of $\hat B(0,z)$ given the sigma-algebra $\mathcal{I}$ generated by all invariant events, it follows that
$$
\hat\rho(z)=\E[\hat B(0,z)|\mathcal{I}]=\big(\E[\hat B(0,\mathbf{e}_1)|\mathcal{I}],\E[\hat B(0,\mathbf{e}_2)|\mathcal{I}]\big)\cdot z.
$$
That is, $\hat\rho$ extends to a (random) linear functional $\hat\rho:\R^2\to\R$ given by $\hat\rho(x)=(\hat\rho({\bf e}_1),\hat\rho({\bf e}_2))\cdot x$. By repeating the argument in the proof of Proposition~\ref{sonoma} we obtain that
\begin{equation}\label{eq:nu-shape}
\nu\Big(\limsup_{|z|\to\infty}\frac{1}{|z|}\big|\hat B(0,z)-\hat\rho(z)\big|=0\Big)=1.
\end{equation}
Moreover, arguing as in the proof of Proposition~\ref{mendoza}, we have by~\eqref{eq:nu-shape}, the shape theorem and the properties of $\hat B$ that for any sequence $(x_n)_{n\ge1}$ such that $|x_n|\to\infty$ and $x_n/|x_n|\to1$ that with $\nu$-probability one
$$
\hat\rho(x)=\lim_{n\to\infty}\frac{1}{|x_n|}\hat B(0,x_n)\le\lim_{n\to\infty}\frac{1}{|x_n|}T(0,x_n)=\mu(x).
$$
Moreover, if $(x_n)_{n\ge1}$ is a subsequence of $\gamma_0$, then equality holds and $\hat\rho(x)=\mu(x)$. It follows that with $\nu$-probability one the (random) line $\{x\in\R^2:\hat\rho(x)=1\}$ is a supporting line for $\partial\ball$ and $\dir(\gamma_0)$ is a subset of $\arc(\hat\rho)$. In particular, the set $\dir(\gamma_0)$ is contained in an arc of $S^1$ of width at most $\pi/2$. By Claim~\ref{tab} we have, with $\nu$-probability at least $1/2$, that $\gamma_0=G(0)$, and thus that $\dir(G)$ is contained in an arc of $S^1$ of width at most $\pi/2$. Since $\dir(G)$ is deterministic, as of Lemma~\ref{rng direction}, the latter occurs with probability one. The arc containing $\dir(G)$ may contain at most three angles of the form $i\pi/4$, and by choosing a set of five consecutive angles among the remaining ones we obtain a half-plane of the form $H_i$. Hence, for some $i=0,1,\ldots,7$ we have $G$ almost surely contained in $H_i$.
\end{proof}

\subsection{The effect of a local modification}

We shall prove that a local modification of the weight configuration may only have a local effect on the image of a random non-crossing geodesic. To formalize the statement, we consider the following local modification coupling (see e.g.~\cite{hagjon06}) of two configurations $(\omega,\omega')$ in $\Omega_1\times\Omega_1$: Let $\mathbf{Q}$ be some probability measure on finite subsets of $\lobby$ that assigns a positive probability to each element.
\begin{enumerate}[\quad (i)]
\item Choose $F$ according to $\mathbf{Q}$ and sample $\omega=\omega'$ on $F^c$ according to $\P$.
\item Conditioned on the outcome in (i), sample $\omega$ and $\omega'$ on $F$ independently from the conditional law $\P(\,\cdot\,|\,\omega_{F^c})$.
\end{enumerate}
Denote the resulting measure on $\Omega_1\times\Omega_1$ by $\P'$. Its marginals of course coincide with $\P$.

\begin{prop}\label{prop:modification}
Let $G$ be of the form $G_\alpha^{cw}$ or $G_\alpha^{ccw}$ for some $\alpha\in\labels_\star$. Then, $\P'$-almost surely, we have
\begin{enumerate}[\quad (a)]
\item $G(v)(\omega)=G(v)(\omega')$ for all but finitely many $v\in G(\omega)$;
\item either $\omega_e>\omega'_e$ for some edge $e$, or $G(v)(\omega)=G(v)(\omega')$ for all $v\in G(\omega)$ such that $\omega_e=\omega'_e$ for all $e\in G(v)(\omega)$.
\end{enumerate}
\end{prop}

The proof of the proposition will roughly amount to showing that: \emph{(i)} the tail of $G(\omega)$ is a geodesic also in $\omega'$, and since it is a geodesic in $\omega'$ it has a label also in $\omega'$; \emph{(ii)} the tail of every geodesic with label $\alpha$ in $\omega$ has label $\alpha$ also in $\omega'$; \emph{(iii)} if $G(\omega)$ was the cw- or ccw-most geodesic with label $\alpha$ in $\omega$, then its tail will be so too in $\omega'$.

To show that the tail of every geodesic in $\tree_0(\omega)$ is a geodesic also in $\omega'$ is straightforward. However, as we shall need a somewhat stronger statement (Lemma~\ref{lma:equal trees} below), a first step in the above outline will be to obtain an analogous statement for random coalescing geodesics.

\begin{lemma}\label{coalescing lakes}
For every random coalescing geodesic $G$ we have that
$$
\P'\big(\text{$G(\omega)$ and $G(\omega')$ coalesce}\big)=1.
$$
\end{lemma}

\begin{proof}
Let $H=H_i$ be a half-plane that $G$ eventually moves into. We define another measure $\P''$ on $\Omega_1\times\Omega_1$ as follows: Sample $(\omega,\omega'')\in\Omega_1\times\Omega_1$ by first sampling $\omega=\omega''$ for all edges with both endpoints in $H$ according to $\P$, and second sample $\omega$ and $\omega''$ independently on $H^c$ according to the conditional measure $\P(\,\cdot\,|\omega_H)$. We first claim that
\begin{equation}\label{half-plane resample}
\P''\big(\text{$G(\omega)$ and $G(\omega'')$ coalesce}\big)=1.
\end{equation}

Assume this is false, so that with positive probability they do not coalesce. $G$ induces a measure on coalescing families of half-plane geodesics via its half-plane restriction $G_H$ defined in~\eqref{eq:GH}. Since $\omega$ and $\omega''$ are equal on $H$, the set of half-plane geodesics coincide in both $\omega$ and $\omega''$. Therefore, assuming that~\eqref{half-plane resample} is false, the events that $G(\omega)$ lies strictly clockwise of $G(\omega'')$ and $G(\omega)$ lies strictly counterclockwise of $G(\omega'')$ each have to occur with positive probability.
By the ergodic theorem we find, almost surely, boundary sites $u$ and $v$ of $H$ such that $G(u)(\omega)$ is clockwise of $G(u)(\omega'')$ and $G(v)(\omega)$ is counterclockwise of $G(u)(\omega'')$. That contradicts that $G$ is almost surely coalescing, proving that~\eqref{half-plane resample} holds.

Let $A(N)$ be the event that $\omega$ and $\omega'$ agree outside $[-N,N]^2$. Since $\mathbf{Q}$ is supported on finite sets, we may for every $\eps>0$ find $N$ large so that $\P'(A(N))>1-\eps$. We then pick $v$ outside $[-N,N]^2$ so that $H+v$ does not intersect $[-N,N]^2$. Since $G$ is coalescing we have for $\P'$-almost every $(\omega,\omega')$ that both $G(\omega)$ and $G(v)(\omega)$, and $G(\omega')$ and $G(v)(\omega')$, coalesce. Write $G_H(v)$ as short for $\sigma_{-v}\circ G_H\circ\sigma_v$. By Lemma~\ref{valpolicella} we have that $G(\omega)$ and $G_H(v)(\omega)$ coalesce, and so does $G(\omega')$ and $G_H(v)(\omega')$, almost surely.
However, by~\eqref{half-plane resample}, $G_H(\omega)$ depends on $\omega$ only through its restriction to $H$. So on the event $A(N)$, which occurs with probability at least $1-\eps$, we have $G_H(v)(\omega)=G_H(v)(\omega')$ almost surely, and $G(\omega)$ and $G(\omega')$ coalesce. Since $\eps>0$ was arbitrary, the lemma follows.
\end{proof}

Let $\rho_i$ denote the linear functional describing the Busemann function associated to $\Gamma_i$. Recall the definition of $\arc(\rho_i)$, and pick two half-planes $H$ and $H'$, both containing the origin as a boundary point, such that $H$ contains $\arc(\rho_0)$ and $H'$ contains $\arc(\rho_1)$ in their interior. If the shape is neither a diamond or a square, then $H$ and $H'$ may be chosen so that both $\arc(\rho_0)$ and $\arc(\rho_1)$ are contained in the interior of both half-planes. If the shape is a diamond or a square, we may then assume that the sector $H\cap H'$ is symmetric around an axis or diagonal and spans an angle of at least $3\pi/4$, and therefore contains points of both $\dir(\Gamma_0)$ and $\dir(\Gamma_1)$ in its intereior. As a consequence, it follows that every geodesic $g\in\tree_0$ that is contained counterclockwise between $\Gamma_0$ and $\Gamma_1$ will visit infinitely many points in $H\cap H'+u$, for all $u\in\Z^2$, almost surely.

\begin{lemma}\label{lma:equal trees}
$\P'$-almost surely, there exists $u\in\Z^2$ such that for every $v\in H\cap H'+u$
\begin{itemize}
\item $\Gamma_0(v)(\omega)=\Gamma_0(v)(\omega')$ and $\Gamma_1(v)(\omega)=\Gamma_1(v)(\omega')$;
\item $\omega$ and $\omega'$ agree on the cone counterclockwise between $\Gamma_0(v)$ and $\Gamma_1(v)$.
\end{itemize}
\end{lemma}

\begin{proof}
Let $G$ be a random coalescing geodesic.
We first claim that for every $\eps>0$ we have, almost surely, that for all but finitely many $z\in\Z^2$ and all $y$ so that $z+y\in G(z)$ either
\begin{equation}\label{extended direction}
|y|\le\eps|z|\quad\text{or}\quad\textup{dist}\big(y/|y|,\arc(\rho_G)\big)\le\eps,
\end{equation}
where $\textup{dist}(x,S):=\inf\{|x-y|:y\in S\}$.

By Lemma~\ref{rioja} we have, with probability one, that $B_G(z,z+y)=T(z,z+y)$ for all $z$ and $y$ such that $z+y\in G(z)$. Proposition~\ref{extended shape} and~\eqref{extended Busemann} together imply that for all but finitely many $z$ and all $y$ such that $z+y\in G(z)$ we have
$$
0\le\mu(y)-\rho_G(y)=\mu(y)-T(z,z+y)+B_G(z,z+y)-\rho_G(y)\le2\eps\max\{|z|,|y|\},
$$
almost surely. So, almost surely, for every $y$ such that $z+y\in G(z)$ we have that either $|y|<\sqrt{\eps}|z|$, or $\mu(y)-\rho_G(y)\le2\sqrt{\eps}|y|$. Since $\eps>0$ was arbitrary, and since the shape is convex, the claim follows.

Fix $\eps$ so that $\arc(\rho_0)$ and $\arc(\rho_1)$ are at distance $2\eps$ from the boundary of $H$ and $H'$, respectively. Since $\mathbf{Q}$ is supported on finite sets, we may pick an almost surely finite $N\ge1$ so that $\omega$ and $\omega'$ agree outside $[-N,N]^2$. We then pick $u$ so that $H\cup H'+u$ does not intersect $[-2N,2N]^2$, and so that~\eqref{extended direction} holds (in both $\omega$ and $\omega'$) for all $z$ in $H\cap H'+u$ and all $y$ such that $z+y\in G(z)$, for both $G=\Gamma_0$ and $G=\Gamma_1$. Then, for no $z$ in $H\cap H'+u$ does $\Gamma_0(z)$ or $\Gamma_1(z)$ visit $[-N,N]^2$. By Lemma~\ref{coalescing lakes} both $\Gamma_0(z)(\omega)$ and $\Gamma_0(z)(\omega')$, and $\Gamma_1(z)(\omega)$ and $\Gamma_1(z)(\omega')$, will coalesce. Due to unique passage times, they coincide.
\end{proof}

We next examine the effect of a local modification on the labeling.
Let $A_1(v)$ denote the event in Lemma~\ref{lma:equal trees}, and note that on $A_1(v)$ the environments counterclockwise between $\Gamma_0(v)$ and $\Gamma_1(v)$ agree, and consequently will the restrictions of $\tree_v$ to the same region. Since every geodesic in $\tree_0(\omega)$ counterclockwise of $\Gamma_0(\omega)$ and $\Gamma_1(\omega)$ visits $H\cap H'+u$, for every $u\in\Z^2$, it follows that their tails are geodesics also in $\omega'$. In fact, for every infinite geodesic in $\omega$ we obtain a geodesic in $\omega'$ by chopping off a finite segment. For a geodesic $g\in\tree_0(\omega)$ we write $F_{\omega'}(g)$ for the label (in $\omega'$) of the portion of $g$ that is a geodesic also in $\omega'$.

\begin{lemma}\label{new lakes}
Let $G$ be of the form $G_\alpha^{cw}$ or $G_\alpha^{ccw}$ for some $\alpha\in\labels_\star$. Then,
$$
\P'\big(F_{\omega'}(G(\omega))=\alpha\big)=1.
$$
\end{lemma}

\begin{proof}
Fix $\alpha\in\labels_\star$. Since random non-crossing geodesics have constant labels we may without restriction, due to Lemma~\ref{ccw-dense}, assume that $\Gamma_0\le G_\alpha^{cw}\le G_\alpha^{ccw}\le\Gamma_1$. Since $G_\alpha^{cw}$ and $G_\alpha^{ccw}$ have label $\alpha$, it follows that every geodesic in $\tree_0$ with label $\alpha$ is contained counterclockwise between $\Gamma_0$ and $\Gamma_1$ almost surely. Moreover, since $\Gamma_0(\omega)$ and $\Gamma_0(\omega')$, and $\Gamma_1(\omega)$ and $\Gamma_1(\omega')$, coalesce, the tail of every geodesic in $\tree_0(\omega)$ is a geodesic in $\omega'$ with label (in $\omega'$) bounded between $F(\Gamma_0)$ and $F(\Gamma_1)$.

By current assumptions, either $\Gamma_0=G_\alpha^{cw}$ or there is $\alpha_-\in\labels_\star$ such that $\alpha>\alpha_-\ge F(\Gamma_0)$. In the former case we let $G^-=G_\alpha^{cw}$ and in the latter we set $G^-=G_{\alpha_-}^{ccw}$. Similarly, either $\Gamma_1=G_\alpha^{ccw}$ or there is $\alpha_+\in\labels_\star$ such that $\alpha<\alpha_+\le F(\Gamma_1)$. Let $G^+=G_\alpha^{ccw}$ in the former case and set $G^+=G_{\alpha_+}^{cw}$ otherwise. Set $\mathcal{F}=\{G^-,G_\alpha^{cw},G_\alpha^{ccw},G^+\}$.

Let $A_1(v)$ denote the event in Lemma~\ref{lma:equal trees}.
For $G\in\mathcal{F}$, let $G(v,n)$ denote the first $n$ steps of $G(v)$. Define $\pre(G,v,n,N)$ to be the maximum likelihood estimate of $G(v,n)$ based on the restriction of the edge configuration to $v+[-N,N]^2$. Let $A_2(\mathcal{F},v,n,N)$ denote the event that in both $\omega$ and $\omega'$
\begin{itemize}
\item $G^-(v,n)$ and $G_\alpha^{cw}(v,n)$ are distinct (assuming $G^-\neq G_\alpha^{cw}$);
\item $G_\alpha^{ccw}(v,n)$ and $G^+(v,n)$ are distinct (assuming $G^+\neq G_\alpha^{ccw}$);
\item $\pre(G,v,n,N)=G(v,n)$ for each $G\in\mathcal{F}$.
\end{itemize}
By first selecting $n$ large so that both $G^-$ and $G_\alpha^{cw}$, and $G_\alpha^{ccw}$ and $G^+$, diverge within $n$ steps (assuming they do not coincide), and later selecting $N\ge n$ so the maximum likelihood estimates are all accurate, we can make the probability of $A_2(\mathcal{F},v,n,N)$ arbitrarily close to 1, uniformly in $v$.

Fix $n$ and $N\ge n$ so that $\P'(A_2(\mathcal{F},v,n,N))\ge1/2$ uniformly in $v$. Pick $N'\ge N$ so that $\omega$ and $\omega'$ agree outside $[-N',N']^2$, and $M\ge N'$ so that the segments of $G_\alpha^{cw}(\omega)$ and $G_\alpha^{ccw}(\omega)$ outside $[-M,M]^2$ are geodesics also in $\omega'$. Next pick $u$ so that $H\cup H'+u$ does not intersect $[-2M,2M]^2$ and so that $A_1(v)$ occurs for all $v\in H\cap H'+u$. Finally pick $v\in H\cap H'+u$ so that $A_2(\mathcal{F},v,n,N)$ occurs (see Figure~\ref{fig:coal0,5})

\begin{figure}[htbp]
\begin{center}
\includegraphics{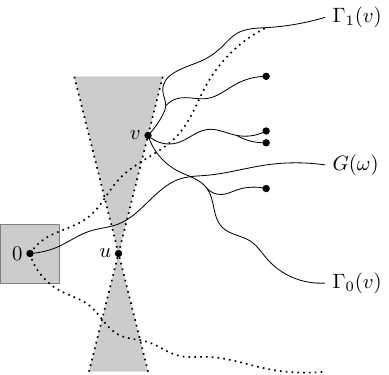}
\end{center}
\caption{The two configurations agree outside the shaded box. For $v$ suitably chosen, $G(\omega)$ will be comparable to the stubs.}
\label{fig:coal0,5}
\end{figure}

We first argue that $F_{\omega'}(G_\alpha^{cw}(\omega))\ge\alpha$ almost surely. In case $\Gamma_0$ has label $\alpha$, then there is nothing to prove, so suppose that $F(\Gamma_0)\le\alpha_-<\alpha$.
Since $A_1(v)$ occurs we have that $\Gamma_0(v)$ and $\Gamma_1(v)$ coincide in both $\omega$ and $\omega'$, and so does the geodesic structure in the cone counterclockwise between the two. Moreover, the predictions are accurate and do only depend on the edge configuration outside $[-N,N]^2$, where $\omega$ and $\omega'$ coincide. It follows that $G(v,n)(\omega)=G(v,n)(\omega')$ for all $G\in\mathcal{F}$ and, since $G^-\neq G_\alpha^{cw}$, the stubs $G^-(v,n)$ and $G_\alpha^{cw}(v,n)$ are distinct.

We claim that $F_{\omega'}(G_\alpha^{cw}(\omega))>\alpha_-$. To see this we first note that every geodesic in $\tree_v$ containing $G^-(v,n)$ has label (in both $\omega$ and $\omega'$) strictly smaller than $\alpha$. Consequently $G_\alpha^{cw}(\omega)$ must lie counterclockwise of each of them, and in particular $G_\alpha^{cw}(\omega)>G^-(v)(\omega')$. If $G_\alpha^{cw}(\omega)$ and $\Gamma_0(v)$ cross (see Figure~\ref{fig:coal0,5}), this implies that the label of $G_\alpha^{cw}(\omega)$ in $\omega'$ is strictly larger than $\alpha_-$. If, on the other hand, $G_\alpha^{cw}(\omega)$ and $\Gamma_1(v)$ intersect, then $G_\alpha^{cw}(\omega)$ cannot lie clockwise of all geodesics in $\tree_v$ that contain $G_\alpha^{cw}(v,n)$, since it is non-crossing. In particular, since all these geodesics have label strictly larger than $\alpha_-$ (in both $\omega$ and $\omega'$), so will $G_\alpha^{cw}(\omega)$. 

Since $\alpha_-<\alpha$ was arbitrary, we conclude that $\P'$-almost surely $F_{\omega'}(G_\alpha^{cw}(\omega))\ge\alpha$. An analogous argument shows that $\P'$-almost surely $F_{\omega'}(G_\alpha^{ccw}(\omega))\le\alpha$, which completes the proof.
\end{proof}

\begin{lemma}\label{rng limits}
For every $\alpha\in\labels_\star$ we have, almost surely, that for every $v\in\Z^2$ and $g=(v_1,v_2,\ldots)$ in $\tree_v$ with label $\alpha$ that
\begin{itemize}
\item if $G_\alpha^{cw}\ge g$, then $\lim_{k\to\infty}\geo(0,v_k)=G_\alpha^{cw}$;
\item if $G_\alpha^{ccw}\le g$, then $\lim_{k\to\infty}\geo(0,v_k)=G_\alpha^{ccw}$.
\end{itemize}
\end{lemma}

\begin{proof}
Since the two cases are similar, we only prove the first. We may assume that $\Gamma_0\le G_\alpha^{cw}\le\Gamma_1$ almost surely. So, suppose that $g=(v_1,v_2,\ldots)$ is a geodesic with label $\alpha$ and satisfies $g\le G_\alpha^{cw}$. We then have $\Gamma_0\le g\le G_\alpha^{cw}$, and there is no restriction in assuming that $g$ is entirely contained in the cone counterclockwise between $\Gamma_0$ and $G_\alpha^{cw}$. Due to Proposition~\ref{cycle} we may thus pick $u$ on $\Gamma_0(v)$ such that $G_\alpha^{cw}(u)$ is clockwise of $G_\alpha^{cw}(v)$. Moreover, we may assume that $u\in\Gamma_0$, in which case $G_\alpha^{cw}(u)$ is also clockwise of $G_\alpha^{cw}$.

Let $g'$ denote any subsequential limit of the sequence $(\geo(0,v_k))_{k\ge1}$. It then follows that $G_\alpha^{cw}(u)\le g'\le G_\alpha^{cw}$. Then $g'$ has label $\alpha$, and therefore coincide with $G_\alpha^{cw}$, as required.
\end{proof}

\begin{proof}[Proof of Proposition~\ref{prop:modification}]
To complete the proof of the proposition, fix $\alpha\in\labels_\star$ and let $G$ denote either of $G_\alpha^{cw}$ and $G_\alpha^{ccw}$. For the first statement, since $G$ visits infinitely many points in $H\cap H'+u$, for every $u$ almost surely, it follows by Lemma~\ref{lma:equal trees} that, almost surely, we find $v\in G(\omega)$ so that $A_1(v)$ occurs. On the event $A_1(v)$, $G(v)(\omega)$ is a geodesic in both $\omega$ and $\omega'$, and the restriction of $\tree_v$ to the cone counterclockwise between $\Gamma_0(v)$ and $\Gamma_1(v)$ coincide in $\omega$ and $\omega'$. By Lemma~\ref{new lakes} it has label $\alpha$ also in $\omega'$. Assume that $G(v)(\omega)$ is the cw-most geodesic with label $\alpha$ in $\tree_v(\omega)$, and that there is another geodesic $g\in\tree_v(\omega')$ clockwise of $G(v)(\omega)$ with label $\alpha$. Then $g$ is a geodesic also in $\tree_v(\omega)$, and by Lemma~\ref{new lakes} has label $\alpha$, contradicting that $G(v)(\omega)$ were the cw-most one. An analogous argument shows that if $G(v)(\omega)$ instead were the ccw-most geodesic in $\tree_v(\omega)$ with label $\alpha$, then it will be the ccw-most also in $\tree_v(\omega')$. In particular $G(v)(\omega)=G(v)(\omega')$ for some $v\in G(\omega)$, and thus for all subsequent sites along the path, almost surely, thus proving the first statement.

We proceed with the second claim of the proposition. On the event that $\omega\le\omega'$ coordinate wise and the two coincide along $G(\omega)$, it follows that $G(\omega)$ is also a geodesic in $\omega'$. By Lemma~\ref{new lakes} it has label $\alpha$. Suppose that $G(\omega)\neq G(\omega')$. In case $G$ denotes $G_\alpha^{cw}$ this means that $G(\omega')<G(\omega)$, and in case $G$ denotes $G_\alpha^{ccw}$ this means that $G(\omega')>G(\omega)$. In either case the tail of $G(\omega')$ is a geodesic with label $\alpha$ also in $\omega$. By Lemma~\ref{rng limits}, taking limits (in $\omega$) along $G(\omega')$ gives $G(\omega)$. Since $\omega\le\omega'$ coordinate wise, this means that there exists a faster path to far out vertices on $G(\omega')$ than using $G(\omega')$ itself, contradicting the assumption that $G(\omega)\neq G(\omega')$.
\end{proof}

\subsection{Half-plane geodesics are coalescing}

We next show that random non-crossing geodesics that eventually move into some half-plane are coalescing.

\begin{prop} \label{picnic table}
Every geodesic of the form $G_\alpha^{cw}$ or $G_\alpha^{ccw}$, for some $\alpha\in\labels_\star$, is coalescing.
\end{prop}

Our proof of Proposition~\ref{picnic table} will be adapted from an argument due to Licea and Newman~\cite{newman95,licnew96}; see also~\cite{damhan14}.
The core of the proof is a local modification argument that shows that if $G$ is not coalescing, then there exists a coalescence class with a least element according to some ordering of the elements in $\Z^2$. Our first lemma says that this cannot happen.

\begin{lemma} \label{photo id}
Let $G$ be a random non-crossing geodesic and let $R$ be some shift invariant order on $\Z^2$. Then no coalescence class for $G$ has a least element for $R$, almost surely.
\end{lemma}

\begin{proof}
This lemma is an application of the mass-transport principle, see~\cite[Theorem~5.3]{hagjon06}. We define a function $m(x,y)$ which is one if $y$ is the least element in the coalescence class of $x$, and otherwise zero.
If $y$ is the least element of its coalescence class then it gets infinite mass as all the points on $G(y)$ have $m(x,y)=1$. If the probability that $y$ is the least element for some class is positive, then the expected amount of mass into $y$ ($\E \sum_{x}m(x,y)$) is infinite while the expected amount of mass out of $y$
($\E \sum_{x}m(y,x)$) is at most one. This is a contradiction to the mass-transport principle, which holds these quantities equal. Thus the probability that $y$ is the least element of its coalescence class is zero, so by the union bound the probability that there exists a coalescence class which has a least element is zero.
\end{proof}




\begin{proof}[Proof of Proposition~\ref{picnic table}]
Fix $\alpha\in\labels_\star$ and let $G$ denote either $G_\alpha^{cw}$ or $G_\alpha^{ccw}$. By Proposition~\ref{riders on the storm} it follows that $G$ eventually moves into some half-plane $H_i$.
For notational simplicity we shall assume that $G$ eventually moves into the right half-plane, and note that any other case can be treated in an analogous manner.

We will argue by contradiction and assume that $G$ is not coalescing with positive probability. Given $i,j\in\Z$, let $A(i,j)$ denote the event that $G(i{\bf e}_2)$ and $G(j{\bf e}_2)$ are entirely contained in the right half-plane and belong to different coalescence classes. Since $G$ eventually moves into the right half-plane and is assumed to be non-coalescing (with positive probability), the event $A(i,j)$ has positive probability to occur for some $i<j$. An application of the ergodic theorem assures that $A(i,i+k)$ will occur for some $k\ge1$ and infinitely many $i\in\Z$ with probability one. Consequently, there are infinitely many coalescence classes for $G$ with probability one.

The remainder of the proof will be divided into two different cases, depending on whether the support of the weight distribution is unbounded or not. We start with the simpler case of unbounded support.\\

\emph{Case 1: Unbounded support.}
Given $m\ge1$, let $R_m$ denote the rectangle $\{(i,j)\in\Z^2:i\in\{0,-1\},0\le j\le m\}$. We partition the set of edges with both their endpoints in $R_m$ into \emph{interior} edges $\textup{int}(R_m)$, connecting $(-1,j)$ to $(0,j)$ for some $0<j<m$, and \emph{boundary} edges $\partial R_m$, being the remaining ones. For $m\ge1$, $t>0$ and $0<j_1<j_2<j_3<m$, let $A=A(m,t,j_1,j_2,j_3)$ denote the event that (see Figure~\ref{fig:coal1})
\begin{itemize}
\item $G(j_1{\bf e}_2)$, $G(j_2{\bf e}_2)$ and $G(j_3{\bf e}_2)$ belong to different coalescence classes;
\item $G(j_1{\bf e}_2)$, $G(j_2{\bf e}_2)$ and $G(j_3{\bf e}_2)$ are entirely contained in the right half-plane;
\item every boundary edge in $R_m$ has weight at most $t$.
\end{itemize}
\begin{figure}[htbp]
\begin{center}
\includegraphics{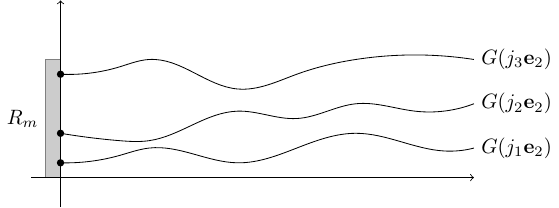}
\end{center}
\caption{The rectangle $R_m$ (shaded area) see three coalescence classes represented.}
\label{fig:coal1}
\end{figure}
It is clear that $A$ has positive probability for some choice of $m$, $j_1$, $j_2$ and $j_3$, and all $t$ large enough, since for all $m$ sufficiently large three coalescence classes will intersect $R_m$, and by making $m$ larger, since $G$ eventually moves into the right half-plane, some representative of each class will visit the vertical axis only inside $R_m$. Then choose $t$ large.

Let $A'$ denote the event that $A$ occurs and that all interior edges in $R_m$ have weight larger than $2mt$. We claim that also $\P(A')>0$ for this choice of parameters. To see this, note that any geodesic that avoids the interior edges of $R_m$ will remain a geodesic after increasing the weight on interior edges of $R_m$. Moreover, it follows from Proposition~\ref{prop:modification} that the images of $G$ at the points $j_1{\bf e}_2$, $j_2{\bf e}_2$ and $j_3{\bf e}_2$ do not change as we increase the weight of the interior edges. In terms of the local modification coupling, the last statement can be formalized as
$$
\P'\Big(\omega\in A,\omega_e\le s<\omega'_e\text{ for }e\in F,F=\textup{int}(R_m)\Big)\le\P'\Big(\omega'\in A,\omega'_e>s:e\in\textup{int}(R_m)\Big)
$$
for all $s>0$. The left-hand side in the above expression can be decomposed as
$$
\E'\Big[\P'\big(\omega\in A,\omega_e\le s\text{ for }e\in F\,\big|\,\omega_{F^c}\big)\P'\big(\omega'_e>s\text{ for }e\in F\,\big|\,\omega_{F^c}\big){\bf 1}_{\{F=\textup{int}(R_m)\}}\Big].
$$
Since $\mathbf{Q}$ has full support, the last expression is strictly positive for all $s$ sufficiently large, due to the upward finite energy property. So, by monotonicity in $s$, we obtain that
\begin{equation}\label{A primula}
\P(A')=\P\big(\omega\in A,\omega_e>2mt\text{ for }e\in\textup{int}(R_m)\big)>0.
\end{equation}

However, we note that on the event $A'$ there are three coalescence classes of $G$ represented in $R_m$, but that no vertex in the left half-plane belongs to the middle of the three. This is a contradiction to Lemma~\ref{photo id}, which shows that our initial assumption, that $G$ has positive probability not to be coalescing, must have been wrong.\\

\emph{Case 2: Bounded support.}
We proceed with the case that the edge weight distribution has bounded support, that is, that
$$
t_1:=\sup\big\{t\ge0:\P(\omega_e>t)>0\big\}
$$
is finite. We will in this case need to replace the rectangle $R_m$ by a larger randomly selected region, delimited by two vertical lines and two finite geodesics obtained from $G$.



Fix $\eps=(t_1-\E[\omega_e])/6$. Given $m\ge1$, $0<j_1<j_2<j_3<\eps m/t_1$, and $-m<j_1'<j_3'<m$ with $j_3'-j_1'<\eps m/t_1$, let $A=A(m,j_1,j_2,j_3,j_1',j_3')$ denote the event that (see Figure~\ref{fig:coal2})
\begin{itemize}
\item $G(j_1{\bf e}_2)$, $G(j_2{\bf e}_2)$ and $G(j_3{\bf e}_2)$ belong to different coalescence classes;
\item $G(j_1{\bf e}_2)$, $G(j_2{\bf e}_2)$ and $G(j_3{\bf e}_2)$ are entirely contained in the right half-plane;
\item the top- respectively bottom-most intersections of $G(j_1{\bf e}_2)$ and $G(j_3{\bf e}_2)$ with the vertical line $\{(m,j):j\in\Z\}$ occur at $(m,j_1')$ and $(m,j_3')$;
\item $T\big((0,j_1),(m,j_1')\big)\le(\E[\omega_e]+\eps)\|(m,j_1'-j_1)\|_1$.
\end{itemize}
\begin{figure}[htbp]
\begin{center}
\includegraphics{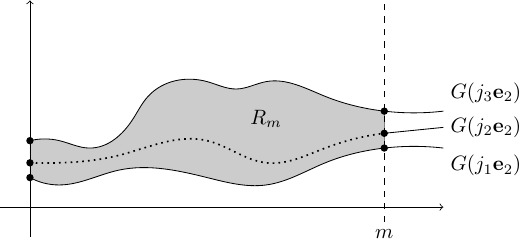}
\end{center}
\caption{Three coalescence classes represented on the vertical axis, and touch the vertical line $\{(x,y):y=m\}$ within distance $\eps m/t_1$ from each other. The shaded region is the one to be resampled.}
\label{fig:coal2}
\end{figure}

We argue that $A$ has positive probability to occur for some set of parameters. First, observe that by making $m$ large we can make sure to find three coalescence classes represented on the vertical axis within $0$ and $\eps m/t_1$, and that some representative in each class is entirely contained in the right half-plane. Second, using the shape theorem\footnote{Note that the time constant satisfies $\mu(z)\le\E[\omega_e]\|z\|_1$.} we increase $m$ so that the fourth condition is met. At last, we note that since $\dir(G)$ is almost surely constant, by Lemma~\ref{rng direction}, there are infinitely many $n$ for which $G(j_1{\bf e}_2)$ and $G(j_3{\bf e}_2)$ intersect $H_0+n{\bf e}_1$ within distance $\eps n/t_1$ of each other. By increasing $m$ further we conclude that $A$ has positive probability to occur for some set of parameters.

Denote by $R_m$ the region confined by the two vertical segments connecting $(0,j_1)$ with $(0,j_3)$ and $(m,j_1')$ with $(m,j_3')$, and the two segments of $G(j_1{\bf e}_2)$ and $G(j_3{\bf e}_2)$ connecting $(0,j_1)$ and $(0,j_3)$ with $(m,j_1')$ and $(m,j_3')$. While the event $A$ defined above is not increasing with respect to the set of interior edges in $R_m$, we may easily modify $A$ slightly to become increasing. Let $A'=A'(m,j_1,j_3,j_1',j_2',j_3')$, where $j_1'<j_2'<j_3'$, be the event that (recall Figure~\ref{fig:coal2})
\begin{itemize}
\item $G(j_1{\bf e}_2)$, $G((m,j_2'))$ and $G(j_3{\bf e}_2)$ belong to different coalescence classes;
\item $G(j_1{\bf e}_2)$ and $G(j_3{\bf e}_2)$ are entirely contained in the right half-plane, while $G((m,j_2'))$ is entirely contained in the half-plane $H_0+m{\bf e}_1$;
\item the top- respectively bottom-most intersections of $G(j_1{\bf e}_2)$ and $G(j_3{\bf e}_2)$ with the vertical line $\{(m,j):j\in\Z\}$ occur at $(m,j_1')$ and $(m,j_3')$;
\item $T\big((0,j_1),(m,j_1')\big)\le(\E[\omega_e]+\eps)\|(m,j_1'-j_1)\|_1$.
\end{itemize}

It is clear that also $A'$ has positive probability to occur for some set of parameters.
Let $A''$ denote the event that $A'$ occurs and that the weight of each interior edge in $R_m$ has weight at least $t_1-\eps$. Using Proposition~\ref{prop:modification}, an argument analogous to the one leading to~\eqref{A primula} shows that $\P(A'')>0$. (Although $R_m$ is random, we may fix some set $R$ such that $R_m=R$ with positive probability.)

On the event $A''$ there are three coalescence classes represented on the vertical segment $\{(m,j):|j|<m\}$, so our argument will be complete if we only can show that no vertex in the left half-plane is contained in the middle of the three, as that would be a contradiction to Lemma~\ref{photo id}. To see this, note that on $A'$ we have that every path $\pi$ between a point $(0,j)$ to a point $(m,j')$, where $j_1<j<j_3$ and $j_1'<j'<j_3'$, that does not touch neither $G(j_1{\bf e}_2)$ nor $G(j_3{\bf e}_2)$, has to pick up at least $\|(m,j'-j)\|_1$ weights of size at least $t_1-\eps$. Since $|j'-j|\ge|j_1'-j_1|-\eps m/t_1$, this amounts to a total weight $T(\pi)$ satisfying
$$
T(\pi)\,\ge\,(t_1-\eps)\|(m,j'-j)\|_1\,\ge\,(t_1-2\eps)\|(m,j_1'-j_1)\|_1.
$$
However, the path obtained by moving vertically from $(0,j)$ to $(0,j_1)$, then follow $G(j_1{\bf e}_2)$ to $(m,j_1')$, and finally move vertically to $(m,j')$, has weight at most
\begin{equation*}
T\big((0,j_1),(m,j_1')\big)+2\eps m\,\le\,(\E[\omega_e]+3\eps)\|(m,j_1'-j_1)\|_1\,=\,(t_1-3\eps)\|(m,j_1'-j_1)\|_1.
\end{equation*}
Hence, $\pi$ cannot be a geodesic, and the middle coalescence class will thus not contain any element in the left half-plane.
\end{proof}

\subsection{Proof of Theorem~\ref{thm:coalescence}}

We first argue that $G_\alpha^{cw}$ and $G_\alpha^{ccw}$ are random coalescing geodesics for each $\alpha\in\labels_\star$. By Theorem~\ref{thm:labels} they are random non-crossing geodesics. By Proposition~\ref{riders on the storm} there exists, for each of them, a half-plane $H_i$ such that it eventually moves into $H_i$, so by Proposition~\ref{picnic table} they are almost surely coalescing.

We now take a random coalescing geodesic $G$. $F(G)$ is an almost sure constant; let's call it $\alpha$. It follows by Lemma~\ref{ccw-dense} that either $G=G_\alpha^{cw}$ or $G=G_\alpha^{ccw}$ with probability one.

Next we fix $\alpha\in\labels_\star$ and argue that $\{G_\alpha^{cw}=G_\alpha^{ccw}\}$ is a 0-1 event. Suppose that the two differ with positive probability, so that by the ergodic theorem there is a density of points for which they do, almost surely. Since both are almost surely coalescing, we must have $G_\alpha^{cw}(v)\neq G_\alpha^{ccw}(v)$ for all $v\in\Z^2$, since the contrary would contradict the coalescence property. This proves~\emph{(a)}.

For~\emph{(b)} we note that $G_\alpha^{cw}$ is a random coalescing geodesic, which by Theorem~\ref{thm:labels} has label $\alpha$ almost surely. By Lemma~\ref{lma:multiplicity} it follows that, with probability one, there are at most two ccw-isolated geodesics in $\tree_0$ with label $\alpha$. Since the existence of three geodesics labeled $\alpha$ would imply the existence of (at least) two ccw-isolated geodesics,~\emph{(b)} follows.




\subsection{Shift-invariant measures on non-crossing geodesics}

In preparation for addressing the midpoint problem below we shall next investigate properties of `sparse' non-crossing geodesics. A measurable map $G:\Omega_1\to\Omega_2$ is a {\bf sparse random non-crossing geodesic} if for almost every $\omega\in\Omega_1$ either $G(\omega)\in\tree_0(\omega)$ or $G(\omega)\equiv0$, and for every $u,v\in\Z^2$ we have that
\begin{itemize}
\item if $G(u)\in\tree_u(\omega)$ and $G(v)\in\tree_v(\omega)$, then they are non-crossing;
\item if $u\in G(v)$ for some $v\in\Z^2$, then $G(u)\in\tree_u(\omega)$;
\end{itemize}
and $\P(G\in\tree_0)>0$. In particular, also a random non-crossing geodesic is a sparse random non-crossing geodesic. (As before we write $G(v)$ as short for $\sigma_{-v}\circ G\circ\sigma_v$.)

A sparse random non-crossing geodesic $G$ induces an equivalence relation on the set $V_G:=\{v\in\Z^2:G(v)\in\tree_v\}$ by declaring $u\sim v$ if $G(u)$ and $G(v)$ coalesce. We refer to the equivalence classes of $G$ as coalescence classes and call $G$ coalescing if there is almost surely a unique coalescence class.

\begin{prop}\label{p:sparse}
A sparse random non-crossing geodesic has at most four coalescence classes almost surely.
A sparse random non-crossing geodesic which is almost surely contained between $\Gamma_0$ and $\Gamma_1$ is coalescing.
\end{prop}

\begin{proof}
First, assume that $G$ is a sparse random non-crossing geodesic which is almost surely contained between $\Gamma_0$ and $\Gamma_1$. We first show that there exists $\alpha\in[0,1]$ such that either $G\equiv0$ or $F(G)=\alpha$. We argue as in the proof of Lemma~\ref{constant}: Suppose the contrary, that there exist disjoint intervals $[a,b]$ and $[c,d]$, where $b<c$, such that $F(G)$ has positive probability to be in either. Let $A(u)$ denote the event that $G(u)\in\tree_u$ and $F(G(u))\in[c,d]$, and let $A'(u)$ be the event that $G(u)\in\tree_u$ and $F(G(u))\in[a,b]$. By Proposition~\ref{cycle} we obtain, almost surely, $u\in\Gamma_0$ and $v\in\Gamma_1$ such that $A(u)$ and $A'(v)$ occur. However, this means that $G(u)$ and $G(v)$ must cross, which is a contradiction. We conclude that the label of $G$ is almost surely constant.
Second, we observe that if $\P(G=G_\alpha^{ccw})>0$, then $\P(G\equiv0\text{ or }G=G_\alpha^{ccw})=1$, since the contrary would contradict $G$ being non-crossing. Conversely, if $\P(G=G_\alpha^{cw})>0$, then $\P(G\equiv0\text{ or }G=G_\alpha^{cw})=1$. In either case it follows, by Theorem~\ref{thm:coalescence}, that $G$ has a unique coalescence class, as claimed.

Next, let $G$ be any sparse random non-crossing geodesic and let $I\subset\tree_0$ be of the form $\{g\in\tree_0:\Gamma_i\le g<\Gamma_{i+1}\}$ for some $i=0,1,2,3$. There are four sets of this form and together they form a partition of $\tree_0$. For each $I$ of this form we define $G_I:\Omega_1\to\Omega_2$ as
\begin{equation*}
G_I(\omega):=\left\{
\begin{aligned}
& G(\omega) & \text{if }G(\omega)\in I,\\
&0 & \text{otherwise}.
\end{aligned}
\right.
\end{equation*}
In case $\P(G\in I)>0$, then $G_I$ is also a sparse random non-crossing geodesic. Then $G_I$ is coalescing by the first part of the proposition. Hence, there may be at most four coalescence classes almost surely.
\end{proof}

A similar statement holds for shift-invariant measures on families of non-crossing geodesics. Let $\nu$ be a probability measure on $\Omega_1\times\Omega_2'$ whose restriction to $\Omega_1$ coincides with $\P$ and whose restriction to $\Omega_2'$ does not put all mass on the all-zero configuration. We shall call $\nu$ a {\bf shift-invariant measure on non-crossing geodesics} if for $\nu$-almost every $(\omega,\eta)$ the graph encoded by $\eta$ has the following properties:
\begin{itemize}
\item each site has either degree zero or out-degree 1;
\item for each site with out-degree 1 the unique infinite forward-path is a geodesic;
\item any two forward-paths are either disjoint or coalesce.
\end{itemize}


\begin{prop}\label{p:sparse measure}
Every shift-invariant measure on non-crossing geodesics is supported on families of geodesics with at most four coalescence classes.
\end{prop}

\begin{proof}
First we decompose $\nu$ into a linear combination of four measures $\nu_0$, $\nu_1$, $\nu_2$ and $\nu_3$ so that $\nu_i$ is supported on geodesics in $\{g\in\tree_0:\Gamma_i\le g<\Gamma_{i+1}\}$. As in Proposition~\ref{p:sparse}, we first show that for $\nu_i$-almost every $(\omega,\eta)$ all geodesics encoded in $\eta$ have equal labels.

Suppose the contrary, that there exists $v$ and intervals $[a,b]$ and $[c,d]$, where $b<c$ such that with positive probability the forward-paths at the origin and $v$ have labels in $[a,b]$ and $[c,d]$, respectively. As in Proposition~\ref{cycle} we then obtain, with positive probability, the existence of $u\in\Gamma_0$ and $v\in\Gamma_1$ such that the forward-path at $u$ has label in $[c,d]$ and the forward-path at $v$ has label in $[a,b]$. However, due to the consistency of the ordering and the labeling, Proposition~\ref{label props}, the two paths must cross, contradicting that $\nu$ is supported on non-crossing families.

Since $\nu_i$ is supported on configurations of geodesics whose labels coincide, $\nu_i$ induces a measure $\pi_i$ on $\Omega_1\times[0,1]$. Through conditional expectation we obtain a measure $\hat\pi_i(\omega)$ on $[0,1]$ which is invariant under the shift maps, i.e.\ $\hat\pi_i(\omega)=\hat\pi_i(\shiftv\omega)$ for almost every $\omega\in\Omega_1$ and every $v\in\Z^2$. By ergodicity of the shift maps we conclude that there exists a measure $\bar\pi_i$ such that $\hat\pi_i(\omega)=\bar\pi_i$ for almost every $\omega\in\Omega_1$, and thus that $\pi_i=\P\times\bar\pi_i$.

Let $M$ denote the set of points $(\omega,\alpha)\in\Omega_1\times[0,1]$ for which
there are at least three geodesics in $\tree_0(\omega)$ with label $\alpha$. Let $C$ denote the set of points $(\omega,\alpha)$ for which there exists $u$ and $v$ in $\Z^2$ such that $G_\alpha^{cw}(u)(\omega)$ and $G_\alpha^{cw}(v)(\omega)$, or $G_\alpha^{ccw}(u)(\omega)$ and $G_\alpha^{ccw}(v)(\omega)$, do not coalesce. Due to the product structure of $\pi_i$ it follows from Fubini's theorem and Theorem~\ref{thm:coalescence} that $\pi_i(M\cup C)=0$. 

Let $A$ denote the event that there exists $\alpha\in\labels_\star$ such that for every $v\in\Z^2$ the geodesic encoded by $\eta$ and starting at $v$ (if there is one) coincides with either $G_\alpha^{cw}(\omega)$ or $G_\alpha^{ccw}(\omega)$. Since $\pi_i(M)=0$ we conclude that $\nu_i(A)=1$. Since $\pi_i(C)=0$, $\nu_i(A)=1$ and $\nu_i$ is supported on non-crossing families of geodesics, we conclude that, with $\nu_i$-probability one, there exists $G\in\rcg$ such that for every $v\in\Z^2$ the geodesic encoded by $\eta$ and starting at $v$ (if there is one) coincides with $G(\omega)$. Finally, since $\pi_i(C)=0$, it follows that $\nu_i$ is supported on coalescing families of geodesics.
\end{proof}

\section{An ergodic theory for infinite geodesics}
\label{other}

In this section we investigate some of the consequences of the theory that we have built in the previous sections. We shall discuss cardinality of infinite geodesics, existence of bigeodesics, asymptotic directions and Busemann functions of infinite geodesics. Together the results of this section will provide a detailed description of the asymptotic properties of infinite geodesics that prove, and in several aspects go beyond, what has been announced in Theorems~\ref{versailles}--\ref{unique support}.

We saw in Section~\ref{sec:coalescence} that the class of random coalescing geodesics coincides with the family of random geodesics of the form $G_\alpha^{cw}$ and $G_\alpha^{ccw}$ for $\alpha\in\labels_\star$. Moreover, we saw in Section~\ref{label of sections} that the set of labels we observe almost surely coincides with $\labels_\star$, and in Section~\ref{cola} that the probability of observing more than two geodesics with a given label is zero. Hence, if $\labels_\star$ is at most countable, then there are almost surely no geodesics in $\tree_0$ that are not the image of a random coalescing geodesic, and the results from Section~\ref{properties} give a precise description of the asymptotic properties of $\tree_0$.

However, the usual predictions (such as strict convexity of $\ball$) suggest that $\labels_\star$ is uncountable in the i.i.d.\ setting. This would imply that there may (and will, recall footnote~\ref{prediction}) be `exceptional' labels attributed to more geodesics than expected. In this case we will have to work harder in order to cover these `exceptional' geodesics. Indeed, we will show, almost surely, that (i) every geodesic has an asymptotically linear Busemann function; (ii) the corresponding set of linear functionals is a deterministic closed set; and (iii) for any given functional supporting to $\ball$ there is at most one geodesic with Busemann function asymptotically linear to that functional. That is, the `exceptional' geodesics are exceptional as they appear randomly in $\tree_0$.

Recall, from Theorem~\ref{thm:coalescence}, that the occurrence of multiple geodesics with a given label is a 0-1 event. This motivates the notation
$$
\double_\star:=\big\{\alpha\in[0,1]:\P(\exists\text{ at least two geodesics in $\tree_0$ with label }\alpha)=1\big\}.
$$

\subsection{Cardinality of infinite geodesics}

As a first consequence of the theory we have built we address the question of the number of topological ends of the geodesic tree, i.e., the graph encoded by $\tree_0$. The following is a reformulation of Theorem~\ref{geo cardinality}.

\begin{theorem}
For almost every $\omega\in\Omega_1$ the cardinality of $\tree_0(\omega)$ equals that of $\labels_\star$.
\end{theorem}

\begin{proof}
By Theorem~\ref{thm:labels} we have $\labels=\labels_\star$ almost surely, so the labeling provides an injection from $\labels_\star$ to $\tree_0(\omega)$ for almost every $\omega\in\Omega_1$. By Theorem~\ref{thm:coalescence}, for any $\alpha\in\labels_\star$ there are almost surely at most two geodesics with label $\alpha$, and the existence of two geodesics labeled $\alpha$ is a 0-1 event. Consequently, if $\labels_\star$ is at most countable, then there is almost surely a one-to-one correspondence between $\tree_0$ and the disjoint union of $\labels_\star$ and $\double_\star$. Of course, $\double_\star\subseteq\labels_\star$ always, and by combining Lemmas~\ref{lma:multiplicity} and~\ref{continuity_labels} it follows that $\double_\star$ is empty whenever $\labels_\star$ is finite. Hence, the statement follows.
\end{proof}

\subsection{Neighboring geodesics}

A way to describe the density of random coalescing geodesics in the space of infinite geodesics is in terms of neighboring geodesics. Two random coalescing geodesics $G$ and $G'$ satisfying $G<G'$ are said to be {\bf neighboring} if
$$
\P\big(\exists g\in\tree_0:G<g<G'\big)=0.
$$

\begin{lemma}\label{neighboring}
Let $G<G'$ be random coalescing geodesics. If $G$ and $G'$ are not neighboring, then there exists a random coalescing geodesic $G''$ such that $G<G''<G'$.
\end{lemma}

\begin{proof}
Assume that $\P\big(\exists g\in\tree_0:G<g<G'\big)>0$. Let $\alpha=F(G)$ and observe that $\alpha'=F(G')>\alpha$, since the contrary would give three geodesics labeled $\alpha$ with positive probability -- a contradiction to Theorem~\ref{thm:coalescence}.

There are three cases to consider: That either, with positive probability, there exists a geodesic with label in $(\alpha,\alpha')$, with label $\alpha$ or with label $\alpha'$. Assume first that there exists a geodesic with label in $(\alpha,\alpha')$ with positive probability. In that case, according to Lemma~\ref{lma:exists_label}, there exists $\beta\in(\alpha,\alpha')$ such that $\beta\in\labels_\star$. By Theorem~\ref{thm:coalescence} the geodesic $G_\beta^{ccw}$ is a random coalescing geodesic and lies strictly between $G$ and $G'$.
Assume instead that with positive probability there exists a geodesic $g\in\tree_0$, strictly larger than $G$, which has label $\alpha$. Then $\alpha\in\double_\star$, by Theorem~\ref{thm:coalescence}, and $G$ must coincide with the clockwise-most geodesic with label $\alpha$. Hence, $G_\alpha^{ccw}$ lies strictly between $G$ and $G'$.
The final case, when with positive probability there exists a geodesic with label $\alpha'$, is similar.
\end{proof}

\begin{remark}\label{rem:multiplicity}
We observe that if $\alpha\in\double_\star$, then $G_\alpha^{cw}$ and $G_\alpha^{ccw}$ are neighbouring. Since their Busemann functions are linear to distinct functionals supporting $\ball$, and since for each functional $\rho$ tangent to $\ball$ there is a unique random coalescing geodesic tangent to $\rho$, we may conclude that elements of $\double_\star$ are associated to corners of $\ball$.
\end{remark}

\subsection{Directions of differentiability}

Outside of the corners of the asymptotic shape we rule out the existence of multiple geodesics and bigeodesics. The following is a restatement of Theorem~\ref{versailles}.

\begin{theorem}\label{re-versailles}
Let $v\in S^1$ be a direction of differentiability of $\partial\ball$. Then,
\begin{enumerate}[\quad (a)]
\item $\P\big(\exists \text{ two disjoint geodesics with $v$ as a direction}\big)=0$;
\item $\P\big(\exists \text{ a bigeodesic with $v$ as a direction}\big)=0$.
\end{enumerate}
\end{theorem} 

\begin{proof}
Since $v\in S^1$ be a direction of differentiability there is a unique linear functional $\rho$ supporting to $\ball$ that contains $v$. So, by Proposition~\ref{bordeaux} there is a unique random coalescing geodesic $G$ for which $v$ is contained in $\arc(\rho_G)=\{x\in S^1:\mu(x)=\rho_G(x)\}$. $\dir(G)$ may or may not contain $v$, but for every other random coalescing geodesic $G'$ we have $v\not\in\dir(G')$ almost surely.

Assume for a contradiction that there are two geodesics in $\tree_0$ with $v$ in its set of directions with positive probability. We may then find $\delta>0$ and $m\ge1$ such that
\begin{equation}\label{eq:two geos}
\P\big(\exists\text{ two geodesics in $\tree_0$ with $v$ as a direction that diverge within $m$ steps}\big)>\delta.
\end{equation}
We define two random coalescing geodesics $G^-$ and $G^+$ as follows: If $G$ is cw-isolated, then let $G^-$ be its clockwise neighbor. If $G$ is cw-dense, then pick $G^-<G$ such that
\begin{equation}\label{eq:three geos}
\P\big(\text{$G$ and $G^-$ diverge within $m$ steps}\big)<\delta/4.
\end{equation}
Define $G^+>G$ similarly. $G^-$ and $G^+$ correspond to arcs of $S^1$ clockwise and counterclockwise of $\arc(\rho_G)$. Consequently, neither of the two may contain $v$ in its set of directions, and with probability 1 no geodesic in $\tree_0$ outside of the cone counterclockwise between $G^-$ and $G^+$ may so either. The assumption in~\eqref{eq:two geos} thus implies that
$$
\P\big(\exists\text{ two geodesics ccw between $G^-$ and $G^+$ that diverge within $m$ steps}\big)>\delta.
$$
Then, with probability $\delta/2$ there must be a geodesic counterclockwise between either $G^-$ and $G$, or $G$ and $G^+$, that diverges within $m$ steps. Assume the former. In this case $G^-$ and $G$ cannot be neighboring, and we have reached a contradiction to~\eqref{eq:three geos}.

We have showed that no geodesic in $\tree_0$ apart from $G$ may contain $v$ as a direction, almost surely. Similarly, for each $u\in\Z^2$ no geodesic apart from $G(u)$ may contain $v$ as a direction, almost surely. Since $G$ is coalescing, it follows that almost surely there are no two disjoint geodesics for which $v$ is a direction. This proves part~\emph{(a)} of the theorem. Part~\emph{(b)} follows from part~\emph{(a)} together with Proposition~\ref{temecula}.
\end{proof}

\subsection{Ordering of Busemann functions and linear functionals}

The remainder of this section will to a large extend aim to compare Busemann functions and linear functionals to which Busemann functions are linear. We shall start this work with a few of lemmas. 

\begin{lemma}\label{another lemma}
Let $G$ and $G'$ be random coalescing geodesics such that $G<G'$.
Then, with probability one, for all $y\in\Z^2$ satisfying
\begin{itemize}
\item $y$ is not in the cone ccw between $G(0)$ and $G'(0)$;
\item $0$ is not in the cone ccw between $G(y)$ and $G'(y)$;
 and  
 \item $G(y)$ intersects $G'(0)$,
\end{itemize}
and for all $g,g'\in\tree_0$ such that $G\le g\le g'\le G'$, we have
$$
B_G(0,y)\,\le\,B_g(0,y)\,\le\,B_{g'}(0,y)\,\le\,B_{G'}(0,y).
$$
Similarly, for $y\in\Z^2$ such that $G'(y)$ intersects $G(0)$ all inequalities are reversed. 
\end{lemma}

\begin{proof}
It follows from Lemma~\ref{rioja} that for almost every realization $\omega\in\Omega_1$ the Busemann function is well-defined for every $g\in\tree_0(\omega)$. Let $g=(0,v_1,v_2,\ldots)$ and $g'=(0,v'_1,v'_2,\ldots)$ be geodesics satisfying $G\le g\le g'\le G'$. For $y$ satisfying the assumptions of the lemma the geodesic between $y$ and $v_k$ will have to intersect $g'$, see Figure~\ref{fig:conseq2}. Let $u_k$ denote the first point on $g'$ visited by $\geo(y,v_k)$. Since $T(0,v_k)\le T(0,u_k)+T(u_k,v_k)$ it follows that
$$
B_g(0,y)\,=\,\lim_{k\to\infty}\big[T(0,v_k)-T(y,v_k)\big]\,\le\,\lim_{k\to\infty}\big[T(0,u_k)-T(y,u_k)\big]\,\le\, B_{g'}(0,y),
$$
where the final inequality follows since $B_{g'}(0,y)$ is the limit of an increasing sequence, while $(u_k)_{k\ge1}$ may not diverge.

The remaining statement is proved analogously.
\end{proof}

\begin{figure}[htbp]
\begin{center}
\includegraphics{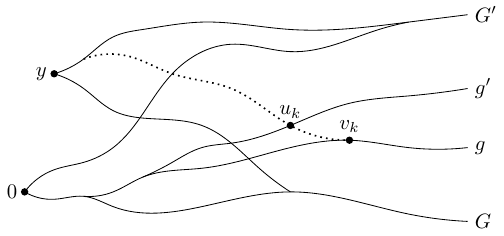}
\end{center}
\caption{Ordering of geodesics.}
\label{fig:conseq2}
\end{figure}

Recall that we by $\support$ denote the set of linear functionals $\rho:\R^2\to\R$ supporting to $\ball$. $\support$ is naturally parametrized by $S^1$, via their gradients, which induces a topology on $\support$. Let $\rho_\star=\rho_{\Gamma_\star}$ be the linear functional associated to the reference geodesic $\Gamma_\star$. We impose an ordering on $\support$ by saying that $\rho\le\rho'$ if $\rho'$ can be reached from $\rho$ via a counterclockwise motion without crossing $\rho_\star$. Our next lemma says that this ordering is consistent with the counterclockwise ordering of random coalescing geodesics.

\begin{lemma}\label{weak ordering}
Let $G$ and $G'$ be random coalescing geodesics. If $G<G'$, then $\rho_G<\rho_{G'}$.
\end{lemma}

\begin{proof}
If there are exactly four random coalescing geodesics, then the lemma follows by rotational symmetry. Assume instead that there are at least eight, and pick eight such geodesics $G^1<G^2<\ldots<G^8$ so that the set of directions of each consecutive pair is contained in an arc of length $\pi/2$. It will suffice to prove the lemma for $G<G'$ counterclockwise between some pair $G^i$ and $G^{i+1}$.

Assuming this is the case, there exists a half-plane $H_i$ such that $G$ and $G'$ eventually move into $H_i$. We will assume that $H_i$ is the right half-plane. We then pick an increasing sequence $(y_k)_{k\ge1}$ of points along the vertical axis for which $G$ is entirely contained in the right half-plane. For all large $k$ we then have that $G(y_k)$ intersects $G'(0)$. Hence, by Lemma~\ref{another lemma} and Proposition~\ref{sonoma}, we conclude that, almost surely,
$$
\rho_G({\bf e}_2)\,=\,\lim_{k\to\infty}\frac{1}{|y_k|}B_G(0,y_k)\,\le\,\lim_{k\to\infty}\frac{1}{|y_k|}B_{G'}(0,y_k)\,=\,\rho_{G'}({\bf e}_2).
$$
By Proposition~\ref{bordeaux} the inequality must be strict, i.e.\ $\rho_G({\bf e}_2)<\rho_{G'}({\bf e}_2)$.

There are now three cases to consider. Either one of $\rho_G({\bf e}_2)$ and $\rho_{G'}({\bf e}_2)$ is zero, they have different signs, or both have the same sign. In the case $\rho_G({\bf e}_2)=0$, then the gradient of $\rho_G$ is parallel with the horizontal axis. If $\rho_G({\bf e}_2)>0$, then the line $\{x\in\R^2:\rho_G(x)=1\}$ intersects the positive half of the vertical axis, so its gradient is directed in the first quadrant. These observations settle the former two cases. In the case that $\rho_G({\bf e}_2)$ and $\rho_{G'}({\bf e}_2)$ have the same sign, then note that the line $\{x\in\R^2:\rho_G(x)=1\}$ intersects the vertical axis above the line $\{x\in\R^2:\rho_{G'}(x)=1\}$. So, also in this case we have $\rho_G<\rho_{G'}$.
\end{proof}

We next show that geodesics close to some random coalescing geodesic $G$ will have its Busemann function approximated by $\rho_G$. Below we shall by $\|\cdot\|$ denote arc-length on $S^1$.

\begin{lemma} \label{promo}
There exists $K\ge1$ such that for every $\eps>0$ and pair of random coalescing geodesics $G$ and $G'$ such that $\|\rho_G-\rho_{G'}\|<\eps$ and $\dir(G)\cup \dir(G')$ is contained in an arc of length at most $\eps$, we have
$$
\P\Big(\forall g\in\tree_0: G\le g\le G'\text{ we have }\big|B_g(0,y)-\rho_G(y)\big|<K\eps|y|\text{ for large }|y|\Big)=1.
$$
\end{lemma}

\begin{proof}
We first choose an almost surely finite $N_1\ge1$ such that for all $|y|\ge N_1$ we have
$$
|T(y,y+z)-\mu(z)|<\eps\max\{|y|,|z|\}\quad\text{for all }z\in\Z^2.
$$
This can be done as of Proposition~\ref{extended shape}. Second, we pick $N_2\ge N_1$ such that for all $|y|\ge N_2$
$$
|B_G(0,y)-\rho_G(y)|<\eps|y|\quad\text{and}\quad|B_{G'}(0,y)-\rho_{G'}(y)|<\eps|y|.
$$
Third, pick $v_1,v_2,\ldots,v_m$ in $\Z^2$ so that for every $y\in\Z^2$ we have $|y-nv_k|<2\eps|y|$ for some $n$ and $k$; write $v_y$ for the point of the form $nv_k$ minimizing $|y-nv_k|$. By assumption of the lemma, we may further assume that neither of the $v_k$ is directed in the arc obtained as the convex hull of $\dir(G)$ and $\dir(G')$, nor in its rotation by an angle $\pi$. That is, we may pick $N_3\ge N_2$ such that neither $v_y$ in contained in the cone ccw between $G(0)$ and $G'(0)$, nor is $0$ contained in the cone ccw between $G(v_y)$ and $G'(v_y)$, for all $|y|\ge N_3$.

The above choices of $v_1,v_2,\ldots,v_m$ assures that for some $K\ge 1$ we have
\begin{equation}\label{eq:linear1}
\big|B_g(0,y)-B_g(0,v_y)\big|\,\le\,|B_g(y,v_y)|\,\le\,T(y,v_y)\,\le\,\eps K|y|
\end{equation}
for all $g\in\tree_0$ and $|y|\ge N_1$. By Lemma~\ref{another lemma} we conclude that for $|y|\ge N_3$ we have
\begin{equation}\label{eq:linear2}
\big|B_g(0,v_y)-B_G(0,v_y)\big|\,\le\,\big|B_G(0,v_y)-B_{G'}(0,v_y)\big|\,\le\,\big|\rho_G(v_y)-\rho_{G'}(v_y)\big|+\eps|v_y|
\end{equation}
when $G\le g\le G'$, which by assumption is bounded above by $4\eps|y|$.
Hence, combining~\eqref{eq:linear1} and~\eqref{eq:linear2} we conclude that for every $g\in\tree_0$ such that $G\le g\le G'$ and every $y\ge N_3$ we have
$$
\big|B_g(0,y)-\rho_G(y)\big|<(K+4)\eps|y|,
$$
as required.
\end{proof}

\subsection{Asymptotic directions and Busemann functions}

Propositions~\ref{sonoma}--\ref{bordeaux} describe the asymptotic properties of individual random coalescing geodesics and their Busemann functions. We shall in what follows give a simultaneous description of the asymptotic properties of all infinite geodesics. A first step in this direction will be to understand the topology of the set of linear functionals $\rho\in\support$ for which there is a random coalescing geodesic with Busemann function linear to $\rho$.

\begin{theorem}\label{deportivo}
The set $\functionals_\star\subseteq\support$ of functionals of the form $\rho_G$ for some random coalescing geodesic $G$ is closed as a subset of $S^1$ and contains every functional tangent to $\ball$.
\end{theorem}

\begin{proof}
That each linear functional tangent to $\ball$ is present in $\functionals_\star$ is a consequence of Theorem~\ref{yahoo}. Let $\rho$ be a limit point of $\functionals_\star$. Then there exists a monotone sequence $(\rho_k)_{k\ge1}$ in $\functionals_\star$ converging to $\rho$. By Lemma~\ref{weak ordering} we find that the sequence $(G_k)_{k\ge1}$ of random coalescing geodesics corresponding to the sequence $(\rho_k)_{k\ge1}$ is again monotone. By Theorem~\ref{thm:labels} the limit $G:=\lim_{k\to\infty}G_k$ exists, and by Theorem~\ref{thm:coalescence} it is a random coalescing geodesic; it is the (ccw- or cw-most) geodesic with label $\lim_{k\to\infty}F(G_k)$. To end the proof it will suffice to prove that $\rho_G=\rho$.
However, by definition of $\rho_G$ we have for every $z\in\Z^2$ that
$$
\big|\rho_G(z)-\rho_k(z)\big|\,=\,\big|\E\big[B_G(0,z)-B_{G_k}(0,z)\big]\big|.
$$
For large $k$ the geodesics $G(0)$ and $G(z)$ will with probability close to one coalesce before $G$ and $G_k$ diverges. Consequently, we find that $\rho_G(z)=\lim_{k\to\infty}\rho_k(z)=\rho(z)$, so $\rho_G=\rho$.
\end{proof}

We next show that every geodesic has its Busemann function asymptotically linear to some element of $\functionals_\star$, and that this linear functional describes its set of directions. The following is a slight strengthening of Theorem~\ref{all geos}.

\begin{theorem} \label{promotion}
For almost every realization $\omega\in\Omega_1$ there exists for every $g\in\tree_0(\omega)$ a linear functional $\rho\in\functionals_\star$ such that the Busemann function of $g$ is asymptotically linear to $\rho$ and $\dir(g)$ is a subset of $\arc(\rho)=\{x\in S^1:\rho(x)=\mu(x)\}$.
\end{theorem}

\begin{proof}
Let $K\ge1$ be the constant declared in Lemma~\ref{promo}. Our first aim will be to find a nested sequence $(\mathcal{F}_n)_{n\ge1}$ of finite families of random coalescing geodesics such that for each $n\ge1$ and each consecutive pair $G$ and $G'$ in $\mathcal{F}_n$ (in the counterclockwise ordering) we have
\begin{equation}\label{eq:main aim}
\P\Big(\forall g\in\tree_0: G<g<G'\text{ we have }\big|B_g(0,y)-\rho_G(y)\big|<(K+1)|y|/n\text{ for large }|y|\Big)=1.
\end{equation}

We define the nested sequence $(\mathcal{F}_n)_{n\ge1}$ as follows. From Theorem~\ref{deportivo} we recall that $\functionals_\star$ is a closed as a subset of $S^1$, so its complement is a countable union of open intervals $I$. We then define $\mathcal{F}_n$ inductively to contain each element of $\mathcal{F}_{n-1}$ and, in addition,
\begin{itemize}
\item every $G$ for which $\arc(\rho_G)$ has width at least $1/n$;
\item every $G$ that corresponds to an endpoint of some interval $I$ of width at least $1/n$;
\item some additional finite number of random coalescing geodesics so that the union of $\arc(\rho_G)$ over $G\in\mathcal{F}_n$ is $1/n$-dense in $S^1$, and so that the set of linear functionals corresponding to the elements of $\mathcal{F}_n$ is $1/n$-dense in $\functionals_\star$.
\end{itemize}

Assume that~\eqref{eq:main aim} fails for some $n\ge1$ and some consecutive pair $G$ and $G'$. We may then find $\delta>0$ and $m\ge1$ so that with probability at least $\delta$ there exists $g\in\tree_0$ with $G<g<G'$ that diverges from both $G$ and $G'$ within $m$ steps and such that for infinitely many $y$
\begin{equation}\label{eq:first_error}
\big|B_g(0,y)-\rho_G(y)\big|\ge (K+1)|y|/n.
\end{equation}
For this to happen the pair $G$ and $G'$ cannot be neighboring. By Lemma~\ref{neighboring} there exists a random coalescing geodesic $G''$ such that $G<G''<G'$. By Lemma~\ref{weak ordering} it follows that $\rho_G<\rho_{G''}<\rho_{G'}$. Consequently, $G$ and $G'$ cannot correspond to boundary points of some interval $I$ of width as large as $1/n$, and hence $\|\rho_G-\rho_{G'}\|<1/n$. In case $\dir(G)\cup\dir(G')$ is contained in an arc of length at most $1/n$ then Lemma~\ref{promo} gives a contradiction. Otherwise, we let $G^\ast>G$ be the counterclockwise neighbor of $G$ if is ccw-isolated, or so that the two has probability at most $\delta/4$ to diverge within $m$ steps. Pick $G^{\ast\ast}<G'$ similarly. In either case we conclude that with probability at least $\delta/2$ there exists a geodesic $g\in\tree_0$ with $G^\ast<g<G^{\ast\ast}$ such that~\eqref{eq:first_error}, and hence that
$$
\big|B_g(0,y)-\rho_{G^\ast}(y)\big|\ge K|y|/n,
$$
holds for some arbitrarily large $y$. However, $\dir(G^\ast)\cup\dir(G^{\ast\ast})$ is contained in an arc of length no larger than $1/n$, so this contradicts Lemma~\ref{promo}. Hence, we have proved~\eqref{eq:main aim}.

We proceed and show how the theorem follows from~\eqref{eq:main aim}. For each $n\ge1$ and $g\in\tree_0$ let
$$
\rho_{g,n}:=\max\{\rho_G:G\in\mathcal{F}_n,G\le g\}.
$$
For almost every $\omega\in\Omega_1$ and every $g\in\tree_0$ the sequence $(\rho_{g,n})_{n\ge1}$ is monotone and thus convergent. Denote by $\rho_g$ the limit $\lim_{n\to\infty}\rho_{g,n}$. Since $\functionals_\star$ is closed the set $\{\rho_g:g\in\tree_0(\omega)\}$ is a subset of $\functionals_\star$ for almost every $\omega\in\Omega_1$. Since $\mathcal{F}_n$ is $1/n$-dense in $\functionals_\star$ it follows that $\|\rho_g-\rho_{g,n}\|<1/n$ uniformly over $g\in\tree_0$, almost surely. From~\eqref{eq:main aim} it then follows that
\begin{equation}\label{eq:main consequence}
\P\Big(\forall g\in\tree_0\text{ we have }\big|B_g(0,y)-\rho_g(y)\big|<(K+2)|y|/n\text{ for all large }|y|\Big)=1.
\end{equation}
By the union bound it follows that with probability one every $g\in\tree_0$ has a Busemann function linear to some element in $\functionals_\star$. Finally, by combining~\eqref{eq:main consequence}, Lemma~\ref{rioja} and the shape theorem, we observe that, outside of a null set, we have for any $g\in\tree_0$ and any $x\in S^1$ for which there is a subsequence $(v_k)_{k\ge1}$ of $g$ such that $v_k/|v_k|\to x$, that
$$
\rho_g(x)=\lim_{k\to\infty}\frac{B_g(0,v_k)}{|v_k|}=\lim_{k\to\infty}\frac{T(0,v_k)}{|v_k|}=\mu(x).
$$
That is, with probability one we have $\dir(g)\subseteq\arc(\rho_g)$ for all $g\in\tree_0$, as claimed.
\end{proof}

We have seen that for almost every realization, every infinite geodesic has a well-defined Busemann function that is asymptotically linear. For $g\in\tree_0$, let $\rho_g$ denote the element in $\functionals_\star$ to which the Busemann function of $g$ is linear. With a slight abuse of notation we shall write $\rho_{G(\omega)}$ for the element in $\functionals_\star$ to which the Busemann function of the image of $G$ is linear in the realization $\omega$. The following theorem says that the ordering on $\support$ is consistent with the counterclockwise ordering of geodesics in $\tree_0$, and that the set of linear functionals associated to Busemann functions is an almost surely deterministic set.
Recall that $\rcg$ denotes the family of random geodesics of the form $G_\alpha^{cw}$ and $G_\alpha^{ccw}$ for $\alpha\in\labels_\star$.

\begin{theorem}\label{Busemann consistency}
For almost every realization $\omega\in\Omega_1$ we have
\begin{enumerate}[\quad (a)]
\item for all $g\le g'$ in $\tree_0(\omega)$ that $\rho_g\le\rho_{g'}$;
\item for every $G\in\rcg$ that $\rho_{G(\omega)}=\rho_G$;
\item the set $\functionals(\omega)=\{\rho\in\support:\rho=\rho_g\text{ for some }g\in\tree_0(\omega)\}$ equals $\functionals_\star$.
\end{enumerate}
\end{theorem}

\begin{proof}
We begin with part~\emph{(a)}. If there are exactly four random coalescing geodesics, then there is nothing more to prove. Assume instead there are at least eight, and pick $G^1<G^2<\ldots<G^8$ such that each consecutive pair spans an arc of $S^1$ of length at most $\pi/2$. It will suffice to prove the statement for the restriction of $\tree_0$ to each consecutive pair $G^i$ and $G^{i+1}$.

Let $G<G'$ be such a consecutive pair. Since the cone counterclockwise between $G$ and $G'$ is asymptotically contained in an arc of length $\pi/2$ there exists a half-plane $H_i$ such that $G$ and $G'$ eventually move into $H_i$. We will assume that $H_i$ is the right half-plane. We then pick an increasing sequence $(y_k)_{k\ge1}$ of points along the vertical axis for which $G$ is entirely contained in the right half-plane. For all large $k$ we have that $G(y_k)$ intersects $G'(0)$. Hence, by Lemma~\ref{another lemma} and Theorem~\ref{promotion}, we conclude that, almost surely, for all $g\le g'$ in $\tree_0(\omega)$ we have
\begin{equation}\label{eq:rho-g}
\rho_g({\bf e}_2)\,=\,\lim_{k\to\infty}\frac{1}{|y_k|}B_g(0,y_k)\,\le\,\lim_{k\to\infty}\frac{1}{|y_k|}B_{g'}(0,y_k)\,=\,\rho_{g'}({\bf e}_2).
\end{equation}
We note that~\eqref{eq:rho-g}, just as in the proof of Lemma~\ref{weak ordering}, implies that $\rho_g\le\rho_{g'}$, as required.

We continue with the proof of part~\emph{(b)}. Let $(\mathcal{F}_n)_{n\ge1}$ be the nested sequence of finite families of random coalescing geodesics defined in the proof of Theorem~\ref{promotion}. For every consecutive pair $G<G'$ in $\mathcal{F}_n$ we have that $\|\rho_G-\rho_{G'}\|\le2/n$. Moreover, since $\mathcal{F}_n$ is finite it follows that $\rho_{G(\omega)}=\rho_G$ for every $G\in\mathcal{F}$ almost surely.

Given a consecutive pair $G<G'$ and let $\rcg^\ast=\{G^\ast\in\rcg:G\le G^\ast\le G'\}$. Since the labeling of geodesics is consistent with the counterclockwise ordering (Proposition~\ref{label props}) it follows that for almost every $\omega$ we have $G(\omega)\le G^\ast(\omega)\le G'(\omega)$ for all $G^\ast\in\rcg^\ast$. By part~\emph{(a)} it follows that $\rho_{G(\omega)}\le\rho_{G^\ast(\omega)}\le\rho_{G'(\omega)}$ for all $G^\ast\in\rcg^\ast$ and almost every $\omega$. We conclude that for every $n\ge1$ we have
$$
\P\big(\text{for all $G\in\rcg$ we have }\|\rho_{G(\omega)}-\rho_G\|\le2/n\big)=1.
$$
Part~\emph{(b)} now follows from a union bound.

Finally, we observe that $\functionals(\omega)\subseteq\functionals_\star$ almost surely as a consequence of Theorem~\ref{promotion}, whereas the reverse inclusion $\functionals(\omega)\supseteq\functionals_\star$ holds with probability one due to part~\emph{(b)}.
\end{proof}

We next address uniqueness of linear functionals. Recall that no functional outside $\functionals_\star$ may appear as the limit of a Busemann function associated to geodesics in $\tree_0$. The next result, together with Theorems~\ref{deportivo} and~\ref{Busemann consistency}, proves Theorem~\ref{unique support}.

\begin{theorem}\label{again unique}
For every functional $\rho\in\functionals_\star$ we have
$$
\P\big(\exists\text{ two disjoint geodesics with Busemann function linear to }\rho\big)=0.
$$
\end{theorem}

\begin{proof}
No two random coalescing geodesics are associated to the same linear functional due to Proposition~\ref{bordeaux}.
Let $G$ be the random coalescing geodesic for which $\rho_G=\rho$. Assume, for a contradiction, that for some $\delta>0$
\begin{equation}\label{eq:two linear}
\P\big(\exists\text{ two geodesics in $\tree_0$ with Busemann function linear to }\rho\big)>\delta.
\end{equation}
In that case there exists $m\ge1$ such that with probability at least $\delta$ there are two geodesics that diverge within $m$ steps that both have Busemann function linear to $\rho$. We then choose $G^-<G<G^+$ as in the proof of Theorem~\ref{re-versailles}, that is, either being neighboring to $G$ or such that $G^-$ and $G$, respectively $G$ an $G^+$, diverge within $m$ steps with probability less than $\delta/4$. Together these choices and~\eqref{eq:two linear} imply that there is probability at least $\delta/2$ to exist a geodesic outside of the cone counterclockwise between $G^-$ and $G^+$ that has Busemann function linear to $\rho$. This is a contradiction to Theorem~\ref{Busemann consistency}.

We conclude that $G$ is almost surely the unique geodesic in $\tree_0$ with Busemann function asymptotically linear to $\rho$. Similarly, for each $v\in\Z^2$ we have that $G(v)$ is almost surely the unique geodesic in $\tree_v$ that has Busemann function linear to $\rho$. Since $G$ is coalescing we conclude that almost surely there are no two disjoint geodesics with Busemann function linear to $\rho$.
%
%
\end{proof}


We end this section by providing a strengthening of part~\emph{(a)} of Theorem~\ref{Busemann consistency} by relating labels to supporting functionals.

\begin{theorem}\label{strong consistency}
For almost every $\omega\in\Omega_1$ and for all $u,v\in\Z^2$ we have for all $g\in\tree_u(\omega)$ and $g'\in\tree_v(\omega)$ that if $g\le g'$, then $\rho_g\le \rho_{g'}$.
\end{theorem}

\begin{proof}
We begin by establishing a 1-1 correspondence between $\functionals_\star$ and the disjoint union of $\labels_\star$ and $\double_\star$. For $\alpha\in\double_\star$, let $\rho_\alpha^{cw}$ and $\rho_\alpha^{ccw}$ denote the linear functionals associated to $G_\alpha^{cw}$ and $G_\alpha^{ccw}$, respectively. For $\alpha\in\labels_\star\setminus\double_\star$, let $\rho_\alpha$ denote the linear functional associated to $G_\alpha^{cw}$. Since $G_\alpha^{cw}$ and $G_\alpha^{ccw}$, for $\alpha\in\labels_\star$ are random coalescing geodesics, this defines a map from the disjoint union of $\labels_\star$ and $\double_\star$ to $\functionals_\star$ which is 1-1 by Proposition~\ref{bordeaux} and onto due to Theorem~\ref{thm:coalescence}.

We next claim that $\double_\star$ is at most countably infinite. One way to see this is to first note that for every $\alpha\in\double_\star$ the pair $G_\alpha^{cw}$ and $G_\alpha^{ccw}$ are neighbouring, due to Theorem~\ref{thm:coalescence}. It then follows by Lemma~\ref{weak ordering} that $\rho_\alpha^{cw}<\rho_\alpha^{ccw}$ and $(\rho_\alpha^{cw},\rho_\alpha^{ccw})\cap\functionals_\star=\emptyset$. Hence, $\double_\star$ can be at most countable.

Since $\double_\star$ is at most countable it follows from Theorem~\ref{thm:coalescence} that, almost surely, for every $\alpha\in\double_\star$ and $v\in\Z^2$ there are exactly two geodesics in $\tree_v$ with label $\alpha$, of which the clockwise-most  form a coalescing family with Busemann function asymptotically linear to $\rho_\alpha^{cw}$ and the counterclockwise-most form a coalescing family with Busemann function asymptotically linear to $\rho_\alpha^{ccw}$. Moreover, by Theorem~\ref{Busemann consistency}, almost surely, for every $\alpha\in\labels_\star\setminus\double_\star$ and $v\in\Z^2$ we have that every $g\in\tree_v$ with label $\alpha$ has Busemann function asymptotically linear to $\rho_\alpha$. The theorem now follows from the consistency of the labeling and the global ordering of geodesics.
\end{proof}

\begin{remark}
As a corollary to Theorems~\ref{Busemann consistency} and~\ref{again unique} one may show that for almost every $\omega$ we have for every monotone sequence $(g_k)_{k\ge1}$ in $\tree_0(\omega)$ that $\lim_{k\to\infty}\rho_{g_k}=\rho_{g_\infty}$, where $g_\infty=\lim_{k\to\infty}g_k$.
However, the converse statement is not generally true, as there may be exceptional elements $\rho\in\functionals_\star$ for which there are two geodesics with Busemann function asymptotically linear to $\rho$. Again, recall footnote~\ref{prediction}.
\end{remark}

\section{The midpoint problem}
\label{midpoint}

In this section we apply the theory constructed around random coalescing geodesics to prove Theorem~\ref{hall of mirrors}, which answers a question raised in the work of Benjamini, Kalai and Schramm~\cite{benkalsch03}.
We begin by outlining our strategy in the case that $v_k=(k,0)$ and $u_k=-v_k$. We argue by contradiction. If there does exist a $\delta>0$ such that, uniformly in $k$,
\begin{equation}  
\P\big((0,0) \in \geo((-k,0),(k,0))\big) >\delta,
\label{sentencing day}
\end{equation}  
then we will construct a shift-invariant measure on (sparse) families of non-crossing geodesics which has infinitely many coalescence classes. This violates Proposition~\ref{p:sparse measure}.\footnote{We could alternatively use this measure to construct a (sparse) random non-crossing geodesic with infinitely many coalescence classes to obtain a contradiction to Proposition~\ref{p:sparse}.}

To construct a non-crossing geodesic we proceed as follows: For each $k$ we form a set
$$
I_k=\big\{i\in\Z: (0,i) \in \geo((-k,i),(k,i))\big\}.
$$
By the ergodic theorem and~\eqref{sentencing day} this set $I_k$ has density at least $\delta$. Given $i,j\in I_k$ we then note, crucially, that if $\geo((-k,i),(0,i))$ and $\geo((-k,j),(0,j))$ touch, then $\geo((0,i),(k,i))$ and $\geo((0,j),(k,j))$ are unlikely to touch too. If they did, then both geodesics would have to visit both $(0,i)$ and $(0,j)$ not to contradict unique passage times (see Figure~\ref{fig:mid1}). As this is unlikely to happen for $i$ and $j$ far apart we next use this observation to thin $I_k$ to get a subset $I_k^+ \subset I_k$ such that for all $i,j \in I_k^+$
$$
\geo\big((0,i),(k,i)\big) \cap \geo\big((0,j),(k,j)\big)=\emptyset.
$$
This gives us a family $\{\geo((0,i),(k,i)):i\in I_k^+\}$ of finite non-crossing geodesics that do not coalesce. Importantly, because of the above observation, the geodesics in this family do not become `more coalescing' as $k$ increases. That is, also $I_k^+$ will have density bounded away from zero.
\begin{figure}[htbp]
\begin{center}
\includegraphics{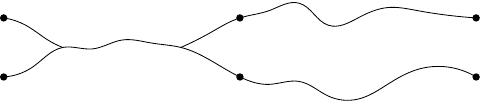}
\end{center}
\caption{The intersection of two geodesics is a continuous path.}
\label{fig:mid1}
\end{figure}

We may now use the Damron-Hanson strategy to form a sequence of measures on finite geodesics. This sequence will have a sub-sequential limit. Since $I_k^+$ has positive density, each sub-sequential limit will be supported on (families of) non-crossing and non-coalescing geodesics. Then, we use this to construct a measure on (sparse) families of non-crossing geodesic which has infinitely many coalescence classes. This is a contradiction to Proposition~\ref{p:sparse measure}.

The rest of this section will be dedicated to making the above outline rigorous. In the case when the asymptotic shape has sufficiently many sides\footnote{Recall that the number of sides equals $n$ if $\ball$ is an $n$-gon and $\infty$ otherwise.} (at least 40), then this will be mostly straightforward based on the above outline. In the case when $\ball$ is a polygon with a small number of sides (less than 40), this will require a much more careful analysis. In either case, we will part from the following assumption, that there exist $\delta>0$ and sequences $(u_k)_{k\ge1}$ and $(v_k)_{k\ge1}$ in $\Z^2$ such that $|u_k|,|v_k|\to\infty$ and
\begin{equation}  
\P\big(0 \in \geo(u_k,v_k)\big) >\delta. 
\label{sentencing}
\end{equation}
By restricting to a further subsequence, we may assume that $u_k/|u_k|\to u$ and $v_k/|v_k|\to v$ for some $u,v\in S^1$.

Before we start we remark that the main complication in the proof of Theorem~\ref{hall of mirrors} comes from the possible existence of `corners', i.e.\ points of non-differentiability, of the asymptotic shape. Under the additional assumption that $\ball$ has no corners it is possible to give a significantly shorter argument. However, we are interested in the full result and we will therefore have to work accordingly.

It seems plausible that in order to maximize the probability in~\eqref{sentencing} one should consider sequences $(u_k)_{k\ge1}$ and $(v_k)_{k\ge1}$ moving in opposite directions. Although it is easier to rule out~\eqref{sentencing} in the case when $v\neq -u$ we see no gain in presenting such an argument separately. An argument of this kind will instead appear somewhat implicitly in what follows.

\subsection{The central argument} \label{ray charles}

We start off by defining an event central for the construction of a family of finite geodesics. This event will involve a family of random coalescing geodesics to be used to control the directions of the finite geodesics. Given $x\in\Z^2$ and eight random coalescing geodesics $\{G^i:i=1,2,\ldots,8\}$, define $\good_k(x)=\good_k(x,G^1,G^2,\ldots,G^8)$ to be the event that the following all occur (see Figure~\ref{fig:mid2}):
\begin{itemize}
\item $x \in \geo(x+u_k,x+v_k)$;
\item $\geo(x,x+v_k)$ is counterclockwise between $G^1(x)$ and $G^2(x)$;
\item $\geo(x+u_k,x)$ is counterclockwise between $G^3(x)$ and $G^4(x)$;
\item $\geo(x+u_k,x)$ is counterclockwise between $G^5(x+u_k)$ and $G^6(x+u_k)$; and
\item $\geo(x,x+v_k)$ is counterclockwise between $G^7(x+v_k)$ and $G^8(x+v_k)$.
\end{itemize}
\begin{figure}[htbp]
\begin{center}
\includegraphics{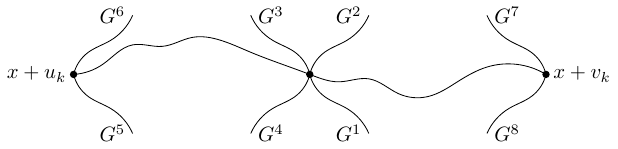}
\end{center}
\caption{The good event.}
\label{fig:mid2}
\end{figure}

Naturally, the $G^i$'s can only help to control $\geo(x+u_k,x+v_k)$ if they can be chosen suitably, while $\good_k(x)$ occurs with positive probability. Recall that $H_i$, $i=0,1,\ldots,7$, denote the eight half-planes in directions $(\pm1,0)$, $(0,\pm1)$ and $(\pm1,\pm1)$.
About half of the effort in proving Theorem~\ref{hall of mirrors} will aim at establishing the following lemma.

\begin{lemma} \label{cunningham}
Assume that~\eqref{sentencing} holds. Then there exists $\delta'>0$, a half-plane $H_i$ and eight random coalescing geodesics $G^1,G^2,\ldots,G^8$ such that $G^1, G^2$ and $G^5,G^6$ eventually move into $H_i$, $G^3,G^4$ and $G^7,G^8$ eventually move into $H_i^c$, and such that for all large $k$ we have
$$
\P\big(\good_k(0)\big)>\delta'.
$$
\end{lemma}

We postpone a full proof of Lemma~\ref{cunningham} to Section~\ref{sec:cunningham} below. However, before we proceed we give a rough outline of the proof in the case when the shape has sufficiently many sides (40 will suffice): Let $\mathcal{F}$ be a set of random coalescing geodesics with the property that if $G$ is in $\mathcal{F}$ then so is the random coalescing geodesic obtained from $G$ by rotation $\pi/2$. These geodesics split the plane into slices. It seems reasonable that if $v_k$ lies in one of these slices and $\geo(u_k,v_k)$ goes through the origin, then $u_k$ has to lie in the opposite slice, or at least some slice close to that one. This is indeed true and a consequence of Lemma~\ref{lma:slices} below. In particular $u$ cannot be too far from $-v$. We may then take random coalescing geodesics $G^1$ and $G^2$ which together span an angle at most $\pi/2$ and contain $v$ between them. Similarly we pick $G^3$ and $G^4$ around $u$. If there are sufficiently many sides to the shape then we can 
pick a set $\mathcal{F}$ with enough candidates to choose from, and thus guarantee that for some half-plane $H_i$ we will have $G^1$ and $G^2$ moving into $H_i$ while $G^3$ and $G^4$ move into $H_i^c$.

This rough sketch can and will be made rigorous below. The proof in the remaining case, when $\ball$ is a polygon with few sides, is considerably more involved, and we postpone further details to Section~\ref{sec:cunningham} below.

Now we proceed to show how to use Lemma~\ref{cunningham} to construct a family of finite non-crossing and non-coalescing geodesics. For ease of notation we further assume that $H_i=H_0$, i.e.\ the right half-plane; the remaining cases are treated verbatim. This assumption, in particular, implies that the projections of $u$ and $v$ along the first coordinate axis are strictly negative respectively positive.

First, let
\begin{equation*}
\begin{aligned}
\well_\ell^+(x)&:=\big\{G^i(x) \cap (x+H_0^c) \subseteq x+[-\ell/3,\ell/3]^2: i=1,2,5,6\big\},\\
\well_\ell^-(x)&:=\big\{G^i(x) \cap (x+H_0)     \subseteq x+[-\ell/3,\ell/3]^2: i=3,4,7,8\big\},
\end{aligned}
\end{equation*}
and $\well_\ell(x):=\well_\ell^+(x)\cap\well_\ell^-(x)$. Next, we set
$$
\good_{k,\ell}(x):=\good_k(x)\cap\well_\ell^+(x+u_k)\cap\well_\ell(x)\cap\well_\ell^-(x+v_k).
$$
Since the geodesics, by assumption, eventually move into the half-plane $H_0$ or its complement we can make the probability of both $\well_\ell^+(x)$ and $\well_\ell^-(x)$ arbitrarily close to 1 by increasing $\ell$. Hence, for some $\ell\ge1$ and all sufficiently large $k$ we have that
\begin{equation}\label{goodie}
\P\big(\good_{k,\ell}(x)\big)>\delta'/2.
\end{equation}

\begin{lemma} \label{frosted}
On the event that $\good_{k,\ell}(i{\bf e}_2)$ and $\good_{k,\ell}(j{\bf e}_2)$ occur, where $|i-j|>\ell$, then at least one of the following occurs:
\begin{enumerate}[\quad (a)]
\item $\geo(i{\bf e}_2+u_k,i{\bf e}_2) \cap \geo(j{\bf e}_2+u_k,j{\bf e}_2)=\emptyset$;
\item $\geo(i{\bf e}_2,i{\bf e}_2+v_k) \cap \geo(j{\bf e}_2,j{\bf e}_2+v_k)=\emptyset$. 
\end{enumerate}
\end{lemma}

\begin{proof}
First we note that the conditions of $\good_k$ and $\well_\ell$ imply that $i{\bf e}_2 \not \in \geo(j{\bf e}_2+u_k,j{\bf e}_2+v_k)$ and $j{\bf e}_2 \not \in \geo(i{\bf e}_2+u_k,i{\bf e}_2+v_k)$. Then we note that for any two finite geodesics the set of points in their intersection is connected, since otherwise would contradict the assumption on unique passage times. Thus the intersection occurs on one side of the midpoint or the other, but not both; recall Figure~\ref{fig:mid1}.
\end{proof}

\begin{lemma} \label{vashti}
Suppose that $\good_{k,\ell}(0)$, $\good_{k,\ell}(i{\bf e}_2)$ and $\good_{k,\ell}(j{\bf e}_2)$ occur for $i,j-i>\ell$.
\begin{enumerate}[\quad (a)]
\item If $\geo(u_k,0) \cap \geo(i{\bf e}_2+u_k,i{\bf e}_2)=\emptyset$, then $\geo(u_k,0) \cap \geo(j{\bf e}_2+u_k,j{\bf e}_2)=\emptyset$.
\item If $\geo(0,v_k) \cap \geo(i{\bf e}_2,i{\bf e}_2+v_k)=\emptyset$, then $\geo(0,v_k) \cap \geo(j{\bf e}_2,j{\bf e}_2+v_k)=\emptyset$.
\end{enumerate}
\end{lemma}

\begin{proof}
Our hypothesis imply that $\geo(u_k,0)$ and $\geo(j{\bf e}_2+u_k,j{\bf e}_2)$ can only intersect if both intersect $\geo(i{\bf e}_2+u_k,i{\bf e}_2)$; see Figure~\ref{fig:mid3}. Hence, if $\geo(u_k,0)$ does not intersect $\geo(i{\bf e}_2+u_k,i{\bf e}_2)$, then it cannot intersect $\geo(j{\bf e}_2+u_k,j{\bf e}_2)$ either. The other case is identical.
\end{proof}
\begin{figure}[htbp]
\begin{center}
\includegraphics{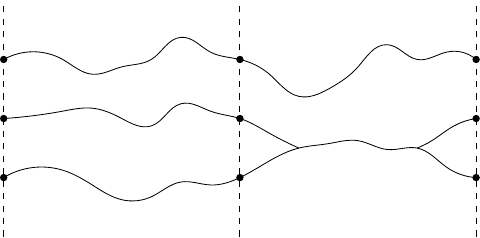}
\end{center}
\caption{Layering of finite geodesics in $I_{k,\ell}$.}
\label{fig:mid3}
\end{figure}

Next define $I_{k,\ell}(x):=\{i\in\ell\Z:\good_{k,\ell}(x+i{\bf e}_2)\text{ occurs}\}$, and $I_{k,\ell}^-(x)$ and $I_{k,\ell}^+(x)$ as
\begin{equation*}
\begin{aligned}
&\big\{j\in I_{k,\ell}:\geo(x+j{\bf e}_2+u_k,x+j{\bf e}_2)\cap \geo(x+i{\bf e}_2+u_k,x+i{\bf e}_2)=\emptyset\text{ for all }i\in I_{k,\ell}, i<j\big\},\\
&\big\{j\in I_{k,\ell}:\geo(x+j{\bf e}_2,x+j{\bf e}_2+v_k)\cap \geo(x+i{\bf e}_2,x+i{\bf e}_2+v_k)=\emptyset\text{ for all }i\in I_{k,\ell}, i<j\big\}.
\end{aligned}
\end{equation*}
It is immediate from the definition that for any $x\in\Z^2$ and $i,j\in I_{k,\ell}^+(x)$ we have
$$
\geo(x+i{\bf e}_2,x+i{\bf e}_2+v_k)\cap \geo(x+j{\bf e}_2,x+j{\bf e}_2+v_k)=\emptyset.
$$
Hence, the set $\net_{k,\ell}(x):=\{\geo(x+i{\bf e}_2,x+i{\bf e}_2+v_k):i\in I_{k,\ell}^+(x)\}$ defines a family of disjoint finite geodesics.

We shall require some knowledge about the density of geodesics in $\net_{k,\ell}(x)$. By Lemmas~\ref{frosted} and~\ref{vashti} it follows that $I_{k,\ell}(x)=I_{k,\ell}^-(x)\cup I_{k,\ell}^+(x)$. Combined with~\eqref{goodie} we conclude that for some $\ell$ each large enough $k$ we have that either $\P\big(0\in I_{k,\ell}^-(x)\big)$ or $\P\big(0\in I_{k,\ell}^+(x)\big)$ exceeds $\delta'/4$. For one of the two this occurs for infinitely many $k$. Since the remainder of the argument is identical in both cases, we proceed assuming, possibly after restricting to a subsequence, that there exists $\ell$ such that for all large $k$ we have
\begin{equation}\label{I-density}
\P\big(0\in I_{k,\ell}^+(x)\big)>\delta'/4.
\end{equation}

We introduce a variable that counts the number of geodesics in $\net_{k,\ell}(x)$ that cross a vertical line segment at $y$ as follows: For $m\ge1$ and $y\in x+[0,v_k\cdot{\bf e}_1]{\bf e}_1$ let
$$
X_{k,\ell}^m(y):=\#\big\{i\in I_{k,\ell}^+(x):\geo(x+i{\bf e}_2,x+i{\bf e}_2+v_k)\cap(y+(-m\ell,m\ell]{\bf e}_2)\neq\emptyset\big\}.
$$
Importantly, the following estimate holds uniformly in $k$ and $m$, and thus guarantees that the family $\net_{k,\ell}(x)$ does not become sparser as $k$ increases.

\begin{lemma}\label{density claim}
Assume~\eqref{I-density}. Then, there exists $\ell\ge1$ so for all large $k$ and all $m$ we have
$$
\P\big(X_{k,\ell}^m(y)>\delta' m/8\big)>(\delta')^2/256\ell^2.
$$
\end{lemma}

\begin{proof}
By invariance with respect to shifts along the vector $\ell{\bf e}_2$ it follows from the ergodic theorem that, almost surely,
$$
\lim_{M\to\infty}\frac{1}{2M+1}\sum_{j=-M}^MX_{k,\ell}^m(y+j2m\ell{\bf e}_2)\,=\,\E[X_{k,\ell}^m(y)].
$$
However, since every path from a site $x+i{\bf e}_2$ to $x+i{\bf e}_2+v_k$ has to cross the vertical line $y+\Z{\bf e}_2$ we also have, as $M\to\infty$, that
$$
\frac{1}{2M+1}\sum_{j=-M}^MX_{k,\ell}^m(y+j2m\ell{\bf e}_2)\,\ge\,\frac{1}{2M+1}X_{k,\ell}^{(2M+1)m}(y)\,\to\, m\,\P\big(0\in I_{k,\ell}^+(x)\big).
$$
Hence $\E[X_{k,\ell}^m(y)]>\delta' m/4$. Since $X_{k,\ell}^m(y)$ is bounded above by $2m\ell$, it follows from the Paley-Zygmund inequality that
$$
\P\big(X_{k,\ell}^m(y)>\delta' m/8\big)\,\ge\,\frac{1}{4}\frac{\E[X_{k,\ell}^m(y)]^2}{\E[X_{k,\ell}^m(y)^2]}\,>\,\frac14\frac{(\delta' m/4)^2}{(2m\ell)^2}\,=\,\frac{(\delta')^2}{256\ell^2},
$$
as required.
\end{proof}

\subsection{Constructing a measure on non-coalescing geodesics}

We will in this section take the set of finite and disjoint geodesics $\net_{k,\ell}(x):=\{\geo(x+i{\bf e}_2,x+i{\bf e}_2+v_k):i\in I_{k,\ell}^+(x)\}$ and construct a measure on non-crossing geodesic with infinitely many coalescence classes.

Let $J_{k,\ell}=\{-3r_k/4,\ldots,-r_k/4\}\times\{0,\ldots,\ell-1\}$, where $r_k=\lfloor v_k\cdot {\bf e}_1\rfloor$. Note that $r_k$ will tend to infinity with $k$ by assumption. Denote by $\bar\Ec^2$ the set of oriented edges of the $\Z^2$ lattice. We encode the family of geodesics $\net_{k,\ell}(x)$ as follows: Let $\eta_{k,\ell}(x)=(\eta_{k,\ell}(x,e))_{e\in\bar\Ec^2}$ be defined as
\begin{equation*}
\eta_{k,\ell}(x,e):=\left\{
\begin{aligned}
& 1 && \text{if $e=(y,z)$ is crossed from $y$ to $z$ by some }g\in\net_{k,\ell}(x),\\
& 0 && \text{otherwise}.
\end{aligned}
\right.
\end{equation*}
\begin{figure}[htbp]
\begin{center}
\includegraphics{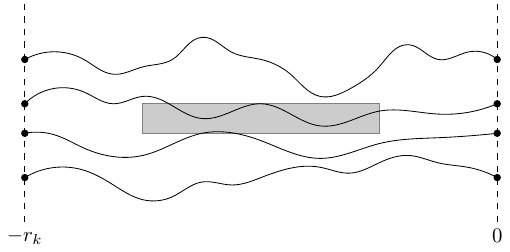}
\end{center}
\caption{Encoding of finite geodesics. Averaging over positions inside $J_{k,\ell}$, the shaded region in this figure.}
\label{fig:mid4}
\end{figure}

We exhibit a measurable map $\Psi_{k,\ell}(x):\Omega_1\to\Omega_1\times\Omega_2'$ via $\omega\mapsto(\omega,\eta_{k,\ell}(x))$. We obtain a measure $\nu_{k,\ell}(x)$ as the push-forward of $\P$ through the mapping $\Psi_{k,\ell}(x)$. Averaging over $x$ in $J_{k,\ell}$ we obtain
$$
\nu_{k,\ell}^\ast:=\frac{1}{|J_{k,\ell}|}\sum_{x\in J_{k,\ell}}\nu_{k,\ell}(x).
$$
By compactness the sequence $(\nu_{k,\ell}^\ast)_{k\ge1}$ has a convergent subsequence. Let $\nu$ be the limiting measure of some convergent subsequence. See Figure~\ref{fig:mid4}.

\begin{lemma}\label{have a cigar}
Assume that~\eqref{I-density} holds. Then, every sub-sequential limit $\nu$ is invariant with respect to translations, and for $\nu$-almost every $(\omega,\eta)\in\Omega_1\times\Omega_2'$ we have that
\begin{enumerate}[\quad (a)]
\item every finite directed path in the graph encoded by $\eta$ is a geodesic;
\item every $z\in\Z^2$ has both in- and out-degree either 0 or 1 in $\eta$;
\item there are no cycles in the (undirected) graph encoded by $\eta$.
\end{enumerate}
\end{lemma}

\begin{proof}
We first aim to show that $\nu\circ\tilde\sigma_z=\nu$ for every $z\in\Z^2$. Hence, fix some sequence $(k_j)_{j\ge1}$ such that $\nu_{k_j,\ell}^\ast\to\nu$ weakly. Then, by continuity of $\tilde\sigma_z$, also $\nu_{k_j,\ell}^\ast\circ\tilde\sigma_z\to\nu\circ\tilde\sigma_z$. It follows from the periodicity of the set $I_{k,\ell}^+$ that $\nu_{k,\ell}(x)\circ\tilde\sigma_{\ell{\bf e}_2}=\nu_{k,\ell}(x+\ell{\bf e}_2)=\nu_{k,\ell}(x)$. Hence, for any bounded continuous function $f:\Omega_1\times\Omega_2'\to\R$ we have that
$$
\big|\nu_{k,\ell}^\ast\circ\tilde\sigma_z(f)-\nu_{k,\ell}^\ast(f)\big|=\frac{1}{|J_{k,\ell}|}\Big|\sum_{x\in J_{k,\ell}}[\nu_{k,\ell}(x+z_1{\bf e}_1)](f)-[\nu_{k,\ell}(x)](f)\Big|\le\frac{2|z_1|\ell}{|J_{k,\ell}|}\max|f|,
$$
which tends to zero as $k\to\infty$. It follows that $\nu\circ\tilde\sigma_z=\nu$.

We proceed with the proof of properties~\emph{(a)-(c)}. Let $\gamma$ be a finite directed path in $\Z^2$, and denote by $A_\gamma$ the event that $\gamma$ is a path in the (directed) graph encoded by $\eta$. $A_\gamma$ is defined in terms of the state of finitely many coordinates in $\eta$ and is thus both open and closed. So is its complement, and also the event $B_\gamma=\{(\omega,\eta):\gamma\text{ is a geodesic}\}$ is closed. Since each path in $\net_{k,\ell}(x)$ is a geodesic, and each pair in $\net_{k,\ell}(x)$ is non-crossing, it follows that $\nu_{k,\ell}^\ast\big(A_\gamma^c\cup(A_\gamma\cap B_\gamma)\big)=1$ for each $k$. By the Portmanteau theorem it follows that also $\nu\big(A_\gamma^c\cup(A_\gamma\cap B_\gamma)\big)=1$. Since the number of finite directed paths is countable, this proves part~\emph{(a)}.

For part~\emph{(b)}, let $C_d(z)$ denote the event that $z$ has both in- and out-degree $d$ in the directed graph encoded by $\eta$. $C_d(z)$ are defined in terms of edges adjacent to $z$ and is thus closed. For fixed $z\in\Z^2$ and $k$ large enough we have $\nu_{k,\ell}^\ast(C_0(z)\cup C_1(z))=1$, since the contrary would imply that with positive probability $z$ would either be an endpoint to some path in $\net_{k,\ell}(x)$ or a point of intersection for two paths in $\eta_{k,\ell}(x)$, for some large value of $|x|$. Hence, also $\nu(C_0(z)\cup C_1(z))=1$, and since $\Z^2$ is countable, part~\emph{(b)} follows.

Finally, let $\gamma$ be a finite undirected cycle and let $D_\gamma$ denote the event that $\gamma$ is a cycle in the (undirected) graph encoded by $\eta$. Again, $D_\gamma$ is both open and closed, and so is its complement. Moreover, $\nu_{k,\ell}^\ast(D_\gamma)=0$ since for each $k$ each path in $\net_{k,\ell}(x)$ has almost surely no loops or crossing paths. Consequently, $\nu(D_\gamma)=0$, and since the set of finite cycles in $\Z^2$ is countable, this proves part~\emph{(c)}.
\end{proof}

It follows by Lemma~\ref{have a cigar} that each site in the graph encoded by $\eta$ has $\nu$-almost surely either out-degree 0 or a unique infinite forwards-path.\footnote{There would also have to be a unique infinite backwards-path, and hence a bigeodesic, but this observation will not be important in what follows.} In the case that $z$ has out-degree 1 in $\eta$, let $\gamma_z(\eta)$ denote the infinite forwards-path starting at $z$. In the case that $z$ has out-degree 0, then set $\gamma_z(\eta)=\{z\}$.

We next anticipate the proof of Lemma~\ref{cunningham} and show how to deduce Theorem~\ref{hall of mirrors}.

\begin{proof}[Proof of Theorem~\ref{hall of mirrors}]
If~\eqref{sentencing} holds, then we obtain~\eqref{goodie} via Lemma~\ref{cunningham}, and may without loss of generality assume that~\eqref{I-density} holds. By Lemma~\ref{have a cigar} each subsequential limit $\nu$ of the sequence $(\nu_{k,\ell}^\ast)_{k\ge1}$ is invariant with respect to translations. Moreover, the measures $\nu_{k,\ell}^\ast$ are invariant with respect to vertical shifts. Hence, for large enough $k$ it follows via~\eqref{I-density} that
$$
\nu_{k,\ell}^\ast\big(0\text{ has out-degree at least }1\big)\,\ge\,\frac{1}{\ell}\P\big(0\in I_{k,\ell}^+(0)\big)\,>\,\frac{\delta'}{4\ell}.
$$
So, by the Portmanteau theorem,
$\nu$ does not put all mass on the all-zero configuration.

With probability one, by part~\emph{(b)} of Lemma~\ref{have a cigar} each site in the graph encoded by $\eta$ has either degree zero or out-degree 1, and by part~\emph{(a)} each infinite forward-path is a geodesic. Since out-degrees are at most 1, it follows that any two forward-paths that intersect must coalesce, so $\nu$ is indeed an shift-invariant measure on non-crossing geodesics.

Given integers $M>m\ge1$ and $n\ge1$ let $A_n^{m,M}$ denote the event that there are at least $n$ points in $V_m:=\{z\in\Z^2:\gamma_z(\eta)\neq\{z\}\}\cap[-m,m]{\bf e}_2$ whose forward-paths remain pairwise disjoint inside $[-M,M]^2$. By Lemma~\ref{density claim} we may for every $n\ge1$ find an $m\ge1$ such that for all $M>m$ we have $\nu_{k,\ell}(A_n^{m,M})>(\delta')^2/256\ell^2$. Since $A_n^{m,M}$ is defined in terms of finitely many edges it is closed, so by the Portmanteau theorem also $\nu(A_n^{m,M})\ge(\delta')^2/256\ell^2$. By continuity of measure, since $A_n^{m,M}$ is decreasing in $M$, we conclude that
$$
\nu\big(V_m\text{ contains at least $n$ coalescence classes}\big)\,=\,\lim_{M\to\infty}\nu(A_n^{m,M})\,\ge\,\frac{(\delta')^2}{256\ell^2}.
$$
Since $n\ge1$ was arbitrary, we conclude that $\nu$ puts positive mass on configurations with infinitely many coalescence classes. 
This is a contradiction to Proposition~\ref{p:sparse measure}, and hence proves Theorem~\ref{hall of mirrors}. 
\end{proof}

The remainder of this section is devoted to prove Lemma~\ref{cunningham}, and hence complete the proof of Theorem~\ref{hall of mirrors}.

\subsection{Interlude on finite geodesics}\label{fred rogers}

Before we continue we shall examine finite geodesics in a few lemmas.
Recall the definition of neighboring geodesics.

\begin{lemma}\label{neighboring 2}
Let $G<G'$ be neighboring random coalescing geodesics.
Then for every $m\ge1$ we have
$$
\limsup_{|z|\to\infty}\P\big(\geo(0,z)\text{ is ccw between $G$ and $G'$ and diverges within $m$ steps}\big)=0.
$$
\end{lemma}

\begin{proof}
We argue by contradiction. Suppose there are $\eps>0$, $m\ge1$ and a sequence $(z_k)_{k\ge1}$ such that $|z_k|\to\infty$ for which
$$
\sup_{k\ge1}\P\big(\geo(0,z_k)\text{ ccw between $G$ and $G'$ and diverges in at most $m$ steps}\big)>\eps.
$$
In that case there exists an infinite geodesic which coincides with either of $G$ and $G'$ for at most $m$ steps with probability at least $\eps$. This contradicts the assumption that $G$ and $G'$ are neighboring.
\end{proof}

We counter the above lemma by showing that if $\geo(y,z)$ goes through the origin, then it cannot coincide with a given random coalescing geodesics for very long.

\begin{lemma}\label{higgs boson}
Let $\mathcal{F}$ be a finite family of random coalescing geodesics. For every $\eps>0$ there exists $m\ge1$ such that for all large $|y|$ and $|z|$
$$
\P\big(0\in\geo(y,z)\text{ and $\geo(y,z)$ coincides with some $G\in\mathcal{F}$ for at least $m$ steps}\big)<\eps.
$$
\end{lemma}

\begin{proof}
Let $G$ be a random coalescing geodesic and set
\begin{equation*}
\begin{aligned}
V_G&:=\big\{z\not\in G(0):\exists g\in\tree_z\text{ such that }G(0)\subseteq g\big\},\\
U_G&:=\big\{z\in G(0):\exists y\not\in V_G\text{ such that }\geo(0,z)\subseteq\geo(y,z)\big\}.
\end{aligned}
\end{equation*}
In words, $V_G$ is the set of vertices to which the geodesic $G(0)$ can be extended backwards, and $U_G$ is the set of vertices $z$ on $G(0)$ for which the finite segment $\geo(0,z)$ can be extended backwards beyond $V_G$.

By Proposition~\ref{temecula}, $V_G$ is almost surely finite. Consequently, the outer boundary of $V_G$ is almost surely finite, and so is $U_G$. Next we note that if the event in the lemma occurs, then either $y$ or $z$ is contained in $V_G$ for some $G\in\mathcal{F}$, or the first $m$ points (counting from the origin) on either $\geo(0,z)$ or $\geo(y,0)$ belong to $U_G$ for some $G\in\mathcal{F}$. That this would happen is increasingly unlikely if $|y|$, $|z|$ and $m$ are all large.
\end{proof}

\begin{lemma}\label{lma:slices}
Let $G$ be a random coalescing geodesic and let $\bar G$ denote the random coalescing geodesic obtained from $G$ via rotation by $\pi$. If $G'$ is a random coalescing geodesic such that $G<G'<\bar G$, then
$$
\limsup_{|y|,|z|\to\infty}\P\big(0\in\geo(y,z)\text{ and $y$, $z$ are both ccw between $G$ and $G'$}\big)=0.
$$
\end{lemma}

\begin{proof}
Since $G<G'<\bar G$ the intersection of the two sets $\{x\in\Z^2:\rho_G(x)\ge0\}$ and $\{x\in\Z^2:\rho_{G'}(x)\ge0\}$ forms an (infinite) sector of which all but finitely many points lie counterclockwise between $G$ and $G'$, almost surely. Since the Busemann functions of $G$ and $G'$ are asymptotically linear to $\rho_G$ and $\rho_{G'}$ respectively (Proposition~\ref{sonoma}), it follows that
$$
\big\{x\in\Z^2:x\text{ lies ccw between $G$ and $G'$, $B_G(0,x)\ge0$ and }B_{G'}(0,x)\ge0\big\}
$$
contains infinitely many points almost surely. Hence, we can find $x\in G$ and $x'\in G'$ for which $\geo(x',x)$ does not go through the origin, but is strictly contained in the region counterclockwise between $G$ and $G'$ (apart from its endpoints); see Figure~\ref{fig:mid4,5}.
\begin{figure}[htbp]
\begin{center}
\includegraphics{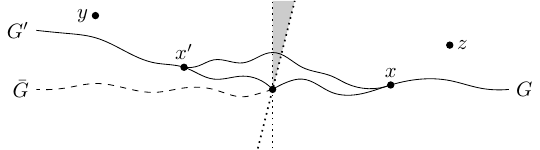}
\end{center}
\caption{The geodesic between $x'$ and $x$ does not visit the origin, and prevents the geodesic between $y$ and $z$ from doing the same. The shaded regions consists of all points for which both $\rho_G$ and $\rho_{G'}$ are non-negative.}
\label{fig:mid4,5}
\end{figure}
There is an almost surely finite $K\ge1$ so that this path is fully contained in $[-K,K]^2$. Finally we observe that for $y$ and $z$ counterclockwise between $G$ and $G'$, but outside $[-K,K]^2$, the geodesic $\geo(y,z)$ cannot go through the origin.
\end{proof}

\subsection{Proof of Lemma~\ref{cunningham}}\label{sec:cunningham}

We are now prepared to prove Lemma~\ref{cunningham} and hence complete the solution to the midpoint problem. We shall divide the proof into two cases, depending on whether the asymptotic shape has many sides or few, in which case the shape is a polygon. We begin with the easier case when it has at least 40 sides.\\

\emph{Case 1: The shape has many sides.}
Assume that $\ball$ has at least 40 sides. From the work of Damron and Hanson~\cite{damhan14} we obtain in Theorem~\ref{yahoo} the existence of (at least) 40 random coalescing geodesics, all of which are asymptotically directed into distinct sides of $\ball$. We shall let $\mathcal{F}$ denote such a set of 40 random coalescing geodesics with the property that if $G$ is in $\mathcal{F}$ then so is the geodesic obtained from $G$ by rotation $\pi/2$. Among these 40 geodesics at most two may contain $v$ in its set of asymptotic directions, and any consecutive sequence of ten geodesics in $\mathcal{F}$ can span an angle at most $\pi/2$.

Let $G^{+1},G^{+2},\ldots,G^{+40}$ be a counterclockwise enumeration of the geodesics in $\mathcal{F}$, starting with the counterclockwise-most geodesic that contains $v$ is its set of directions (should such a geodesic exist, otherwise start with the first geodesic counterclockwise of $v$). Similarly we let $G^{-1},G^{-2},\ldots,G^{-40}$ be a clockwise enumeration of the geodesics starting with the clockwise-most geodesic containing $v$.
With probability tending to one $v_k$ will line counterclockwise between $G^{-2}$ and $G^{+2}$. By Lemma~\ref{lma:slices}, with probability tending to zero, we have $0\in\geo(u_k,v_k)$ while $u_k$ lie counterclockwise between $\bar G^{+3}$ and $\bar G^{-3}$, where again $\bar G^{\pm i}$ is the geodesic obtained from $G^{\pm i}$ by rotation $\pi$. In particular, both $v$ and $-u$ has to lie strictly counterclockwise between $G^{-5}$ and $G^{+5}$. Since these geodesics cannot span an angle larger than $\pi/2$ it is possible to find a half-plane $H_i$ which $G^{-5}$ and $G^{+5}$ eventually move into. These geodesics and the two obtained via rotation $\pi$ together give a set of random coalescing geodesics for which Lemma~\ref{cunningham} holds.\\


\emph{Case 2: The shape has few sides.}
We shall, for the remainder of this section, assume that~\eqref{sentencing} holds and that $\ball$ is a polygon (possibly with few sides). The argument will in this case be extensive and will have to accommodate a range of possibilities, such as whether the geodesics we consider are clockwise isolated or clockwise dense.

Fix $\eps\in(0,\delta/100)$. If $v$ is a direction of differentiability of $\ball$, then let $G^+=G^-$ denote the (unique) random coalescing geodesic associated with the tangent line in direction $v$. If instead $\ball$ has a corner at $v$, then let $G^-$ denote the random coalescing geodesic associated with the side clockwise of $v$ and define $G^+$ idem counterclockwise of $v$; see Figure~\ref{fig:mid5}. Next we choose an increasing sequence of geodesics $(G_n^-)_{n\ge1}$ as follows: If $G^-$ is cw-isolated, then let $G_n^-$ denote its clockwise neighbor. If $G^-$ is cw-dense, then pick $G_n^-<G^-$ so that
\begin{equation}\label{eq:Gn-}
\P\big(\text{$G_n^-$ and $G^-$ diverge within $n$ steps}\big)<\eps.
\end{equation}
By Proposition~\ref{bordeaux} there is a unique geodesic for each tangent line of $\ball$, so there is no restriction to assume that each geodesic in the sequence $(G_n^-)_{n\ge1}$ has the same set of limiting directions; either they are all directed in the corner clockwise of the side associated to $G^-$, or they all coincide with the geodesic associated with the side clockwise of that corner. We choose similarly a decreasing sequence $(G_n^+)_{n\ge1}$ counterclockwise of $G^+$.
\begin{figure}[htbp]
\begin{center}
\includegraphics{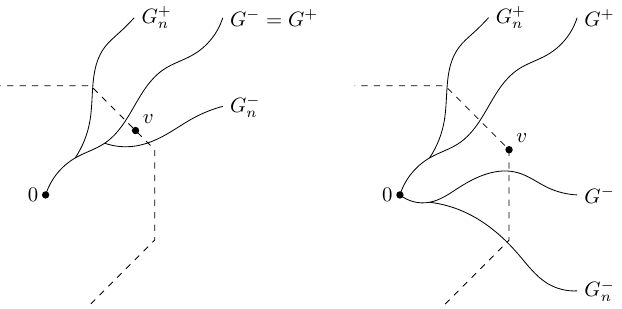}
\end{center}
\caption{The selected geodesics. The dashed curve illustrates the asymptotic shape, which here is an octagon.}
\label{fig:mid5}
\end{figure}

The geodesics specified so far satisfy the following properties.

\begin{claim}\label{G-claim1}
For every $n\ge1$ we have
\begin{enumerate}[\quad (a)]
\item $\displaystyle\lim_{k\to\infty}\P\big(\geo(0,v_k)\text{ is ccw between $G_n^+$ and $G_n^-$}\big)=0$;
\item $\displaystyle\limsup_{|z|\to\infty}\P\big(\geo(0,z)\text{ is ccw between $G_n^-$ and $G^-$ and diverges within $n$ steps}\big)<\eps$;
\item $\displaystyle\limsup_{|z|\to\infty}\P\big(\geo(0,z)\text{ is ccw between $G^+$ and $G_n^+$ and diverges within $n$ steps}\big)<\eps$.
\end{enumerate}
\end{claim}

\begin{proof}[Proof of claim]
Part~\emph{(a)} is an immediate consequence of $v$ not being a limiting direction of neither $G_n^-$ nor $G_n^+$. For part~\emph{(b)}, note that if $G^-$ is cw-dense, then this follows from~\eqref{eq:Gn-}. If $G^-$ is cw-isolated, then it follows from Lemma~\ref{neighboring 2}. The proof of part~\emph{(c)} is identical.
\end{proof}

Combining Lemma~\ref{higgs boson} and Claim~\ref{G-claim1} we see that $\geo(0,v_k)$ has nowhere to go but between $G^-$ and $G^+$, and which of course implies that $\ball$ must have a corner at $v$.

\begin{claim}\label{v is a corner}
If~\eqref{sentencing} holds, then $\ball$ has a corner in direction $v$. Moreover, there exists $m\ge1$ such that for all sufficiently large $k$ we have
$$
\P\big(0\in\geo(u_k,v_k)\text{ and $\geo(0,v_k)$ ccw between $G^-$, $G^+$ and diverges in $m$ steps}\big)>\delta-6\eps.
$$
\end{claim}

\begin{proof}[Proof of claim]
Let $\mathcal{F}=\{G_1^-,G^-,G^+,G_1^+\}$. Recall Lemma~\ref{higgs boson}, and pick $m$ such that
$$
\P\big(0\in\geo(u_k,v_k)\text{ and $\geo(u_k,v_k)$ coincides with some $G\in\mathcal{F}$ for at least $m$ steps}\big)<\eps
$$
for all large enough $k$.

Now we observe that on the event that $0\in\geo(u_k,v_k)$ there are five possibilities for $\geo(0,v_k)$: Either $\geo(0,v_k)$ is (i) ccw between $G_m^+$ and $G_m^-$; (ii) ccw between $G_m^-$ and $G^-$ and diverges within $m$ steps; (iii) ccw between $G^+$ and $G_m^+$ and diverges within $m$ steps; (iv) coincides with either of $G_m^-,G^-,G^+,G_m^+$ for at least $m$ steps; or (v) $\geo(0,v_k)$ lies ccw between $G^-$ and $G^+$ and splits from both within $m$ steps. The first three each have probability at most $\eps$ to occur by Claim~\ref{G-claim1}, while the fourth has probability at most $3\eps$ to occur by the choice of $m$ and~\eqref{eq:Gn-}. The only other possibility is that $\geo(0,v_k)$ lies counterclockwise between $G^-$ and $G^+$ and splits from both within $m$ steps. However, this can only happen if $G^-\neq G^+$, in which case $\ball$ must have a corner in direction $v$.
\end{proof}

Combining Lemma~\ref{neighboring} and Claim~\ref{v is a corner} we conclude that $G^-$ and $G^+$ are not neighboring, and that there are random coalescing geodesics between $G^-$ and $G^+$. We now choose a decreasing sequence $(G'_n)_{n\ge1}$ as follows: If $G^-$ is ccw-isolated, then let $G'_n$ denote its counterclockwise neighbor. If $G^-$ is ccw-dense, then pick $G'_n$ strictly ccw between $G^-$ and $G^+$ such that
\begin{equation}\label{eq:G'n}
\P\big(\text{$G^-$ and $G'_n$ diverge within $n$ steps}\big)<\eps.
\end{equation}
Choose an increasing sequence $(G''_n)_{n\ge1}$ analogously, with the additional condition that $G''_1\ge G'_1$. All of these geodesics $(G'_n)_{n\ge1}$ and $(G''_n)_{n\ge1}$ are directed in the direction of non-differentiability $v$.

\begin{claim}\label{G-claim2}
For every $n\ge1$ we have
\begin{enumerate}[\quad (a)]
\item $\displaystyle\limsup_{|z|\to\infty}\P\big(\geo(0,z)\text{ is ccw between $G^-$ and $G'_n$ and diverges within $n$ steps}\big)<\eps$;
\item $\displaystyle\limsup_{|z|\to\infty}\P\big(\geo(0,z)\text{ is ccw between $G''_n$ and $G^+$ and diverges within $n$ steps}\big)<\eps$.
\end{enumerate}
\end{claim}

\begin{proof}[Proof of claim]
Note that if $G^-$ is ccw-dense, then this follows from~\eqref{eq:G'n}. If $G^-$ is ccw-isolated, then it follows from Lemma~\ref{neighboring 2}. The remaining case is similar.
\end{proof}

\begin{figure}[htbp]
\begin{center}
\includegraphics{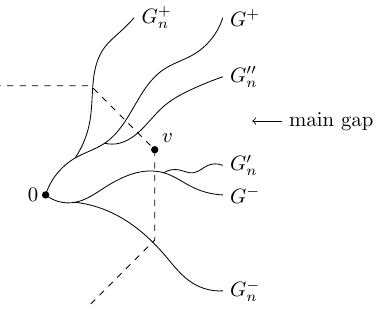}
\end{center}
\caption{The family $\mathcal{F}_n$ of geodesics. The main gap is indicated. The remaining four gaps are minor.}
\label{fig:mid6}
\end{figure}

Given a random coalescing geodesic $G$, let $\bar G$ denote the random coalescing geodesic obtained by rotation by $\pi$. We have now defined the geodesics with which we will work. Set
\begin{equation*}
\begin{aligned}
\mathcal{F}_n&:=\{G_n^-,G^-,G'_n,G''_n,G^+,G^+_n\},\\
\bar{\mathcal{F}}_n&:=\{\bar G_n^-,\bar G^-,\bar G'_n,\bar G''_n,\bar G^+,\bar G^+_n\}.
\end{aligned}
\end{equation*}

We will refer to the counterclockwise gap between $G'_n(x)$ and $G''_n(x)$ as the {\bf main gap} of $\mathcal{F}_n$ at $x$, and say that $\geo(x,x+v_k)$ {\bf moves into the main gap} of $\mathcal{F}_n$ at $x$ if $\geo(x,x+v_k)$ is counterclockwise between $G'_n(x)$ and $G''_n(x)$. The gaps between $G_n^-(x)$ and $G^-(x)$, between $G^-(x)$ and $G'_n(x)$, between $G''_n(x)$ and $G^+(x)$, and between $G^+(x)$ and $G^+_n(x)$ will be referred to as {\bf minor gaps}. Define $\dining_{k,n}(x)$ to be the event that following three occur:
\begin{itemize}
\item $x \in \geo(x+u_k,x+v_k)$;
\item $\geo(x,x+v_k)$ moves into the main gap of $\mathcal{F}_n$ at $x$;
\item $\geo(x+u_k,x)$ moves into the main gap of $\bar{\mathcal{F}}_n$ at $x$.
\end{itemize}
Clearly, $\dining_{k,n}$ can only occur for large $k$ if $u=-v$.

\begin{claim}\label{wall or walrus}
Assume that~\eqref{sentencing} holds. Then $u=-v$ and there exists $n_0\ge1$ such that for all $n\ge n_0$ and sufficiently large $k$ we have
$$
\P\big(\dining_{k,n}(0)\big)>\delta/2.
$$
\end{claim}

\begin{proof}[Proof of claim]
Via Lemma~\ref{higgs boson} we know that for some $m\ge1$ we have for all large $k$ that
$$
\P\big(0\in\geo(u_k,v_k)\text{ and $\geo(0,v_k)$ coincides with some $G\in\mathcal{F}_1$ for at least $m$ steps}\big)<\eps.
$$
As in the proof of Claim~\ref{v is a corner}, but this time with seven possibilities for $\geo(0,v_k)$, we reach the conclusion that for all $n\ge m$ we have
$$
\lim_{k\to\infty}\P\big(0\in\geo(u_k,v_k)\text{ and $\geo(0,v_k)$ ccw between $G'_n$ and $G''_n$}\big)>\delta-10\eps.
$$
In particular, by Lemma~\ref{neighboring}, the geodesics $G'_n$ and $G''_n$ cannot be neighboring.

In order to get also $\geo(u_k,0)$ between $\bar G'_n$ and $\bar G''_n$ we shall iterate the above argument and apply Lemma~\ref{lma:slices}.
We pick a decreasing sequence $(G^\ast_n)_{n\ge1}$ as follows: If $G^-$ was ccw-dense, then let $G^\ast_n=G'_n$. Otherwise, if $G'_1$ is ccw-isolated, let $G^\ast_n$ be its counterclockwise neighbor, and if not, then let $(G^\ast_n)_{n\ge1}$ be a decreasing sequence such that $G'_1$ and $G^\ast_n$ diverge within $n$ steps with probability less than $\eps$. Pick similarly an increasing sequence $(G^{\ast\ast}_n)_{n\ge1}$ such that $G^\ast_1\le G^{\ast\ast}_1$. All these geodesics are directed in the corner at $v$.

Let $\mathcal{F}'_1=\mathcal{F}_1\cup\{G^\ast_1,G^{\ast\ast}_1\}$ and for $n\ge2$ let $\mathcal{F}'_n=\mathcal{F}_n\cup\{G^\ast_{n-1},G^{\ast\ast}_{n-1}\}$. By Lemma~\ref{higgs boson} we find $m'\ge m$ such that for all large $k$ we have
$$
\P\big(0\in\geo(u_k,v_k)\text{ and $\geo(0,v_k)$ coincides with some $G\in\mathcal{F}'_1$ for at least $m'$ steps}\big)<\eps.
$$
As in the proof of Claim~\ref{v is a corner}, but this time with nine possibilities for $\geo(0,v_k)$, based on the geodesics in $\mathcal{F}'_n$, we conclude that for all $n\ge m'$ we have
$$
\lim_{k\to\infty}\P\big(0\in\geo(u_k,v_k)\text{ and $\geo(0,v_k)$ ccw between $G^\ast_n$ and $G^{\ast\ast}_n$}\big)>\delta-14\eps.
$$
Let $n_0=m'$. Again by Lemma~\ref{higgs boson} we find $m''\ge m'$ such that for all large $k$ we have
$$
\P\big(0\in\geo(u_k,v_k)\text{ and $\geo(u_k,0)$ coincides with some $G\in\bar{\mathcal{F}}'_{n_0}$ for at least $m''$ steps}\big)<\eps.
$$
Hence, for large values of $k$, on the event that $0\in\geo(u_k,v_k)$, $\geo(0,v_k)$ is ccw between $G^\ast_{n_0}$ and $G^{\ast\ast}_{n_0}$, and $\geo(u_k,0)$ diverges from the geodesics in $\bar{\mathcal{F}}'_{n_0}$ within $m''$ steps, then Lemma~\ref{lma:slices} says that $\geo(u_k,0)$ has to go between $\bar G'_{n_0}$ and $\bar G''_{n_0}$. More precisely, we conclude that for all large enough $k$ we have
$$
\P\big(\dining_{k,n_0}(0)\big)>\delta-16\eps.
$$
In particular $u=-v$, and since $\dining_{k,n}(0)$ is monotone in $n$, the claim follows.
\end{proof}

Define $\stool_{k,n}(x)$ to be the event that
\begin{itemize}
\item $\dining_{k,n}(x)$ occurs;
\item $\geo(x+u_k,x)$ moves into the main gap of $\mathcal{F}_n$ at $x+u_k$;
\item $\geo(x,x+v_k)$ moves into the main gap of $\bar{\mathcal{F}}_n$ at $x+v_k$.
\end{itemize}
Note that $\stool_{k,n}(x)$ in nothing but a version of the event $\good_k(x)$ for a specific set of geodesics $G^i$ moving into direction either $v$ or $u=-v$.

We further define $\stool_{k,n}^-(x)$ to be the event that
\begin{itemize}
\item $\dining_{k,n}(x)$ occurs;
\item $\geo(x+u_k,x)$ moves into the main gap of $\mathcal{F}_n$ at $x+u_k$;
\item $\geo(x,x+v_k)$ moves into a minor gap of $\bar{\mathcal{F}}_n$ at $x+v_k$,
\end{itemize}
and define $\stool_{k,n}^+(x)$ analogously, interchanging the location of `the main' and `a minor'.

\begin{claim}\label{futurama}
If~\eqref{sentencing} holds, then there exist $\delta'>0$ and $n\ge1$ such that for all large $k$
$$
\P\big(\stool_{k,n}(0)\cup\stool_{k,n}^-(0)\cup\stool_{k,n}^+(0)\big)>\delta'.
$$
\end{claim}

\begin{proof}[Proof of claim]
We may without restriction assume that $v$ is strictly contained in $H_0$, the right half-plane. Denote by $A_{k,\ell,n}(x)$ the event that $\dining_{k,n}(x)$ and the following occur:
\begin{itemize}
\item $\big(G'_n(x)\cup G''_n(x)\big) \cap (x+H_0^c) \subseteq x+[-\ell/3,\ell/3]^2$;
\item $\big(\bar G'_n(x)\cup \bar G''_n(x)\big) \cap (x+H_0) \subseteq x+[-\ell/3,\ell/3]^2$.
\end{itemize}
Since $G'_n$ and $G''_n$ have direction $v$ they eventually move into $H_0$. Hence, due to Claim~\ref{wall or walrus} we can make the probability of $A_{k,\ell,n}$ as close to $\delta/2$ as we like by increasing $\ell$.
Let $\Z_M=\{0,1,\ldots,M-1\}$ and $S_{k,\ell,n}^M:=\{i\in\ell\Z_M: A_{k,\ell,n}(i{\bf e}_2)\text{ occurs}\}$. We shall next fix a whole slew of parameters as follows:
\begin{enumerate}[\quad (i)]
\item Fix $\ell$, $n_0$ and $k_0$ so that $\P\big(A_{k,\ell,n}(x)\big)>\delta/4$ for all $n\ge n_0$ and $k\ge k_0$.
\item Mimicking the argument of Lemma~\ref{density claim} we find $M$ large so that $\P\big(\#S_{k,\ell,n}^M>36\big)>1/2$ for all $n\ge n_0$ and $k\ge k_0$.
\item Since the geodesics in $\mathcal{F}_n$ are coalescing, we fix $m$ and $n_1\ge n_0$ so that
$$
\P\big(\text{$G(i{\bf e}_2)$ and $G(j{\bf e}_2)$ coalesce in $m$ steps $\forall i,j\in\ell\Z_M, G\in\mathcal{F}_n, n\ge n_1$}\big)
$$
is at least $1-\eps$. Note that the bound is uniform in $n\ge n_1$ since if $G^-(x)$ and $G^-(y)$ coalesce within $m$ steps and $G^-_{n_1}(x)$ and $G^-_{n_1}(y)$ coincide with $G^-(x)$ and $G^-(y)$ for at least $m$ steps, then $G^-_n(x)$ and $G^-_n(y)$ coalesce for all $n\ge n_1$.
\item Pick $n_2\ge n_1$ and $k_1\ge k_0$ so that for all $k\ge k_1$ the probabilities that for some $i\in\ell\Z_M$ the path $\geo(i{\bf e}_2+u_k,i{\bf e}_2)$ moves into a minor gap of $\mathcal{F}_{n_2}$ at $i{\bf e}_2+u_k$ within $m$ steps, and that for some $i\in\ell\Z_M$ the path $\geo(i{\bf e}_2,i{\bf e}_2+v_k)$ moves into a minor gap of $\bar{\mathcal{F}}_{n_2}$ at $i{\bf e}_2+v_k$ within $m$ steps, are at most $\eps$.
\item Finally, select a $k_2\ge k_1$ large enough so that for all $k\ge k_2$ the probabilities that there exists $i\in\ell\Z_M$ such that $\geo(i{\bf e}_2+u_k,i{\bf e}_2)$ is ccw between $G_{n_2}^+(i{\bf e}_2+u_k)$ and $G_{n_2}^-(i{\bf e}_2+u_k)$, and that there exists $i\in\ell\Z_M$ such that $\geo(i{\bf e}_2,i{\bf e}_2+v_k)$ is ccw between $\bar G_{n_2}^+(i{\bf e}_2+v_k)$ and $\bar G_{n_2}^-(i{\bf e}_2+v_k)$, are at most $\eps$.
\end{enumerate}

We continue with an argument similar to that in Lemma~\ref{frosted}. Consider the event that $\# S_{k,\ell,n}^M\ge37$ and that for all $i,j\in S_{k,\ell,n}^M$ we have that
\begin{itemize}
\item $\geo(i{\bf e}_2+u_k,i{\bf e}_2)$ is not ccw between $G^+_n(i{\bf e}_2+u_k)$ and $G^-_n(i{\bf e}_2+u_k)$, nor is $\geo(i{\bf e}_2,i{\bf e}_2+v_k)$ ccw between $\bar G^+_n(i{\bf e}_2+v_k)$ and $\bar G^-_n(i{\bf e}_2+v_k)$;
\item $\geo(i{\bf e}_2+u_k,i{\bf e}_2)$ does not move into a minor gap in $\mathcal{F}_n$ at $i{\bf e}_2+u_k$, nor does $\geo(i{\bf e}_2,i{\bf e}_2+v_k)$ move into a minor gap of $\bar{\mathcal{F}}_n$ at $i{\bf e}_2+v_k$, in less than $m$ steps;
\item $G(i{\bf e}_2+u_k)$ and $G(j{\bf e}_2+u_k)$ coalesce within $m$ steps for all $G\in\mathcal{F}_n$;
\item $\bar G(i{\bf e}_2+v_k)$ and $\bar G(j{\bf e}_2+v_k)$ coalesce within $m$ steps for all $\bar G\in\bar{\mathcal{F}}_n$.
\end{itemize}
For parameters $\ell$ and $M$ specified as above, and for all $n\ge n_2$ and $k\ge k_2$, the above event occurs with probability at least $1/2-6\eps$. We further note that on the above event, then for no two $i\in S_{k,\ell,n}^M$ and no pair $G\in\mathcal{F}_n$, $G'\in\bar{\mathcal{F}}_n$ may $\geo(i{\bf e}_2+u_k,i{\bf e}_2+v_k)$ coincide with both $G(i{\bf e}_2+u_k)$ and $G'(i{\bf e}_2+v_k)$ for as long as $m$ steps. Indeed, the contrary would contradict the assumption of unique passage times; see Figure~\ref{fig:mid7}.
\begin{figure}[htbp]
\begin{center}
\includegraphics{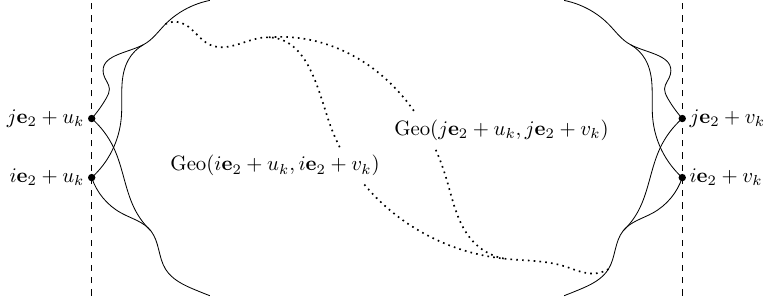}
\end{center}
\caption{Two geodesics that coincide with some pair $(G,G')$ for too long contradicts unique passage times.}
\label{fig:mid7}
\end{figure}
Since there are 36 such combinations of elements from $\mathcal{F}_n$ and $\bar{\mathcal{F}}_n$, but $S_{k,\ell,n}^M$ contains at least 37 elements, then the only possibility is that for some $i\in S_{k,\ell,n}^M$, one of the two ends of $\geo(i{\bf e}_2+u_k,i{\bf e}_2+v_k)$ moves into the main gap, while the other either moves into the main or one of the minor gaps. That is, for all $k\ge k_2$
$$
\P\big(\exists i\in\ell\Z_M:
\stool_{k,n_2}(i{\bf e}_2)\cup\stool_{k,n_2}^-(i{\bf e}_2)\cup\stool_{k,n_2}^+(i{\bf e}_2)\big)>1/4,
$$
from which the conclusion of the claim follows.
\end{proof}

We may now complete the proof of Lemma~\ref{cunningham}.

By Lemma~\ref{futurama}, possibly after restriction to a subsequence, either of the three events $\stool_{k,n}(0)$, $\stool_{k,n}^-(0)$ or $\stool_{k,n}^+(0)$ will occur with probability at least $\delta'/3$ for all large $k$. Recall that $\stool_{k,n}(0)$ is a version of the event $\good_k(0)$ for which the $G^i$'s move into direction $v$ and $-v$. So, if $\P\big(\stool_{k,n}(0)\big)>\delta'/3$ for all large $k$, then there is nothing more to prove.

In the remaining cases we need to verify that there exists a half-plane $H_i$ such that the random coalescing geodesics involved eventually move into either $H_i$ or $H_i^c$. The two cases are symmetric to each other, so we only consider the case of $\stool_{k,n}^-(0)$.

By restricting to a further subsequence we can specify which of the four minor gaps that the event $\stool_{k,n}^-(0)$ occurs for. By Claim~\ref{v is a corner} we have that $\ball$ has a corner at $v$. By assumption, $G^-$ and $G^+$ are directed in flat regions clockwise and counterclockwise of $v$. The geodesics $G'_n$ and $G''_n$ both move in direction $v$. Whether $G_n^-$ and $G_n^+$ are directed in corners or flat pieces may (possibly) depend on the number of sides of the shape.

Consider first the case that $\ball$ has four sides, i.e.\ that $\ball$ is either a square or a diamond. By symmetry, there are then geodesics directed in all four corners, and both $G_n^-$ and $G_n^+$ are thus directed in (opposite) corners. Each of the minor gaps therefore spans an angle of directions at most $\pi/2$. For each minor gap we can thus easily find a half-plane $H_i$ such that $v$ is strictly contained in $H_i$ and the geodesics constituting the minor gap eventually move into $H_i$.

Consider next the remaining case, that $\ball$ has at least eight sides. Each pair of consecutive sides may span an angle of at most $\pi/2$ due to symmetry. Hence, regardless of whether $G_n^-$ and $G_n^+$ are directed in corners or not, each minor gap cannot span an angle of directions greater than $\pi/2$. So, also in this case we may find a half-plane $H_i$ such that $v$ is strictly contained in $H_i$ and which the geodesics constituting the minor gap eventually move into.\\

This completes the proof of Lemma~\ref{cunningham}, and hence the proof of Theorem~\ref{hall of mirrors}.

\section{Open problems}\label{sec: open}

We conclude this work with some open problems. As we mentioned in Section~\ref{bolt}, the shape theorem and the results in~\cite{hoffman08} and~\cite{damhan14} are in some ways best possible for stationary and ergodic first-passage percolation.
There are various ways to test the optimality of the results presented in this paper. Our first question is closely related to the fact that we have very limited understanding for what goes on in corners of the asymptotic shape.

\begin{question} Does there exist a model of first-passage percolation satisfying condition {\bf A1} or {\bf A2} and a direction $v \in S^1$ such that there are two geodesics in direction $v$ almost surely? Can there be infinitely many geodesics in a given direction? Is it possible for $\functionals_\star$ to contain supporting functionals that are not tangent to $\ball$?
\end{question}

\begin{remark}
This question has been partly answered by Alexander and Berger~\cite{aleber18}. Note that in the case such a model exists, $\ball$ has to have a sharp corner in direction $v$. Alexander and Berger provide a model of first-passage percolation for which $\ball$ is an octagon with corners in the coordinate directions, and where all geodesics (at least eight in total) are directed in the coordinate directions.
\end{remark}

The question of corners is interesting in part as it relates to the existence of bigeodesics.

\begin{question} 
Can we use the theory developed here to rule out bigeodesics in a given direction $v\in S^1$, 
even if $\partial\ball$ has a sharp corner in the direction $v$?
\end{question}

An easy consequence of the results of this paper is that every model of first-passage percolation for which $\partial\ball$ is everywhere differentiable, i.e.\ has no corners, then every random coalescing geodesic is almost surely both cw- and ccw-dense. In the case when $\tree_0$ is almost surely finite, then every random coalescing geodesic is cw- and ccw-isolated. It would be interesting to know whether there are models that admit the existence of both. One way to address this question is in terms of labels.

\begin{question} Does there exist a model of first-passage percolation satisfying condition {\bf A1} or {\bf A2} and a label $\alpha$ such that there are two geodesics with label $\alpha$ almost surely?
\end{question}

We have in this paper proved that every geodesic has an asymptotic direction and an asymptotically linear Busemann function. Moreover, for no linear functional $\rho\in\support$ may there be two geodesics with Busemann function asymptotically linear to $\rho$ with positive probability. However, one would expect to be able to find exceptional directions in which this happens. The following question asks whether this description is optimal.

\begin{question}
Is it true that, with probability one, there is no $\rho\in\support$ for which there are three geodesics with Busemann function asymptotically linear to $\rho$? Equivalently, may there exist three geodesics with the same label with positive probability?
\end{question}

An analogous statement is known to hold for the related model of last-passage percolation in the exactly solvable setting, see~\cite{coupier11}.

Our theory of coalescing geodesics uses the fact that we are working with first-passage percolation on $\Z^2$ in a very strong way. We also use the uniqueness of geodesics, but in a much weaker way.
The next two questions ask whether we can eliminate these conditions.

\begin{question} 
Can one extend Theorem~\ref{hall of mirrors} to distributions with atoms in the following manner:
Define $\geo(u,v)$ to be the union of all geodesics between $u$ and $v$.
For any two diverging sequences $(u_k)_{k\ge1}$ and $(v_k)_{k\ge1}$ in $\Z^2$ is it still true that
$$
\P\big(0 \in \geo(u_k,v_k)\big) \to 0?
$$
\end{question}

\begin{question} 
Can we develop a version of this theory for $d>2$? In particular, do random coalescing geodesics exist for $d>2$?
\end{question}

\bibliographystyle{amsalpha}
\bibliography{ahlberg}

\end{document}